\documentclass{amsart}
\usepackage{amsmath,amsthm,amssymb}
\usepackage{verbatim}
\usepackage{marvosym}
\usepackage{graphicx}
\usepackage{color}
\usepackage{epsfig}
\usepackage{rotating}
\usepackage{lscape}
\usepackage{hyperref}

\usepackage{pinlabel}
\usepackage[all]{xy}

\theoremstyle{plain}
\newtheorem{thm}{Theorem}[section]
\newtheorem{cor}[thm]{Corollary}

\newtheorem{prop}[thm]{Proposition}
\newtheorem{lemma}[thm]{Lemma}

\theoremstyle{definition}
\newtheorem{definition}[thm]{Definition}
\newtheorem{remark}[thm]{Remark}
\newtheorem{question}[thm]{Question}

\makeatletter
\newtheorem*{rep@theorem}{\rep@title}
\newcommand{\newreptheorem}[2]{%
\newenvironment{rep#1}[1]{%
 \def\rep@title{#2 \ref{##1}}%
 \begin{rep@theorem}}%
 {\end{rep@theorem}}}
\makeatother

\newreptheorem{theorem}{Theorem}

 \newcommand{\F}{\ensuremath{{\mathcal{F}}}}

\newcommand{\bdry}{\ensuremath{\partial}}

\DeclareMathOperator{\spin}{spin}
\DeclareMathOperator{\flip}{flip}

\newcommand{\A}{\ensuremath{\mathbb{A}}}

\newcommand{\Q}{\ensuremath{\mathbb{Q}}}
\newcommand{\R}{\ensuremath{\mathbb{R}}}
\newcommand{\Z}{\ensuremath{\mathbb{Z}}}
\newcommand{\C}{\ensuremath{\mathbb{C}}}
\renewcommand{\H}{\ensuremath{\mathbb{H}}}

\newcommand{\SL}{\text{SL}}
\newcommand{\PSL}{\text{PSL}}

\newcommand{\tr}{\text{tr}}
\renewcommand{\Re}{\text{Re}}

\newcommand{\Zeta}{\mathcal{Z}}
\newcommand{\Rb}{ \mathcal{R}}

\newcounter{nootje}
\newcommand{\mat}[4]{\left(\begin{array}{cc}#1 & #2 \\ #3 & #4 \end{array}\right)}

\title[Character varieties of once-punctured torus bundles]{Character varieties of once-punctured torus bundles with tunnel number one}

\author{Kenneth L.\ Baker}
\address{Department of Mathematics\\University of Miami\\\newline Coral Gables, FL 33146 \\ USA}
\email{k.baker@math.miami.edu}

\author{Kathleen L.\ Petersen}
\address{Department of Mathematics\\Florida State University\\\newline Tallahassee, FL 32306\\USA}
\email{petersen@math.fsu.edu}

\subjclass[2000]{Primary 57M27; Secondary 57M50,57N10,20C15}
\begin{document}

\begin{abstract}
We determine the $\PSL_2(\C)$ and $\SL_2(\C)$ character varieties of the once-punctured torus bundles with tunnel number one, i.e.\ the  once-punctured torus bundles that arise from filling one boundary component of the Whitehead link exterior.  In particular, we  determine `natural' models for these algebraic sets, identify them up to birational equivalence with smooth models, and compute the genera of the canonical components.  This enables us to compare dilatations of the monodromies of these bundles with these genera.  We also determine the minimal polynomials for the trace fields of these manifolds.  Additionally we study the action of the symmetries of these manifolds upon their character varieties, identify the characters of their lens space fillings, and compute the twisted Alexander polynomials for their representations to $\SL_2(\C)$.
\end{abstract}

\maketitle


\section{Introduction}

The $\SL_2(\C)$ representation variety of a finitely presented group $\Gamma$ is the set of all representations from $\Gamma$ to $\SL_2(\C)$ and naturally carries the structure of a complex algebraic set.  The set of all characters of these representations form a complex algebraic set as well, the $\SL_2(\C)$ character variety.  Restricting attention to irreducible representations, this character variety effectively records the set of (irreducible) $\SL_2(\C)$ representations modulo conjugation.

For hyperbolic $3$--manifolds, the $\SL_2(\C)$ character varieties of their fundamental groups have  been shown to carry much topological data about the underlying manifold; in particular, see \cite{cs1983}, Chapter 1 of \cite{MR881270}, and the survey \cite{shalen-handbook}.  However, even the simplest invariants of these algebraic sets have proven difficult to compute in general.  On an individual basis, one can often compute defining equations for the character variety of a specific manifold.  In some cases, families of manifolds have been studied. Representations of two-bridge knot groups were studied in \cite{MR745421}, \cite{MR1070927} and \cite{MR1321291}, and a recursively defined formula for character varieties of twist knots was obtained in \cite{MR2124555}.  However, smooth models are required to compute many invariants.    Smooth models for the character varieties of double twist knot exteriors were determined in  \cite{MR2827003}; a smooth minimal model for the character variety the Whitehead link exterior was determined in \cite{MR2831837} and for all four of the arithmetic two-bridge links in \cite{Harada}. 

 Any component of the $\SL_2(\C)$ character variety which contains a character of a discrete faithful representation is called a canonical component.  Work of Thurston indicates that if $M=\H^3/\Gamma$ is a finite volume hyperbolic $3$--manifold, then the complex dimension of a canonical component equals the number of cusps of $M$. Therefore, if $M$ has a single cusp, the genus of this real surface is a natural invariant. As genus is a measure of complexity of a (real) surface, a natural expectation is that the genus of these canonical components is related to many measures of complexity of the underlying manifold.

In this article we determine the $\SL_2(\C)$ character varieties of the fundamental groups of the once-punctured torus bundles with tunnel number one. 
Up to homeomorphism, this is an infinite family of $3$--manifolds, $\{M_n\}_{n\in\Z}$, which are finite volume hyperbolic $3$--manifolds with exactly one cusp when $|n|>2$. Moreover their fundamental groups admit presentations $\pi_1(M_n) = \langle \alpha, \beta : \beta^n =\omega \rangle$ where $\omega$ is a simple word in the generators independent of $n$.  This pleasant structure enables us to explicitly compute `natural' models for these algebraic sets, identify them up to birational equivalence with smooth models, and determine the genera of the canonical components.   To avoid cumbersome statements, we will call a curve {\em hyperelliptic} if is is birational to a plane curve given by an equation of the form $y^2=f(x)$ where $f$ has no repeated roots.  Thus we include in this set the genus zero conics for which $\deg(f) = 1,2$ and elliptic curves for which $\deg(f) = 3,4$ in addition to usual hyperelliptic curves of genus at least two for which $\deg(f) \geq 5$.  

We show the following.
\begin{reptheorem}{thm:mainsummary1}\em
 If $|n|>2$  then there is a  unique canonical component of the $\SL_2(\C)$ character variety of $M_n$, and it is birational to the hyperelliptic curve given by $w^2=-\hat{h}_n(y)\hat{\ell}_n(y)$.  
 The genus of the canonical component is  $\lfloor \tfrac12 |n-1|-1 \rfloor $ if $n\not \equiv 2 \pmod 4$ and is  $\ \lfloor \tfrac12 |n-1|-2 \rfloor$ if $n\equiv 2 \pmod 4$. 
    If $n\not \equiv 2 \pmod 4$ this is the only component of the $\SL_2(\C)$ character variety containing the character of an irreducible representation.  If $n\equiv 2 \pmod 4$ there is an additional  component which is isomorphic to $\C$. 
\end{reptheorem}
Here the polynomials $\hat{h}_n$ and $\hat{\ell}_n$ are specific factors of Fibonacci polynomials.  See Definitions~\ref{definition:fibonacci}, \ref{definition:fibs}, and \ref{definition:hats}.
As an immediate corollary we have:
\begin{cor}
For every $m\in \Z$ there is a once punctured  torus bundle with tunnel number one such that the canonical component of the $\SL_2(\C)$ character variety has genus greater than $m$.
\end{cor}

In fact, we make this more precise by observing the genus of the canonical component grows with the dilatation of the monodromy of the fibration. Here we give a condensed statement of Theorem~\ref{thm:dilatationgenus}.
\begin{thm}
For $|n|>2$ the dilatation of the monodromy of the fibration of $M_n$ is approximately twice the genus of the canonical component of the $\SL_2(\C)$ character variety of $M_n$.
\end{thm}

\begin{question}
What is the relationship between the dilatation of a fibered hyperbolic manifold with one cusp and the genera of the  canonical components of its $\SL_2(\C)$ character variety?
\end{question}

We also show that if $n\neq -2$  all of the components consisting of characters of reducible representations are isomorphic to affine conics (including lines) and consist of characters of abelian representations.  The structure of these components is summarized in    Proposition~\ref{prop:redreps}.
We realize the intersection of the canonical component with these lines as characters of  lens space fillings.

Let $p_n(y)=f_{n+1}(y)-f_{n-1}(y)-y^2+6$. If $n$ is even let $\hat{p}_n(y)=p_n(y)$, and if $n$ is odd, let $\hat{p}_n(y)$ be the polynomial obtained by factoring out the $y+2$ factor from $p_n(y)$.  In Lemma~\ref{lemma:tracepolyirred} we show that $\hat{p}_n$ is irreducible.  This proof is an extension of  Farshid Hajir's proof of the irreducibility of $\hat{p}_n$  for $n$ even.  From the explicit equations we obtain for the natural model (Proposition~\ref{prop:irreps}),   we determine the trace field explicitly.  
\begin{reptheorem}{thm:tracefield}
When $|n|>2$ the polynomial $\hat{p}_n(y)$ is irreducible over $\Q$ and is the minimal polynomial for the trace field of $M_n$. The degree of the trace field is $|n|-e$ where $e=0$ if $n$ is even and $e=1$ if $n$ is odd.
\end{reptheorem}

There is an obstruction to lifting representations from $\PSL_2(\C)$ to $\SL_2(\C)$.  When $n$ is odd, this obstruction vanishes and  the full $\PSL_2(\C)$ character variety is an algebro-geometric quotient of the $\SL_2(\C)$ character variety.  When $n$ is even this quotient is not the complete $\PSL_2(\C)$ character variety. We compute this quotient, and when $n$ is even we determine the remaining $\PSL_2(\C)$ characters that do not lift.  For $n$ odd, we show that there is one component of this set containing  characters of irreducible representations, and it is birationally equivalent to $\A^1$.  We also show that this corresponds to the line as the quotient of the above  hyperelliptic curve by the hyperelliptic involution.  That is, we show the following theorem, where by using the terms $\SL_2(\C)$ and $\PSL_2(\C)$ character variety we mean the Zariski closure of the characters corresponding to irreducible representations.  (The text of Theorem~\ref{thm:pslasinvolution} in the body of the manuscript varies from the text below as it uses terminology that we have not yet defined.)

\begin{reptheorem}{thm:pslasinvolution}\em
When $n$ is odd, the
identification of the $\PSL_2(\C)$ character variety as the quotient of the $\SL_2(\C)$ character variety  by the action of $\mu_2\cong \Z/2\Z$ corresponds to the identification of the hyperelliptic curve given by 
\[ w^2 = -\hat{h}_n(y)\hat{\ell}_n(y)\] by the hyperelliptic   involution $(y,w) \mapsto (y,-w)$.  The $\PSL_2(\C)$ character variety is  birational to an affine line, $\A^1$. 
\end{reptheorem}

When $n$ is even the situation is quite different.

\begin{reptheorem}{thm:psleven}\em
When $n$ is even, the quotient map from the $\SL_2(\C)$ character variety to the $\PSL_2(\C)$ character variety is determined by the hyperelliptic involution $(y,w) \mapsto (y,-w)$, and the involution $(y,w)\mapsto (-y,w)$.  The canonical component is birational to an affine line. When $n\equiv 2 \pmod 4$ there is an additional line of characters in the image of this map. 

 In addition, for any even $n$ there is a parametric curve of characters of representations which do not lift to $\SL_2(\C)$, along with a finite set of points which do not lift.  This curve is birational to an affine line.  If $n\equiv 0 \pmod 4$ there is an additional component of characters of representations which do not lift; this component is birational to an affine line. 

\end{reptheorem}

We compute the twisted Alexander polynomials of these manifolds. 
\begin{reptheorem}{thm:twistedalex}\em
The twisted Alexander polynomial of $M_n$, twisted by a representation corresponding to the point $(x,y,z)$ on the character variety is
\[\mathcal{T}_{M_n}^\rho(T)= T^{-1}+\frac{2(z-x)}{y-2}  + T.\]
\end{reptheorem}

Further, we investigate the symmetries of these manifolds and study the effect of these symmetries on the character varieties. 
A symmetry of the  manifold $M$ induces an action on $\pi_1(M)$ and on the character variety of $M$. 
For two-bridge knots, some of these symmetries act trivially on the character variety, while others effective factor it \cite{MR2827003}.     The action of these symmetries and such factorization is not well understood in general.
In this light, we study the symmetries of the family of once-punctured torus bundles of tunnel number one and their effect on the character variety.
For $|n|>2$ the symmetry group of $M_n$ is generated by two involutions, `spin' and `flip'.  We investigate the  induced action of these isometries on the character variety.  We conclude that each irreducible representation $\rho:\pi_1(M_n) \rightarrow \SL_2(\C)$ factors through a representation from $\pi_1(O_n) \rightarrow \SL_2(\C)$, where $O_n$ is the orbifold obtained as a quotient of $M_n$ by the group of isometries.

The manifolds $M_n$ are precisely the once-punctured torus bundles that are obtained by filling one component of the Whitehead link exterior.  Each of the manifolds $M_n$ has at least one lens space filling, and for $n=-1,0,1,2,3$ and 5 there are additional fillings (two additional for $n=3$).  We investigate the characters associated to these fillings, which correspond to surjections of $\pi_1(M_n)$ onto cyclic groups and their quotients. 

We also study the global structure of the character variety by investigating the points of intersection of the various components.

\subsection{Organization}

In Section~\ref{section:preliminaries} we introduce the family $\{ M_n\}_{n\in \Z}$ of  once-punctured torus bundles with tunnel number one and their fundamental groups, $\{ \Gamma_n\}_{n\in \Z}$, and establish their basic algebraic and topological properties.   Section~\ref{section:charvarpreliminaries} introduces the construction of the $\SL_2(\C)$ and $\PSL_2(\C)$ character varieties.   Section~\ref{section:fibonacci} contains many useful facts about a family of polynomials, the Fibonacci polynomials, that are used extensively throughout the manuscript.  

In Section~\ref{section:charactervarietycalculations}  we  explicitly  compute  the sub-set of the $\SL_2(\C)$ character variety consisting of characters of reducible representations, a `natural model' for the irreducible characters, and a `smooth model'  for this set.   We show that if $n\not \equiv 2 \pmod 4$ the smooth model  is a hyper-elliptic curve. If $n\equiv 2 \pmod 4$ there is an additional $\A^1$ component as well. We identify the canonical component, which is the hyper-elliptic curve  in Section~\ref{section:charactervarietycalculations}.  Section~\ref{section:PSLcalc} contains the determination of the  $\PSL_2(\C)$ character varieties as quotients of the $\SL_2(\C)$ character varieties, for odd $n$; these are all birational to affine lines.

We determine polynomials which define the trace fields in Section~\ref{section:tracefield}.  Each once punctured torus bundle with tunnel number one  has at least one lens space filling, and is a filling of the Whitehead link.  In Section~\ref{section:whitehead} we investigate these fillings.  Each such filling corresponds to a surjection of the fundamental group onto a cyclic group, and  we realize these characters as intersections with the canonical component and the reducible represenations. In Section~\ref{section:intersections} we explicitly determine the global structure of the $\SL_2(\C)$ character variety, by determining the intersections of all the components.  We investigate the action of the symmetries of the underlying manifold on the character varieties in Section~\ref{section:symmetries}.  The twisted Alexander polynomials are computed in Section~\ref{section:alexanderpoly}, and the dilatation of the pseudo-Anosov map corresponding to the manifold $M_n$ is computed in Section~\ref{section:dilatation}.

%
%
\subsection{Acknowledgements}\label{section:acknowledgements}
%
%
The authors would like to thank Eriko Hironaka and Ronald van Luijk for helpful conversations.   
The authors are also indebted to Farshid Hajir for the idea behind the proof of Lemma~\ref{lemma:tracepolyirred}. This work was partially supported by  Simons Foundation grant  \#209184  to  Kenneth Baker  and grant  \#209226 to Kathleen Petersen.

 %
\section{Once-punctured Torus Bundles with Tunnel Number One}\label{section:preliminaries}
%

As mentioned in the introduction, the once-punctured torus bundles with tunnel number one, up to mirroring, form a one parameter family $\{M_n\}_{n \in \Z}$.  The monodromy $\phi_n$ of $M_n = T \times [0,1] / (x,0) \sim (\phi_n(x),1)$ may be presented as $\phi_n = \tau_c \tau_b^{n+2}$ where $c$ and $b$ are curves on the once-punctured torus fiber $T$ transversally intersecting once and $\tau_a$ is a right-handed Dehn twist along the curve $a$.  This monodromy is given as its conjugate  $\tau_c \tau_b \tau_c \tau_b^n$ in \cite{BJK}.  The manifold $M_n$ is hyperbolic if and only if $|n|>2$, contains an essential torus if and only if $|n|=2$, and is a Seifert fiber space if and only if $|n|<1$ as we observe in Lemma~\ref{lem:geometry}.  

Each manifold $M_n$ may be obtained by $-(n+2)$--Dehn filling one boundary component of the Whitehead link exterior and is the exterior of a certain genus one fibered knot in the lens space $L(n+2, 1)$, \cite[Theorems 1.2, 1.3]{BJK}.  In fact, these are exactly the fibered manifolds obtained by filling one boundary component of the Whitehead link exterior.  Also, certain members of this family may be viewed as exteriors of other knots in lens spaces as well, \cite{hmw:sotwlygsm}.  
We discuss these features in Section~\ref{section:whitehead}.

\begin{remark}\label{remark:someoftheMs}
For small $n$ we have some familiar manifolds.
$M_{-1}$ is the positive trefoil knot exterior, $M_{-3}$ is the figure eight knot exterior, and $M_3$ is the figure eight sister manifold.  
\end{remark}

\begin{remark}\label{remark:M-2}
The manifold $M_{-2}$ exhibits some markedly different behavior than the other manifolds in the family.  Throughout this article it will appear as an exceptional case.  One may expect such exceptional behavior as the first Betti number of $M_{-2}$ is two, whereas it is one for all the other $M_n$.  Indeed, as the monodromy of $M_{-2}$ is a single Dehn twist along the essential simple closed curve $c$, this curve sweeps out an essential non-separating torus. 
\end{remark}

\subsection{Algebra}\label{sec:fundgroup}

Let $\gamma$ and $\beta$ be oriented loops on $T$ based at a point $* \in \bdry T$  so that $\gamma, \beta$ are freely homotopic to $c,b$ respectively and, conflating based curves with their homotopy classes, $\pi_1(T,*) = \langle \gamma, \beta \rangle$. We use \: $\bar{}$\:  to denote inverses.   Then, on the level of $\pi_1$, 
\[ (\tau_c)_* \colon \begin{cases} \gamma \mapsto \gamma \\ \beta \mapsto \beta\gamma \end{cases} 
\quad \mbox{ and } \quad
(\tau_b)_* \colon \begin{cases} \gamma \mapsto \gamma \bar{\beta}  \\ \beta \mapsto \beta. \end{cases}\]
So this implies
\[ \phi_* = (\tau_c)_* \circ (\tau_b^{n+2})_* \colon 
\begin{cases} 
\gamma  \mapsto \gamma\bar{\beta} ^{n+2}  \mapsto \gamma ( \bar{\gamma} \bar{\beta})^{n+2}  \\ 
\beta \mapsto  \beta \mapsto \beta \gamma. \end{cases}\]
Therefore  
\begin{align*}
\pi_1(M) &= \langle \gamma, \beta,\mu \colon \phi_*(\gamma) = \mu \gamma \bar{\mu} , \phi_*(\beta) = \mu \beta \bar{\mu} \rangle \\
&= \langle \gamma, \beta,\mu \colon
 \gamma ( \bar{\gamma} \bar{\beta})^{n+2}  = \mu \gamma \bar{\mu},  
  \beta \gamma =\mu \beta \bar{\mu} \rangle\\
&= \langle \gamma, \beta,\mu \colon
  \gamma(\mu \bar{\beta} \bar{\mu})^{n+2}  = \mu \gamma \bar{\mu},  
  \beta \gamma =\mu \beta \bar{\mu} \rangle\\
&= \langle \gamma, \beta,\mu \colon
   \bar{\beta}^{n+2} =  \bar{\mu} \bar{\gamma} \mu\gamma,  
   \gamma = \bar{\beta} \mu \beta \bar{\mu} \rangle\\
&= \langle \beta,\mu \colon
   \bar{\beta}^{n+2} =   \bar{\mu} (\mu \bar{\beta} \bar{\mu} \beta) \mu (\bar{\beta} \mu \beta \bar{\mu})  \rangle\\
&= \langle \beta,\mu \colon
   \bar{\beta}^{n} =   \bar{\mu} \beta \mu \bar{\beta} \mu \beta \bar{\mu}\beta  \rangle\\
&= \langle \beta,\mu \colon
   \bar{\beta}^{n} =   (\bar{\mu} \beta) \beta (\bar{\beta} \mu) (\bar{\beta} \mu) \beta (\bar{\mu}\beta)  \rangle\\   
&= \langle  \alpha, \beta \colon
\bar{\beta}^n = \bar{\alpha} \beta \alpha \alpha \beta \bar{\alpha} \rangle \mbox{ where } \mu = \beta \alpha.
\end{align*}

\begin{definition}
Notice, $\mu$ represents primitive element of $\pi_1 (\bdry M_n)$ and may be realized as an embedded curve in $\bdry M_n$ transversally intersecting each fiber once. We call this curve {\em the meridian} of $M_n$. Killing $\mu$ gives the cyclic group $\langle \beta \colon \bar{\beta}^{n+2} = 1 \rangle$ which corresponds to a lens space Dehn filling $M_n$ along the slope of $\mu$, see Section~\ref{section:whitehead}.  In particular, the meridian of $M_n$ is the meridian of a knot in this lens space whose exterior is $M_n$.
\end{definition}

\begin{definition}
Define $\Gamma_n = \langle \alpha, \beta : \bar{\beta}^n=\omega \rangle$ where $\omega= \bar{\alpha}\beta \alpha^2\beta\bar{\alpha}$.
\end{definition}

\begin{lemma}\label{lemma:boundarypi1}
The fundamental group of a once-punctured torus bundle $M_n$ with tunnel number one is
$\pi_1(M_n) \cong \Gamma_n$, the boundary of a fiber corresponds to $(\bdry T)_* = \lambda = (\alpha \beta\bar{\alpha})\beta(\alpha \bar{\beta}\bar{\alpha})\bar{\beta}$, and the meridian corresponds to $\mu = \beta \alpha$.  
\end{lemma}
\begin{proof}
That $\pi_1(M_n) \cong \Gamma_n$ and $\mu = \beta \alpha$ follows from  above.  Giving $\bdry T$ the boundary orientation, $\lambda = (\bdry T)_* = \gamma \beta \bar{\gamma} \bar{\beta}$.  Now use that $\gamma = \bar{\beta} \mu \beta \bar{\mu} = \alpha \beta \bar{\alpha} \bar{\beta}$.
\end{proof}

\begin{lemma}\label{lem:H2Gamma}
$H^2(\Gamma_n;\Z/2\Z) = 0$ if $n$ is odd and $\Z/2\Z$ if $n$ is even.
\end{lemma}
\begin{proof}
Since our once-puctured torus bundle $M_n$ is aspherical and $\pi_1(M_n) = \Gamma_n$, it is a $K(\Gamma_n,1)$.   Thus $H^2(\Gamma_n; \Z/2\Z) \cong H^2(M_n; \Z/2\Z)$.  Since $H_1(M_n;\Z) = \Z/(n+2)\Z$ and $H_2(M_n;\Z)=0$, the Universal Coefficient Theorem for cohomology implies that $H^2(M_n;\Z/2\Z) = {\rm Ext}(\Z/(n+2)\Z, \Z/2\Z) = \Z/{\rm gcd}(n+2,2)\Z$.
\end{proof}

\begin{lemma}\label{lem:HomGamma}
${\rm Hom}(\Gamma_n; \Z/2\Z) = \Z/2\Z$ if $n$ is odd and $ \Z/2\Z \times \Z/2\Z$ if $n$ is even.
\end{lemma}

\begin{proof}
An element $\varphi \in \text{Hom}(\Gamma_n; \Z/2\Z)$ is given by a mapping of the generators $\alpha$ and $\beta$ into $\Z/2\Z$ such that the relation $\varphi(\bar{\beta}^n) = \varphi(\bar{\alpha} \beta \alpha^2 \beta \bar{\alpha})$ holds.  Since this relation reduces to $\varphi(\beta)^n = 1$, $\varphi(\alpha)$ has no constraints while we must have $\varphi(\beta) = 1$ if $n$ is odd.  
\end{proof}

\subsection{Geometry}

\begin{lemma}\label{lem:geometry}
$M_n$ is hyperbolic if and only if $|n|>2$, contains an essential torus if and only if $|n|=2$, and is a Seifert fiber space if and only if $|n|<2$.
\end{lemma}
\begin{proof}
On the level of homology, using that $H_1(T;\Z) = \Z^2$ is generated by 
$[\gamma] = (1,0)^{T}$ and $[\beta] = (0,1)^{T}$, we have 
\[ [\tau_c] = \mat{1}{1}{0}{1} \quad \mbox{ and } \quad [\tau_b] = \mat{1}{0}{-1}{1} 
\quad \mbox{so that} \quad [\phi_n] = \mat{-(n+1)}{1}{-(n+2)}{1}.\]
Hence $|\tr [\phi_n]| = |n|$.
Therefore if $|n|>2$ then $\phi_n$ is pseudo-Anosov and $M$ is hyperbolic, see e.g.\  \cite{cassonbleiler}.  

For $|n| = 2$, we may observe that $\phi_{-2}$ fixes the curve $c$ while $\phi_{2}$ reverses it. Thus $M_{-2}$ contains an essential non-separating torus and $M_2$ contains an essential separating torus bounding a neighborhood of a Klein bottle to one side.  For $|n|<2$, we may observe that $\phi_n$ is periodic thus giving $M_n$ the structure of a Seifert fiber space over the disk with two exceptional fibers.
\end{proof}

Each element  $\gamma \in \Gamma \cong \pi_1(M)$ is either {\em parabolic} if $\tr(\overline{\rho}_0(\gamma)) = \pm2$ or {\em hyperbolic} (including loxodromic) otherwise, where $\overline{\rho}_0$ is a discrete faithful representation of $\gamma$ to $\PSL_2(\C)$.  Note that since $\Gamma$ is torsion-free there are no elliptic elements, those elements with real trace of magnitude less than $2$.

\begin{lemma}\label{lemma:nonperipheralgenerators}
If $M_n$ is a hyperbolic manifold then the elements $\alpha$ and $\beta$ of $\pi_1(M_n)$ are hyperbolic.
\end{lemma}
\begin{proof}
Since $\mu$ is a peripheral element and $\mu=\beta\alpha$, we may take either $\{\mu, \alpha\}$ or $\{\mu, \beta\}$ instead of $\{\alpha, \beta\}$ as a generating set for $\pi_1(M_n)$.
If either $\alpha$ or $\beta$ is a peripheral element, then \cite[Corollary 5]{BoileauWeidmann} implies that $M_n$ is homeomorphic to the exterior of a two-bridge knot in $S^3$ (because $M_n$ is a compact, orientable, irreducible, and $\bdry$-irreducible $3$-manifold).  In particular, since $M_n$ is a once-punctured torus bundle, it is necessarily homeomorphic to the exterior of the figure eight knot or a trefoil,  i.e.\ $n=-3$ or $n=-1$ respectively.  Assuming $M_n$ is hyperbolic, we must have $n=-3$.

With $n=-3$, we may write 
\begin{align*}
 \pi_1(M_{-3}) &= \langle \beta,\mu \colon   \beta^3 =  \bar{\mu} \beta \mu \bar{\beta} \mu \beta \bar{\mu}\beta  \rangle\\
 &= \langle \beta,\mu \colon   1 =  \bar{\beta} \bar{\mu} \beta \mu \bar{\beta} \mu \beta \bar{\mu}\bar{\beta}  \rangle\\
 &=  \langle \mu, \nu \colon  1 = \bar{\nu} \bar{\mu} \nu \mu \bar{\nu} \mu \nu \bar{\mu} \bar{\nu} \mu \rangle \mbox{ where } \nu = \mu \beta\\
 &= \langle \mu, \nu \colon  \bar{\mu} \nu \mu \bar{\nu} \mu = \nu \bar{\mu} \nu \mu \bar{\nu} \rangle
 \end{align*}
 which is a familiar presentation generated by peripheral elements derived from viewing the figure eight knot as a two-bridge link.  Following \cite{riley},  this group has a discrete faithful representation given by
 \[  \mu \mapsto \mat{1}{1}{0}{1} \quad \mbox{ and } \quad \nu \mapsto \mat{1}{0}{e^{\pi i/3}}{1}. \]
 Thus we obtain
 \[ \beta \mapsto \mat{1-e^{\pi i/3}}{-1}{e^{\pi i/3}}{1} \quad \mbox{ and } \quad \alpha \mapsto \mat{1}{2}{-e^{\pi i/3}}{1-2e^{\pi i/3}} \]
 which have complex traces $2-e^{\pi i/3}$ and $2-2e^{\pi i/3}$ respectively.  Hence both $\alpha$ and $\beta$ are hyperbolic elements.
\end{proof}

%
%
 \section{Character Varieties Preliminaries}\label{section:charvarpreliminaries}
 %
 %
 
 Let $\Gamma$ be a finitely generated group with generating set $\{ \gamma_1, \dots, \gamma_N\}$.   The {\em $\SL_2(\C)$ representation variety}, $R(\Gamma)=\text{Hom}(\Gamma, \SL_2(\C))$ is an affine algebraic set defined over $\Q$ as seen by using the four entries of the images of the $\gamma_n$ under $\rho\in R(\Gamma)$ as coordinates for $\rho$. The isomorphism class of this algebraic set is independent of the choice of generators.  In general, $R(\Gamma)$ is not irreducible as an affine algebraic set.  
 
  A representation $\rho\in R(\Gamma)$ is {\em reducible} if all $\rho(\gamma)$ with $\gamma \in \Gamma$ share a common one-dimensional eigenspace. Otherwise, $\rho$ is called irreducible.  A representation $\rho\in R(\Gamma)$ is {\em abelian} if its image is an abelian subgroup of $\text{SL}_2(\C)$.  
 
\subsection{$\SL_2(\C)$ Character Varieties}\label{sec:charvar}

The {\em character} of  a representation $\rho\colon \Gamma \rightarrow \text{SL}_2(\C)$ is the function $\chi_{\rho}\colon \Gamma \rightarrow \C$ defined by $\chi_{\rho}(\gamma)=\tr(\rho(\gamma))$.  The set of characters of representations of $\Gamma$ into $\SL_2(\C)$, is the set 
\[ \tilde{X}(\Gamma)=\{\chi_{\rho}:\rho\in R(\Gamma)\}. \]  This is often denoted $X(\Gamma)$, but we will reserve this notation for a specific subset of $\tilde{X}(\Gamma)$ (namely the Zariski closure of the subset of characters of irreducible representations as described below).   For all $\gamma \in \Gamma$ define the function $t_{\gamma}\colon R(\Gamma) \rightarrow \C$ by $t_{\gamma}(\rho) = \chi_{\rho}(\gamma)$. 

Let $T$ be the subring of all functions from $R(\Gamma)$ to $\C$ generated by $1$ and the functions $t_{\gamma}$ for all $\gamma \in \Gamma$.  The ring $T$ is generated by the elements 
\[ t_{\gamma_{i_1}, \dots, \gamma_{i_r}}, 1\leq i_1 <\dots < i_r \leq N. \]
This implies that a character $\chi_{\rho}$ is determined by its values on finitely many elements of $\Gamma$. 
In the case when $\Gamma$ has a presentation with only two generators, $\gamma_1$ and $\gamma_2$, $T$ is generated by $t_{\gamma_1}$, $t_{\gamma_2}$, and $t_{\gamma_1\gamma_2}$.  
If $h_1(\rho), \dots h_m(\rho)$ are generators of $T$, then the map $R(\Gamma)\rightarrow \C^m$ given by $\rho \rightarrow (h_1(\rho), \dots, h_m(\rho))$ induces an injection $\tilde{X}(\Gamma)\rightarrow \C^m$. This gives $\tilde{X}(\Gamma)$ the structure of a closed algebraic subset of $\C^m$ (see \cite{cs1983}). It follows that $\tilde{X}(\Gamma)$ has coordinate ring $T_{\C}=T\otimes \C$.  The isomorphism class (over $\Z$) of this algebraic set is independent of the choice of generators for $\Gamma$.   
 The set $\tilde{X}(\Gamma)$ is called the {\em $\text{SL}_2(\C)$ character variety} of $\Gamma$.

 The group $\text{SL}_2(\C)$ acts on $R(\Gamma)$ by conjugation.  Let $\hat{R}(\Gamma)$ denote the set of orbits.  Two representations $\rho, \rho' \in R(\Gamma)$ are conjugate if they lie in the same orbit. Two conjugate representations give the same character, so that the trace map $R(\Gamma) \rightarrow \tilde{X}(\Gamma)$ induces a well-defined map $\hat{R}(\Gamma) \rightarrow \tilde{X}(\Gamma)$.  This map need not be injective, but it is injective when restricted to irreducible representations.  Specifically,  if $\rho, \rho' \in R(\Gamma)$ have equal characters $\chi_{\rho}=\chi_{\rho'}$ and $\rho$ is irreducible, then $\rho$ and $\rho'$ are conjugate, and therefore both $\rho$ and $\rho'$ are irreducible (see \cite{cs1983} Proposition 1.5.2). We will say that a character $\chi_{\rho}$ is  {\em reducible}, {\em irreducible} or  {\em abelian}  if $\rho$ is. 
  
Let $\tilde{X}_{a}(\Gamma)$, $\tilde{X}_{red}(\Gamma)$ and $\tilde{X}_{irr}(\Gamma)$ denote the set of characters of abelian, reducible, and irreducible representations $\rho\in \R(\Gamma)$, respectively. The set $\tilde{X}_{red}(\Gamma)$ is a Zariski closed algebraic set (see \cite{MR1321291} Proposition 1.3(ii)) and as a set $\tilde{X}_{irr}(\Gamma)=\tilde{X}(\Gamma)-\tilde{X}_{red}(\Gamma)$.

We will explicitly compute both $\tilde{X}_{red}(\Gamma_n)$ and $\tilde{X}_{irr}(\Gamma_n)$, and will show (Proposition  \ref{prop:redreps}) that  $\tilde{X}_a(\Gamma_n)=\tilde{X}_{red}(\Gamma_n)$ and that this set is a collection of affine conics and lines (unless $n=-2$).  Therefore we will focus primarily on the Zariski closure of $\tilde{X}_{irr}(\Gamma)$, which we will denote $X(\Gamma_n)$.  We will often refer to this set as the {\em$\text{SL}_2(\C)$ character variety of $\Gamma_n$}.

 There is a non-singular projective model in every birational equivalence class of algebraic curves. This model is unique up to isomorphism. \cite[\S4.5 p.121]{Shafarevich}.   Therefore we will study the birational equivalence class of $X(\Gamma_n)$ and determine such a smooth model after projectivization.  We use the notation $\A$ for the affine line, $\C$.

\subsubsection{Hyperbolic $3$-manifolds and the canonical component.}
  
If $M$ is a finite volume oriented hyperbolic $3$-manifold, then $M$ is isomorphic to a quotient of $\mathbb{H}^3$ by a torsion-free discrete group, $\Gamma$ and $\pi_1(M) \cong \Gamma$.  By Mostow-Prasad rigidity there is a discrete faithful representation  $\overline{\rho}_0:\Gamma \rightarrow \text{Isom}^{+}(\mathbb{H}^3) \cong \text{PSL}_2(\C)$  that is unique up to conjugation.  (This conjugation is why we use the character variety instead of the representation variety.)  This defines an action of $\Gamma$ on $\mathbb{H}^3$ whose quotient $\mathbb{H}^3/\Gamma$ is isometric  to $M$.  The representation $\overline{\rho}_0$ can be lifted 
 to a discrete faithful representation $\rho_0:\Gamma\rightarrow \text{SL}_2(\C)$.  (There may be more than one such lift.) The character $\chi_{\rho_0}$ of such a representation is contained in a single component of $X(\Gamma)$ (rather than in the intersection of multiple components).  The complex dimension of this component equals the number of cusps of $\Gamma$ (see \cite{thurston}).  Such a component is called a {\em canonical component}, and any canonical component is denoted $X_0(\Gamma)$.

 \subsection{$\PSL_2(\C)$ Character Varieties}\label{section:PSL}

 There are various constructions for the $\PSL_2(\C)$ representation and character varieties.   We follow the treatment in \cite{MR2199350}. 
 We let $\tilde{Y}(\Gamma_n)$ denote the full $\PSL_2(\C)$ character variety, and let $\tilde{Y}_{red}(\Gamma_n)$ denote the characters of reducible representations to $\PSL_2(\C)$.  Let $Y(\Gamma_n)$ be the Zariski closure of the set of irreducible characters, the closure of ${\tilde{Y}(\Gamma_n)-\tilde{Y}_{red}(\Gamma_n)}$.  
If $M_n$ is hyperbolic, then we denote a component of $Y (\Gamma_n)$ that contains the character of the
discrete faithful representation of $\Gamma_n$ by $Y_0(\Gamma_n)$.  
 
  There is a natural action of $H^1(\Gamma_n; \Z/2\Z)$  on $\tilde{X}(\Gamma_n)$.    The quotient of this action is contained in $\tilde{Y}(\Gamma_n)$.   By Lemma~\ref{lem:HomGamma} this covering  is of order two if $n$ is odd and order four if $n$ is even.  If $\rho$ is a representation of $\Gamma_n$ to $\PSL_2(\C)$, the  second Stiefel-Whitney class  $\omega_2(\rho) \in H^2(\Gamma_n; \Z/2\Z)$ is precisely the obstruction for $\rho$ to lift to a representation into $\SL_2(\C)$.  By Lemma~\ref{lem:H2Gamma}, $H^2(\Gamma_n; \Z/2\Z)$ is trivial if $n$ is odd.  Therefore all representations of $\Gamma_n$ to $\PSL_2(\C)$ lift to $\SL_2(\C)$  and the $\SL_2(\C)$ character variety is a two-fold cover of the $\PSL_2(\C)$ character variety.  When $n$ is even, $H^2(\Gamma_n; \Z/2\Z)$  has order two.  Therefore some representations to $\PSL_2(\C)$  do not lift to representations of $\SL_2(\C)$.  The $\SL_2(\C)$ character variety is a four-fold cover of the subvariety of $\tilde{Y}(\Gamma_n)$ consisting of those representations that lift. 
   
    Those representations to $\PSL_2(\C)$ which do not lift to $\SL_2(\C)$ can be determined by considering a representation $\rho$ to $\SL_2(\C)$ such that $\rho$ sends the defining relation to negative the identity matrix, instead of to the identity matrix.

 We will now describe the action of  $H^1(\Gamma_n; \Z/2\Z)$  on $\tilde{X}(\Gamma_n)$ in more detail.
 Let $\mu_2\cong\{\pm1\}$ be the kernel of the homomorphism from $\SL_2(\C)$ to $\PSL_2(\C)$.  An element $\sigma \in \text{Hom}(\Gamma_n,\mu_2) $ acts on $\chi \in \tilde{X}(\Gamma_n)$ by $(\sigma \chi)(\gamma)=\sigma(\gamma)\chi(\gamma)$ for all $\gamma\in \Gamma_n$.  
   By Lemma~\ref{lem:HomGamma} if $\sigma\in \text{Hom}(\Gamma_n; \mu_2)$ then $\sigma(\beta)$ is trivial if $n$ is odd, but this is not the case for $n$ even.
    
 First consider the case when $n$ is odd. 
  Define $\Gamma_n^e \subset \Gamma_n$ as the set consisting of all words when written in terms of $\alpha$ and $\beta$ have even exponent sum in $\alpha$.  This is a subgroup of $\Gamma_n$ of index two, as the cosets are $\Gamma_n^e$ and $\alpha \Gamma_n^e$.  
 The action of $\mu_2$ on $R(\Gamma_n)$ is given by $(-\rho)(\gamma)=-\rho(\gamma)$ for $\gamma \not \in \Gamma_n^e$ and $(-\rho)(\gamma)=\rho(\gamma)$ for $\gamma \in \Gamma_n^e$.  This induces an action of $\tilde{X}(\Gamma_n)$ given by  $-\chi_{\rho}=\chi_{-\rho}$.  The corresponding action on $T$ is negation of all $t_{\gamma}$ for $\gamma \not \in \Gamma_n^e$.  This can be fully expressed by its action on the triple $(\chi_{\rho}(\alpha), \chi_{\rho}(\beta), \chi_{\rho}(\alpha \beta)) \mapsto (-\chi_{\rho}(\alpha), \chi_{\rho}(\beta), -\chi_{\rho}(\alpha \beta)).$
  Therefore the $\PSL_2(\C)$ character variety $\tilde{Y}(\Gamma_n)$ is isomorphic to $\tilde{X}(\Gamma_n)/\mu_2$ and its coordinate ring is $T_e\otimes \C$ where $T_e=T^{\mu_2}$ is the subring of $T$  consisting of all elements invariant under the action of $\mu_2$.

 If $n$ is even, there is an action as above determined by $(\chi_{\rho}(\alpha), \chi_{\rho}(\beta), \chi_{\rho}(\alpha \beta)) \mapsto (-\chi_{\rho}(\alpha), \chi_{\rho}(\beta), -\chi_{\rho}(\alpha \beta)).$  In addition, as $\sigma(\beta)$ need not be trivial, we can form an analogous action with respect to words with even exponent sum in $\beta$.  This corresponds to an action defined by $(\chi_{\rho}(\alpha), \chi_{\rho}(\beta), \chi_{\rho}(\alpha \beta)) \mapsto (\chi_{\rho}(\alpha), -\chi_{\rho}(\beta), -\chi_{\rho}(\alpha \beta)).$  Together, they generate a $\Z/2\Z \times \Z/2\Z$ action on $\tilde{X}(\Gamma_n)$. 
The $\PSL_2(\C)$ character variety $\tilde{Y}(\Gamma_n)$ is isomorphic to $\tilde{X}(\Gamma_n)/\mu_2$ and its coordinate ring is the subring of $T$ consisting of all elements invariant under the action of $\mu_2$.

%
%
\section{The Fibonacci Polynomials}\label{section:fibonacci}
%
%

In this section we collect various facts about a recursively defined family of polynomials which will be used extensively throughout the manuscript.

\begin{definition}\label{definition:fibonacci} 
For any integer $n$ define the $n^{th}$ Fibonacci polynomial $f_n(u)$ by $f_0(u)=0$, $f_1(u)=1$ and for all other integers 
\[ f_{n-1}(u)+f_{n+1}(u) = uf_n(u). \]
If $u=s+s^{-1}$ and $u \neq \pm2$ then $f_n(u)$ can be explicitly written as $\displaystyle f_n(u)=\frac{s^n-s^{-n}}{s-s^{-1}}$.
\end{definition}

\begin{remark}
Our Fibonacci polynomials are a reparametrization of the Chebyshev polynomials of the second kind
and are not the standard Fibonacci polynomials.  The standard Fibonacci polynomials use the recurrence relation $-f_{n-1}(u)+f_{n+1}(u) = uf_n(u)$.
\end{remark}

Simple induction arguments using the recursion relation give the following information about the polynomials $f_n$.
\begin{lemma}\label{lemma:fkvals}
$ f_{2m}(0) = 0, \ \ f_{2m+1}(0) = (-1)^m, \ \ f_n(2) = n, \ \ f_n(-2) = (-1)^{n+1} n$. \qed
\end{lemma}

\begin{lemma}\label{lemma:fibonaccievenodd}\

\begin{enumerate}
\item  If $n\neq 0$ is even then $f_n(u)$ is odd, and if $n$ is odd then $f_n(u)$ is even. 
\item If $n\neq 0$ the degree of $f_n(u)$ is $|n|-1$. \qed
\end{enumerate}

\end{lemma}

\begin{lemma}\label{lemma:ydividesf} 
The polynomial $f_n(u)$ is divisible by $u$ if and only if $n$ is even.  If $n \neq 0$, $u^2$ does not divide $f_n(u)$. 
\end{lemma} 

\begin{proof}  By the recursion (Lemma~\ref{lemma:fkvals}) we see that $f_{2m}(0)=0$ and $f_{2m+1}(0)=(-1)^m$ from which the first assertion follows.  This also implies that the constant term of $f_{2m+1}(u)$ is $(-1)^m.$  Together with the recursion, this implies that  the lowest order term of $f_{2m+2}(u)$ is $\pm u$ minus the lowest order term of $f_{2m}(u)$.  As $f_{2m}(u)$ is an even function, we conclude that the lowest order term of $f_{2m+2}(u)$ is $\pm u$ and therefore it cannot be divisible by $u^2$.
\end{proof}

\begin{definition}\label{definition:rootsofunity}
Let $\Zeta_n$ be the set of all $|n|^{th}$ roots of unity, and let $\Zeta_{n}^{fib}$ be the set of the $n^{th}$ roots of unity other than $1$, and $-1$ (if $n$ is even). 
That is, with $\zeta_{n}=e^{2 \pi i/n}$, 
\[\Zeta_{n} = \{ \zeta_n^k :  k =0, \dots, n-1\}, \ \ \ \Zeta_n^{fib} = \Zeta_n -\{\pm1\}.\]
We also define $\Rb_n$ and $\Rb_n^{fib}$ to be the set of all numbers of the form $\zeta_n^k+\zeta_n^{-k}$ for $\zeta_n$ in $\Zeta_n$ or $\Zeta_n^{fib}$ respectively.  That is,
\[ \Rb_n = \{ 2 \Re(\zeta_n) : \zeta_n\in \Zeta_n \}, \ \ \  \Rb_n^{fib} = \{ 2 \Re(\zeta_n) : \zeta_n\in \Zeta_n^{fib} \} \]
Note that $\Rb_n^{fib}=\Rb_n-\{\pm 2\}$ though $-2 \not \in \Rb_n$ when $n$ is odd.
\end{definition}

The reason for this distinction is that both sets will appear often in our computations; the sets $\Zeta_{2n}^{fib}$ and $\Rb_{2n}^{fib}$ are intimately related to the Fibonacci polynomials.

\begin{lemma}\label{lemma:fibroots}
For $n\neq 0$,  $f_n(u)=0$ if and only if $u \in \Rb_{2n}^{fib}$.
\end{lemma}
\begin{proof}
Any $u\in \C$ may be expressed as $u = s + s^{-1}$ for some $s$.
If $u \neq \pm2$ then 
\[f_n(u)=0 \Leftrightarrow \frac{s^n-s^{-n}}{s-s^{-1}}=0 \Leftrightarrow s^{2n}=1, s \neq \pm1.\]
Thus $s \in \Zeta_{2n}^{fib}$ and so $u \in \Rb_{2n}^{fib}$.  Since $|\Rb_{2n}^{fib}|=|n|-1$ is the degree of $f_n(u)$ by Lemma~\ref{lemma:fibonaccievenodd}, $\Rb_{2n}^{fib}$ is the set of roots of $f_n(u)$.
\end{proof}

We define auxiliary polynomials.  These will be used extensively, as the factorization of certain Fibonacci polynomials will be key to the reducibility of certain algebraic sets. 
\begin{definition}\label{definition:fibs}
Define
\begin{align*}
h_n(u)  & = { \left\{  \begin{aligned}  f_{m-1}(u) & \text{ if } n=2m \\ f_{m}(u)+f_{m-1}(u) & \text{ if } n=2m+1 \end{aligned}  \right. }\\
 j_n(u)  &  = { \left\{  \begin{aligned} f_m(u) & \text{ if } n=2m \\ f_{m+1}(u)+f_{m}(u) & \text{ if } n=2m+1 \end{aligned}  \right.}\\
k_n(u)  & = {\left\{  \begin{aligned} f_{m+2}(u)-f_m(u) & \text{ if } n=2m \\ f_{m+2}(u)-f_{m+1}(u) & \text{ if } n=2m+1 \end{aligned}  \right. } \\
\ell_n(u)  & = { \left\{  \begin{aligned} f_{m+1}(u)-f_{m-1}(u) & \text{ if } n=2m \\ f_{m+1}(u)-f_{m}(u) & \text{ if } n=2m+1. \end{aligned}  \right.} 
\end{align*}

\end{definition}
In Section~\ref{section:newmodel} we will need two further auxiliary polynomials $\hat{h}_n$ and $\hat{\ell}_n$ that remove the factor of $u$ from the polynomials $\ell_n(u)$ and $h_n(u)$ when such a factor exists.  These will be presented in Definition~\ref{definition:hats}.

The next few lemmas can be seen by direct calculation.
\begin{lemma}\label{lemma:fibonacciidentities} 
For all $n$, 
\begin{align*}
f_n(u) & = j_n(u) \ell_n(u)\\
f_{n+1}(u)-1 & = j_n(u) k_n(u)\\
f_{n-1}(u)-1 & = h_n(u) \ell_n(u)\\
f_n(u)-u & = h_n(u)k_n(u),
\end{align*}

whence  $(f_{n+1}(u)-1)(f_{n-1}(u)-1)  =f_n(u)(f_n(u)-u)$.\qed
\end{lemma}

\begin{lemma}\label{lem:constantpolys}
The polynomials $f_n$, $h_n$, $j_n$, $k_n$, $\ell_n$ are all non-constant except for the following:
\begin{align*}
f_{-1} &=-1   & h_0=h_1&=-1   & j_{-2}=j_{-1}&=-1  &k_{-3} &=1  & \ell_{-1} &=1\\
f_0 &=0        & h_2 &=0            &  j_0 &=0              &k_{-2} &=2  & \ell_{0} &=2\\
f_1 &=1        & h_3=h_4 &=1     & j_1 = j_2 &=1     &k_{-1} &=1  & \ell_{1} &=1
\end{align*}
In particular, the degree of $j_n$ is zero only if $|n|<2$. \qed
\end{lemma}

\begin{lemma}\label{lem:kiszero}
If $k_n(u) = 0$ then $|n+2| >1$ and $u \in \Rb_{n+2}^{fib}$. \qed
\end{lemma}

\begin{lemma}\label{lem:hellzeros}
The zeros of $\ell_n(y)$ form the set $\Rb^{fib}_{2n}-\Rb^{fib}_n$. 
The zeros of $h_n(y)$ form the set $\Rb^{fib}_{n-2}$.

The zeros of $\hat{\ell}_n(y)$ are the zeros of $\ell_n(y)$ and the zeros of $\hat{h}_n(y)$ are the zeros of $h_n(y)$  except we must discard $0$ when $n \equiv 2 \pmod 4$ in each case. 
\qed
\end{lemma}

\begin{lemma}\label{lemma:nocommonfactors} 
Let $(a,b)$ denote the ideal generated by polynomials $a(u),b(u)\in\A[u]$.  Then we have the following:
\begin{enumerate} 
\item \label{lemma:jnhn}   For all $n$, $(h_n,j_n) = \A[u]$.
\item \label{lemma:knln}  For all $n$, $(k_n,\ell_n)= \A[u]$. 
\item \label{lemma:hnkn} If $n\neq 2$, then  $(h_n,k_n)=\A[u]$.  Otherwise, $(h_2, k_2) = (u^2-2)$.
\item \label{lemma:jnkn}  If $n\equiv 0 \pmod 4$, then  $(j_n,k_n)=(u)$. Otherwise, $(k_n, j_n)= \A[u]$. 
\item \label{lemma:hnln} If $n\equiv 2 \pmod 4$, then $(h_n, \ell_n)= (u)$.  Otherwise, $(h_n,\ell_n)= \A[u]$.
\end{enumerate}
\end{lemma}

\begin{proof}
Since $(a,b)$ is governed by the common roots of $a(u)$ and $b(u)$, we first determine the roots of $h_n, j_n, k_n, \ell_n$.  
By Lemma~\ref{lemma:fibonacciidentities}, $h_n(u)=0$ implies that $f_{n-1}(u)=1$ and $f_n(u)=u$.  Using the recursion, $f_{n-2}(u)=0$ and therefore $u\in{\Rb}_{2(n-2)}^{fib}$ by Lemma~\ref{lemma:fibroots}.
That is:
\[
\begin{array}{lll}
  h_n(u)=0 & \implies f_{n-1}(u)=1, f_n(u)=u \implies f_{n-2}(u)=0 & \implies u\in{\Rb}_{2(n-2)}^{fib}
  \end{array}
  \]
In a similar manner we have:
\[
\begin{array}{lll}
j_n(u) = 0 &\implies f_n(u)=0,  f_{n+1}(u)=1 & \implies u\in {\Rb}_{2n}^{fib} \\
k_n(u) = 0 & \implies f_{n+1}(u)=1,   f_n(u)=u \implies f_{n+2}(u)=0 & \implies u\in {\Rb}_{2(n+2)}^{fib}\\
\ell_n(u)=0 & \implies f_n(u)=0,  f_{n-1}(u)=1 & \implies u\in {\Rb}_{2n}^{fib}
\end{array}
\]
We will prove cases (\ref{lemma:knln}) and (\ref{lemma:hnln}) as the proofs of the other cases are similar.

Case (\ref{lemma:knln}):  By the above, we conclude that if both $k_n(u)=0$ and $\ell_n(u)=0$ then $f_n(u)=0=u$ and $f_{n-1}(u)=1$. Yet by the recursion for the $k_n(u)$ data, $f_{n-1}(u)=uf_n(u)-f_{n+1}(u)= -1$.  Therefore there can be no common roots. 

Case (\ref{lemma:hnln}):  As $h_n(u)=0$ implies $f_n(u)=u$ while $\ell_n(u)=0$ implies $f_n(u)=0$,  $u=0$ is the only possible common factor of $h_n(u)$ and $\ell_n(u)$.  Both $h_n(u)=0$ and $\ell_n(u)=0$ also then imply $f_{n-1}(0)=1$.  As $f_{2m}(0)=0$ and $f_{2m+1}(0)=(-1)^m$ for all integers $m$, we obtain a contradiction (and hence no common factor) unless $n=4i+2$ for some integer $i$.
 In this case $\ell_n(u)=f_{i+1}(u)-f_{i-1}(u)$ and $h_n(u)= f_{i-1}(u)$.  By Lemma~\ref{lemma:ydividesf}  it follows that the multiplicity of the common factor is one. 
\end{proof}

Recall that a polynomial is {\em separable} if it has distinct roots.  

\begin{lemma}\label{lemma:fisseparable} For all integers $n$, the polynomials $f_n(u)$, $f_{n+2}(u)-f_n(u)$, $f_{n+1}(u)-f_n(u)$ and $f_{n+1}(u)+f_n(u)$ are all separable, with the exception of $f_0(u)$, which is identically zero.  \end{lemma} 

\begin{proof} 
Let $u=s+s^{-1}$.  For an integer $r\neq 0$, $R=s^r-1$ is separable as $s \tfrac{dR}{ds}-rR =r$ is a non-zero constant for $r\neq 0$.   Since $R$ and $\tfrac{dR}{ds}$ are linearly independent, we conclude that they cannot share a common factor and $R$ is separable.

Set $p=(s^{n+1}-s^{n-1})f_n(u) \in \mathbb{Z}[u][s]/(s^2-us+1)\cong \mathbb{Z}[s,s^{-1}]$.  Then $p=s^{2n}-1$ is separable and therefore $f_n(u)$ has no multiple factors either. 

Set $p=(s^{k+l+1}-s^{k+l-1})(f_{k+l}(u)-f_k(u)).$ Then \[p=s^{2k+2l}-s^{2k+l}+s^l-1=(s^l-1)(s^{2k+l}+1).\]  The equation $(s^{2k+l}+1)$ is separable as $(s^{2k+l}+1)(s^{2k+l}-1)=s^{4k+2l}-1$ is separable from above, as is $s^l-1$.  It suffices to see that they do not share common roots for $l=1$ and 2.  

Finally, we set $p=(s^{n+1}-s^k)(f_{n+1}(u)+f_n(u))$ so that \[p=s^{2n+2}+s^{2n+1}-s-1=(s+1)(s^{2n+1}-1)\]
The separability of $p$ and hence of $f_{n+1}(u)+f_n(u)$ follows from the separability above as $s+1$ and $s^{2n+1}-1$ do not share any factors. 
\end{proof}

We will also make use of the following identity when computing $\PSL_2(\C)$ character varieties.  It can be easily verified using the closed form for $f_n(u)$ writing $u=s+s^{-1}$.  Compare with the similar identity given in Lemma~\ref{lemma:fibonacciidentities}.
\begin{lemma}\label{lemma:pslfib}
For all $n$, $(f_{n-1}(u)+1)(f_{n+1}(u)+1)=f_n(u)(u+f_n(u))$. \qed
\end{lemma}

%
%
\section{Character Variety Calculations}\label{section:charactervarietycalculations}
%
%

We now calculate the character variety $\tilde{X}(\Gamma_n)$ of
\[ \Gamma_n =\langle \alpha, \beta : \bar{\beta}^n=\omega \rangle\]
where $\omega= \bar{\alpha}\beta \alpha^2\beta\bar{\alpha}$.

The full $\SL_2(\C)$ character variety, $\tilde{X}(\Gamma_n)$ is the union of the reducible characters, $\tilde{X}_{red}(\Gamma_n)$, and the irreducible characters, $X(\Gamma_n)$.  (The set $\tilde{X}_{red}(\Gamma_n)$ is Zariski closed.  By definition  $X(\Gamma_n)$ is the Zariski closure of its complement, $\tilde{X}(\Gamma_n)-\tilde{X}_{red}(\Gamma_n).$) 
In this section we explicitly compute these two sets; later, in \S\ref{section:intersections}  we investigate their intersection.  
 First, in \S\ref{section:exceptionalrepresentations} and \S\ref{section:genericreducibles} we study the reducible representations of $\Gamma_n$.  We will show that $\tilde{X}_{red}(\Gamma_n)$ is exactly the set of characters of abelian representations of $\Gamma_n$ and as well as the set of characters of diagonal representations.  Thereafter we focus on the irreduble characters.   In \S\ref{section:irreps} we determine a natural model for $X(\Gamma_n)$.  We compute the character varieties of the non-hyperbolic $M_n$ (that is, for $|n|\leq 2$) in \S\ref{section:nonhyperbolic}.  In \S\ref{section:multiplicity}  we examine the points of multiplicity two in $X(\Gamma_n)$.   In \S\ref{section:newmodel} we determine a new model for $X(\Gamma_n)$, birational to the natural model.   When $M_n$ is hyperbolic, we will show that there is a unique canonical component $X_0(\Gamma_n) \subset X(\Gamma_n)$.  When $n\not \equiv 2\pmod 4$ this is the whole of $X(\Gamma_n)$; that is, $X_0(\Gamma_n)=X(\Gamma_n)$.  When $n\equiv 2 \pmod 4$, we show that $X(\Gamma_n)=X_0(\Gamma_n) \cup L$, where $L$ is birational to $\A^1$.  We determine smooth models for the canonical components $X_0(\Gamma_n)$.  Though we refer to them generically as hyperelliptic curves for $|n|>2$, these are actually rational surfaces when $n=3,4,6$ and elliptic curves when $n=-6,-4,-3,5$.  This is summarized in the following theorem.

\begin{thm}\label{thm:mainsummary1}
 If $|n|>2$  then $X_0(\Gamma_n)$, the unique canonical component, is birational to the hyperelliptic curve given by  $w^2=-\hat{h}_n(y)\hat{\ell}_n(y)$.  If $n\not \equiv 2 \pmod 4$ this is the only component of $X(\Gamma_n)$ and has genus $\lfloor \tfrac12 \mid n-1\mid - \: 1 \rfloor$.  If $n\equiv 2 \pmod 4$ there is an additional $\A^1$ component and the genus of $X_0(\Gamma_n)$ is $\lfloor \tfrac12  \mid n-1 \mid -\: 2 \rfloor$.
\end{thm}

\begin{remark}
When $n\not \equiv 2 \pmod 4$ this hyperelliptic curve is $w^2=1-f_{n-1}(y)$.  For $n\equiv 2 \pmod 4$, Proposition~\ref{prop:extraline} describes the characters on the $\A^1$ component.  This line component is the line $L$ discussed  in Section~\ref{section:newmodel} and in particular in   Proposition~\ref{prop:extraline3}.
\end{remark}

Let us first set up some basic terminology and facts about representations of two generator groups.

\subsection{$\SL_2(\C)$ representations of two generator groups.}

Recall from Section~\ref{sec:charvar} that for each element $\gamma$ in a finitely generated group $\Gamma$, $t_{\gamma}$ is the function $t_{\gamma}\colon R(\Gamma) \rightarrow \A$ defined by $t_{\gamma}(\rho)=\tr(\rho(\gamma))$. The ring  $T$ is the subring of the ring of all regular functions from $R(\Gamma)$ to $\A$ generated by $1$ and the $t_{\gamma}$.  If $\Gamma$ is generated by $\alpha$ and $\beta$, then this ring $T$ is generated by $t_{\alpha}$, $t_{\beta}$, and $t_{\alpha \beta}$.  Therefore the $\SL_2(\C)$ character variety $\tilde{X}(\Gamma)$ can be identified with the image of $R(\Gamma)$ under the map \[ (t_{\alpha}, t_{\beta}, t_{\alpha \beta}):R(\Gamma_n) \rightarrow \mathbb{A}^3.\]

\begin{definition}\label{definition:xyz}
Assume $\Gamma$ is generated by $\alpha$ and $\beta$.  
Then for a representation $\rho:\Gamma \rightarrow \SL_2(\C)$ let   
\[ x=t_{\alpha}=\text{tr}(\rho(\alpha)),  \ \ y=t_{\beta}=\text{tr}(\rho(\beta)), \ \ \text{ and } \  \ z=t_{\alpha \beta}=\text{tr}(\rho(\alpha \beta)).\]    
\end{definition}

\begin{lemma}\label{lemma:2generatorconjugation} 
Let $\Gamma$ be a  two generator group, generated by $\gamma_1$ and $\gamma_2$ with representation $\rho:\Gamma \rightarrow \SL_2(\C)$. Then up to conjugation in $\SL_2(\C)$, 
\[  \rho(\gamma_1) = \left( \begin{array}{cc} g_1 & u \\ t & g_1^{-1} \end{array}\right) 
\quad\quad\quad
\rho( \gamma_2) = \left( \begin{array}{cc} g_2 & s \\ 0 & g_2^{-1} \end{array}\right)  \]
 where $tu=0$.  
 
 Furthermore: 
 \begin{itemize}
 \item[(i)] If $st=0$ the representation is reducible. (If $s=0$ the representation is conjugate to both a representation of this form with $t=0$ and one of this form with $u=0$.)
 \item[(ii)] If $u=0$ then up to conjugation either $s=t$ or $\{s,t\}=\{0,1\}$. 
\end{itemize}
\end{lemma}

\begin{proof}
Let $G_i=\rho(\gamma_i)$ for $i=1,2$ and let $X=  \left( \begin{array}{cc} \alpha & \beta \\ 0 & \alpha^{-1} \end{array}\right)$.
We begin by conjugating so that $G_2$ has the above form and 
\[ G_1 =  \left( \begin{array}{cc} a & b \\ c & d \end{array}\right)\]
for $a,b,c,d \in \C$ and $ad-bc=1$.
If $c=0$, then $\rho$ has the above form with $t=0$. 
If $c \neq 0$, then conjugate by $X$ with $\alpha=1$ to obtain
\begin{align*} X^{-1}G_1 X & =  \left( \begin{array}{cc} a-\beta  c & \beta (a-d)-\beta^2  c+b \\ c & \beta  c+d \end{array}\right)\\
 X^{-1}G_2X & =  \left( \begin{array}{cc} g_2 & \beta(g_2-g_2^{-1})+s \\ 0 & g_2^{-1} \end{array}\right). \end{align*}
 Since $c\neq 0$, we can solve the equation $ -\beta^2  c+\beta (a-d)+b=0$ with the appropriate choice of $\beta$ so that $\rho$ has the above form with $u=0$.  This is the initial statement of the lemma.

Now consider $\rho(\gamma_i)$ as given in the lemma.
If $t=0$, then the representation is upper triangular and hence reducible.   If $t\neq0$ then by conjugation we may assume $u=0$ so that if $s=0$ then the representation is lower triangular and hence reducible.  
If $s=0$, then conjugating by $\left( \begin{array}{cc} 0 & -1 \\ 1 & 0 \end{array}\right)$ keeps the matrices in the above form but switches the role of $u$ and $t$  since $tu=0$. Thus we have (i).

If $u=0$, then conjugating by $X$ with $\beta =0$ gives 
\[ X^{-1}G_1 X  =   \left( \begin{array}{cc} g_1 & 0 \\ \alpha^2t & g_1^{-1} \end{array}\right) 
 \quad \text{and} \quad 
X^{-1}G_2 X  = \left( \begin{array}{cc} g_2 & s / \alpha^2 \\ 0 & g_2^{-1} \end{array}\right). \]
from which (ii) follows upon choosing a suitable $\alpha$. 
\end{proof}

\begin{definition}\label{defn:genericexceptional}
For $a,b \in \C^*$ and $s,t \in \C$, define
\[ A(a,t) = \mat{a}{0}{t}{a^{-1}} 
\quad \text{ and } \quad 
B(b,s) = \mat{b}{s}{0}{b^{-1}}.\]

Let $\Gamma$ be a group generated by $\alpha$ and $\beta$.
A representation $\rho \colon \Gamma \rightarrow \text{SL}_2(\C)$  that is conjugate to a representation of the form 
\[ \alpha \mapsto A(a,t) \quad \text{ and } \quad 
\beta \mapsto B(b,s) \]
 is  {\em generic}.  We refer to  a representation given in the form above as being in {\em standard form}.
 Any representation conjugate to 
\[ \alpha \mapsto \pm \left( \begin{array}{cc} 1 & 1 \\ 0 & 1 \end{array} \right) \quad \mbox{ and } \quad 
\quad \beta \mapsto \pm \left( \begin{array}{cc} 1 & s \\ 0 & 1 \end{array} \right), s \neq 0\]  
is called  {\em exceptional}.   
Note that an exceptional representation is both reducible and abelian.

\end{definition}

\begin{remark}\label{rem:genericxyz}
For a generic representation, Definition~\ref{definition:xyz} gives $x=a+a^{-1}$, $y=b+b^{-1}$, and $z=ab+a^{-1}b^{-1}+st$.  
\end{remark}

\begin{lemma}\label{lemma:genericreps} 
Let $\Gamma$ be generated by two elements. 
A representation $\rho \colon \Gamma \rightarrow \SL_2(\C)$ is either generic or exceptional.  
Any generic reducible representation is conjugate to a representation in standard form with $st=0$.
Any irreducible representation is conjugate to a representation in standard form with $s=t\neq0$. 
\end{lemma}

\begin{proof} Let $\alpha$ and $\beta$ be generators of $\Gamma$.
First, we will assume that $\rho$ is reducible. 
Since $\rho$ is reducible,  $\rho$ is conjugate to an upper triangular representation $\rho'$ of the form
\[ \rho'(\alpha) = \left( \begin{array}{cc} a &v \\ 0 & a^{-1} \end{array} \right) \quad \text{and} \quad 
\rho'(\beta) = \left( \begin{array}{cc} b & s \\ 0 & b^{-1} \end{array} \right). \]
Define the matrices $X$ and $F$ by 
\[ X =\mat{\sigma}{ \gamma}{0}{ \sigma^{-1} }, \quad F=\mat{0}{-1}{1}{0}.\]
Conjugating $\rho'$ by $X$ we obtain the representation $\rho$ defined by  
\begin{align*}
\alpha \mapsto \left( \begin{array}{cc} a &   \sigma^{-1}(  \gamma [a-a^{-1}]+v \sigma^{-1})  \\ 0 & a^{-1} \end{array} \right), \\ 
  \beta \mapsto  \left( \begin{array}{cc} b &  \sigma^{-1}(\gamma [b-b^{-1}] +s \sigma^{-1} ) \\ 0 & b^{-1} \end{array} \right).  \end{align*}

Unless $a=a^{-1}$, choosing $X$ with $\gamma=-v[a-a^{-1}]^{-1} \sigma^{-1}$   makes  $\rho(\alpha)$  a diagonal matrix. 
This representation is in standard form. Similarly, if $b\neq b^{-1}$ choosing $\gamma=-s[b-b^{-1}]^{-1} \sigma^{-1}$ makes $\rho(\beta)$ a diagonal matrix.

Therefore we will assume that $a=\pm1$ and $b=\pm1$.  Such a  representation $\rho$ is conjugate to 
\[ \alpha \mapsto \mat{a}{v\sigma^{-2}}{0}{a^{-1}}, \quad \beta \mapsto \mat{b}{s\sigma^{-2}}{0}{b^{-1}}.\]
If $s$ or $v$ is zero then either $\rho(\alpha)$ or $\rho(\beta)$ is diagonal and $\rho$ is in standard form or can be conjugated into standard form by $F$.  Otherwise, we take $\sigma$ so that $\sigma^2=\pm v$ and $\rho$ is defined by  (since $a,b \in \{\pm 1\}$)
\[ \alpha \mapsto \pm \mat{1}{1}{0}{1}, \quad \beta \mapsto \pm \mat{1}{\pm sv}{0}{1}.\]
Such a representation is exceptional.  We have shown that a reducible representation is either exceptional or is conjugate to  a representation in standard form with $st=0$.

Now we consider irreducible representations.  By Lemma~\ref{lemma:2generatorconjugation}  we  may take $\alpha=\gamma_1$ and $\beta=\gamma_2$, with $u=0$ and $st\neq 0$ since the representation is irreducible.  Therefore, we can take $s=t$.
\end{proof}

\bigskip

From here on we focus on the group $\Gamma_n$ and consider its representations $\rho \colon \Gamma_n \to \SL_2(\C)$.  We will use the variables $(x,y,z)$ of Definition~\ref{definition:xyz} (and Remark~\ref{rem:genericxyz}) as the variables of definition for $\tilde{X}(\Gamma_n)$.

The following  lemma  will be useful in identifying canonical components.

\begin{lemma}\label{lemma:disc1}
If $M_n$ is hyperbolic and  $\rho_0$ is discrete and faithful then 
$x,y \not \in  \Rb_k$ for any $k\neq0$.
\end{lemma}

\begin{proof}\label{lemma:tracelambda}
By Lemma~\ref{lemma:nonperipheralgenerators}, the elements $\alpha$ and $\beta$ are hyperbolic.  Hence neither $x= \tr(\rho_0(\alpha))$ nor $y=\tr(\rho_0(\beta))$ is $\pm2$.
If  $x=2\Re(\zeta)$ for some root of unity $\zeta (\neq \pm1)$ then $x=\zeta+\zeta^{-1}$.
Since a discrete faithful representation is generic, we may take
\[ \rho(\alpha)^{\pm1} = \left( \begin{array}{cc} \zeta & 0 \\  \tau  &  \zeta^{-1} \end{array} \right).\] By the Cayley-Hamilton theorem, if $\zeta^k=1$ then 
\[ \rho_0(\alpha)^{\pm k} = \left( \begin{array}{cc} \zeta^k & 0 \\  \tau f_k(x) &  \zeta^{-k} \end{array} \right)=I\] 
implying that $\alpha$ is not a hyperbolic element, a contradiction.   A similar argument applies for $y$.
\end{proof}

\subsection{Exceptional Representations}\label{section:exceptionalrepresentations}

\begin{lemma}\label{lemma:exceptional}
A representation $\rho:\Gamma_n \rightarrow \text{SL}_2(\C)$ is exceptional only if $n=-2$. 
Furthermore for such a representation $\rho$, 
\[(x,y,z) \in \{ (2,2,2), (2,-2,-2), (-2,2,-2), (-2,-2,2)\}.\]
\end{lemma}

\begin{proof}
Since $\rho$ is exceptional it is conjugate to a representation $\tau$ such that 
\[ \tau(\alpha) = \epsilon_\alpha \left( \begin{array}{cc} 1 & 1 \\ 0 & 1 \end{array} \right), 
\quad \mbox{and} \quad \tau(\beta) = \epsilon_\beta \left( \begin{array}{cc} 1 & s \\ 0 & 1 \end{array} \right)\]  
where $\epsilon_\alpha = \pm1$, $\epsilon_\beta = \pm1$, and $s \neq 0$.  The group relation in $\Gamma_n$ under $\tau$ implies that 
\[ (\epsilon_\beta)^n \left( \begin{array}{cc} 1 & -ns \\ 0 & 1 \end{array} \right) = \left( \begin{array}{cc} 1 & 2s \\ 0 & 1 \end{array} \right). \]
Since $s \neq 0$ for an exceptional representation we must have $n=-2$  and hence there are no restrictions on $\epsilon_\alpha$ or $\epsilon_\beta$.  A computation shows that $\tr(\tau(\alpha \beta)) = 2\epsilon_\alpha \epsilon_\beta$.
\end{proof}

\begin{remark} If $\rho$ is an exceptional representation, then $\rho$ is abelian and $\rho(\alpha)$ has infinite order.   If $s\not \in \Q$ then $\rho(\Gamma_{-2}) \cong \Z \times \Z$. Otherwise if $s\in \Q$ then  up to conjugation  \[ \rho(\alpha)^c\rho(\beta)^{-d}=\pm \mat{1}{c-sd}{0}{1}\] and there are integers $c$ and $d$ so that $\rho(\alpha)^c= \rho(\beta)^d$. We see that $\rho(\Gamma_{-2}) \cong \Z \times \Z / \ell \Z$ for some non-negative integer $\ell$.   The manifold $M_{-2}$ is toroidal and hence not hyperbolic, as discussed in Remark~\ref{remark:M-2}. 
\end{remark}

\subsection{Generic Reducible and Abelian Representations}\label{section:genericreducibles}

We will now explicitly study the structure of the reducible representations.  We show that all characters of reducible representations are characters of abelian representations, and these are all characters of diagonal representations.  (The results of this section are summarized in Proposition~\ref{prop:redreps}.) We also explicitly compute these sets.  These explicit calculations will also be used in Section~\ref{section:intersections} to study the geometry of the intersection of the reducible representations with the canonical component.  
We use notation set in Definition~\ref{definition:rootsofunity}.

\begin{lemma}\label{lemma:redchars}
The characters of any reducible representation  of $\Gamma_n$ satisfy  
\[ x^2+y^2+z^2-xyz=4. \]
These are all characters of generic representations. 
Furthermore if $n\neq -2$ then $y\in \Rb_{n+2}$, where either $y\neq \pm 2$ and these characters are irreducible conics or $y=\pm 2$ and they are the lines given by $x=\pm z$. 
\end{lemma}

\begin{proof}
Let $\rho$ be a non-exceptional reducible representation. 
Lemma~\ref{lemma:genericreps} implies $\rho$ is generic. Thus by conjugation we may take $\rho$ to be in standard form with $st=0$.   As $st=0$ we see that $z=\tr(\rho( \alpha \beta))$ depends on $x$ and $y$.  Specifically, $x=a+a^{-1}$, $y=b+b^{-1}$ and $z=ab+a^{-1}b^{-1}$ as in Remark~\ref{rem:genericxyz}, and so
\[ 2z = (a+a^{-1})(b+b^{-1}) + (a-a^{-1})(b-b^{-1}) = xy + (a-a^{-1})(b-b^{-1})\]
Since $(a-a^{-1})^2=x^2-4$ and $(b-b^{-1})^2=y^2-4$ we conclude that 
\[ (2z-xy)^2=(x^2-4)(y^2-4)\] 
for all non-exceptional reducible representations.  This reduces to 
\[ x^2+y^2+z^2-xyz=4.\]

The single relation in $\Gamma_n$ implies that $B^{-n} = A^{-1}BA^2BA^{-1}$. The $(1,1)$ entry of $B^{-n}$ is $b^{-n}$, and its elementary to see that the $(1,1)$ entry of $A^{-1}BA^2BA^{-1}$ is
\[  b^2-sta^{-3}b^{-1}-(st)^2(1+a^{-2})+st(-a+a^{-1}+a^{-3})b. \] 
Since $st=0$ we conclude that $b^{-n}=b^2$ and therefore $b\in \Zeta_{n+2}$ for $n\neq -2$. We conclude that $y=b+b^{-1} \in \Rb_{n+2}.$ 

It is elementary to verify that for $y \in \Rb_{n+2}^{fib}$ that $x^2+y^2+z^2-xyz=4$ is irreducible, and if $y=\pm2$ the equation reduces to $(x\mp z)^2=0$ and determines a line. 

Now let $\rho$ be an exceptional representation.  
By Lemma~\ref{lemma:exceptional} $(x,y,z)$ is one of 
\[  (2,2,2), (2,-2,-2), (-2,2,-2), \  \text{or}   \ (-2,-2,2).\]  Each of these points satisfies the equation $x^2+y^2+z^2-xyz=4$.  Therefore, every character of a reducible representation is the character of a generic representation. 
\end{proof}

\begin{remark}  Since representations are either generic or exceptional, and exceptional representations are reducible,  Lemma~\ref{lemma:redchars} implies that all characters are characters of generic representations.\end{remark}

\begin{prop}\label{prop:redreps}  
The abelianization of $\Gamma_n$ is 
\[ \Gamma_n^{ab}= \langle \alpha, \beta : \beta^{n+2} \rangle \cong 
\begin{cases} 
\Z \times \Z/(n+2)\Z &\mbox{ if } n \neq -2 \\
\Z \times \Z &\mbox{ if } n = -2.
\end{cases}
 \]  
The character variety $\tilde{X}(\Gamma_n^{ab})$ for $n\neq -2$ is determined by the vanishing set  of the $\lfloor \frac12|n+2|+1\rfloor$ genus zero conics and lines
\[x^2+y^2+z^2-xyz= 4 \mbox{ such that } y\in \Rb_{n+2}. \] 
The character variety $\tilde{X}(\Gamma_{-2}^{ab})$ is determined by the vanishing set of
 \[x^2+y^2+z^2-xyz= 4 \]
 (without restriction on $y$).

For all $n$, 
\[ \tilde{X}_{red}(\Gamma_n)=\tilde{X}(\Gamma_n^{ab})=\tilde{X}_a(\Gamma_n)=\tilde{X}_d(\Gamma_n)\]
where $\tilde{X}_a(\Gamma_n)$ are the abelian representations of $\Gamma_n$ and $\tilde{X}_d(\Gamma_n)$ are the diagonal representations of $\Gamma_n$. 

\end{prop}

We omit the proof as it is straightforward and similar to previous proofs. 
%
%
%
%
%
%

\subsection{Irreducible Representations}\label{section:irreps}

We begin with a few matrix calculations that will be essential in our  calculation of $X(\Gamma_n)$.

\begin{definition} For a matrix $M$, let $M_{ij}$ denote the $(i,j)^{th}$ entry of $M$, and let $t_M$ denote $\tr M$. 

Let $A=A(a,t)$ and $B=B(b,s) \in \text{SL}_2(\C)$ be as in Definition~\ref{defn:genericexceptional}.  Define
\[ W(a, b, s, t) = A^{-1}BA^2BA^{-1} \]
 and 
\[F(a, b, s, t)= B^{-n}-W(a, b, s, t).\]
\end{definition}

\begin{definition}\label{defn:charvareqns}
In the ring $\Q[a^{\pm 1},b^{\pm 1}, s]$, let $H$ denote the ideal generated by the four entries of the matrix $F = F(a,b,s,s)$.   That is, $H$ is generated by $F_{11}$, $F_{12}$, $F_{21}$, $F_{22}$.
\end{definition}

\begin{prop}\label{prop:charvareqns}
The ideal $H$ is  generated by $D=F_{11}+F_{22}$, $F_{12}$, $F_{21}$ and $S=F_{11}-F_{22}$.  
\end{prop}

\begin{proof}
The ideal $H$ is generated by $F_{11}$, $F_{12}$, $F_{21}$, $F_{22}$ by definition.  Since $D+S=2F_{11}$ and $D-S=2F_{22}$ it follows that $H$ is also generated by $D$, $S$, $F_{12}$ and $F_{21}$ as well.
\end{proof}

\begin{lemma}\label{lemma:firstequations}  The entries of $F(a,b,s,t)$  are given by 
\begin{align*}
F_{11}& = b^{-n} -\Big( b^2-sta^{-3}b^{-1}-(st)^2(1+a^{-2})+st(-a+a^{-1}+a^{-3})b\Big) \\
F_{12} & = -sf_n(t_B)  - s(t_{AB}t_A-t_B)\\
F_{21}& =  -t(t_At_{AB}^2-t_A^2t_Bt_{AB}+t_A^3+t_At_B^2-t_Bt_{AB}-2t_A)\\
F_{22} & = b^{n} -\Big(b^{-2} -stba^3-(st)^2(1+a^2)-st(-a+a^{-1}-a^3)b^{-1}\Big).
\end{align*}

\end{lemma}

\begin{proof}
The Cayley-Hamilton theorem then implies that 
\[ B^{-n} = \left( \begin{array}{cc} b^{-n} & -sf_n(t_B) \\ 0 & b^n \end{array} \right)\]
using $t_B=b+b^{-1}$ with $f_n(t_B)$, the $n^{th}$ Fibonacci polynomial. 
 The matrix $W=A^{-1}BA^2BA^{-1}$ and therefore $F$ is then obtained by direct calculation.
\end{proof}

\begin{lemma}\label{lemma:secondequations}  Let
\[ F_{12}' = f_n(t_{B})+t_{AB}t_{A}-t_{B}\]
and
\[ F_{21}' = t_{A}t_{AB}^2-t_{A}^2t_{B}t_{AB}+t_{A}^3+t_{A}t_{B}^2-t_{AB}t_{B}-2t_{A}. \]
The ideal $H$ is generated by  $F_{12}  =  -sF_{12}'$, $F_{21} =  sF_{21}'$
\[D  = f_{n+1}(t_{B})-f_{n-1}(t_{B}) + \Big(t_{A}^2t_{AB}^2-t_{A}^3t_{B}t_{AB}+t_{A}^4+t_{A}^2t_{B}^2-4t_{A}^2-t_{B}^2+2   \Big)  \]
and
\[ S  = -(b-b^{-1})F_{12}' -(a-a^{-1}) F_{21}'.\]
\end{lemma}

\begin{proof}
Proposition~\ref{prop:charvareqns} implies that $H$ is generated by $D,S,F_{12}$ and $F_{21}$.  Since $D=F_{11}+F_{22}$ and $S=F_{11}-F_{22}$,  Lemma~\ref{lemma:firstequations} gives exact expressions for these four polynomials in terms of the variables $a, b$, and $s$.  An elementary calculation using these explicit equations shows that they can be written as stated.
\end{proof}

Now we connect the matrix calculations above with the coordinate ring of $X(\Gamma_n)$ which determines the character variety.  First, we make a brief remark about our notation.

\begin{remark}
In Definition~\ref{definition:xyz} we established that 
\[ x=t_{\alpha} = \tr(\rho(\alpha)), \ \ y=t_{\beta}=\tr(\rho(\beta)), \ \ z=t_{\alpha \beta}=\tr(\rho(\alpha \beta)). \]  
By Lemma~\ref{lemma:redchars} all characters are characters of generic representations, so we may take $\rho$ to be in standard form with $\rho(\alpha)=A(a,s)$ and $\rho(\beta)=B(b,t)$.  Therefore,  if $\rho$ is generic
\[ x=t_A= a+a^{-1}, \ \ y=t_B = b+b^{-1}, \ \ z=t_{AB}= ab+a^{-1}b^{-1}+st.\] 
If $\rho$ is irreducible,  we may assume that $s=t$. 
\end{remark}

The following proposition summarizes the concrete  relationship between the matrix equations and the character variety as the set determined by the trace maps $x=t_{\alpha}$, $y=t_{\beta}$, and $z=t_{\alpha \beta}$.

\begin{prop}\label{prop:generalreps}
The map \[(t_\alpha, t_\beta, t_{\alpha\beta})(\rho) = (a+a^{-1}, b+b^{-1}, ab+a^{-1}b^{-1}+s^2)\]  identifies $\tilde{X}_{irr}(\Gamma_n)=\tilde{X}(\Gamma_n)-\tilde{X}_{red}(\Gamma_n)$ with  the subset of $\mathbb{A}^3(x,y,z)$ where $x=a+a^{-1}$, $y=b+b^{-1}$ and $z=ab+a^{-1}b^{-1}+s^2$ and $s\neq 0$. Under this identification, a point $(x,y,z) \in \mathbb{A}^3$ is contained in  $X(\Gamma_n)=\overline{ \tilde{X}_{irr}(\Gamma_n)}$  if and only if there are are 
$a, b \in \C^*$ and $s\in \C$, such that 
$(x,y,z)=(a+a^{-1}, b+b^{-1}, ab+a^{-1}b^{-1}+s^2)$ and
the assignments 
\[ \alpha \mapsto A(a,s) \quad \text{ and } \quad \beta \mapsto B(b,s)\]
can be extended to a representation $\rho \in R(\Gamma_n)$. 
\end{prop}

\begin{proof}
By Lemma~\ref{lemma:genericreps}, an irreducible representation is conjugate to one in standard form with $s=t\neq0$.  Therefore  its traces give a point in $\A^3$ with  the desired form.
Conversely, given such a point, the representation it  induces is generic and in standard form. By Lemma~\ref{lemma:genericreps} it is irreducible if $s\neq 0$.
\end{proof}

With $A=\rho(\alpha)$ and $B=\rho(\beta)$ the single group relation  in $\Gamma_n$ corresponds to the matrix $F$. 
The assignment $\alpha \mapsto A(a,s)$ and $\beta \mapsto B(b,s)$ extends to a representation of $\Gamma_n$ if and only if $F(a,b,s,s)=0$. Therefore, we have the following proposition.  We use the notation $F_{12}'$ and $F_{21}'$ which appeared in Lemma~\ref{lemma:secondequations}  as these are the same polynomials as below when $A=\rho(\alpha)$ and $B=\rho(\beta)$.

\begin{prop}\label{prop:coordring1}
The coordinate ring of $X(\Gamma_n)$ is $\A^3[x,y,z]/(D',F_{12}', F_{21}')$ where 
\begin{align*}
D' & = f_{n+1}(y)-f_{n-1}(y)+ \Big(x^2z^2-x^3yz+x^4+x^2y^2-4x^2-y^2+2   \Big)  \\
F_{12}' & = f_n(y)  +  zx-y\\
F_{21}' & =  xz^2-x^2yz+x^3+xy^2-yz-2x .
\end{align*}
\end{prop}

\begin{proof}
By Proposition~\ref{prop:generalreps}, the point $P=(x,y,z)$ is contained in $X(\Gamma_n)$ if and only if $\rho$ extends to a representation of $\Gamma_n$.  This occurs exactly when the group relation is satisfied.  On the level of matrices, this occurs exactly when $F$ is the zero matrix.   The condition that $F=[F_{ij}]$ is the zero matrix is determined by the vanishing of the polynomials $F_{ij}$, which are polynomials in $\A[a^{\pm 1}, b^{\pm 1}, s]$.  The ideal $H$ is generated by these $F_{ij}$.   
By Lemma~\ref{lemma:secondequations}, $H$ is generated by $D'$, $-sF_{12}'$, $sF_{21}'$ and the element $S$, under the identification $x=t_A = a+a^{-1}$, $y=t_B = b+b^{-1}$, and $z=t_{AB} = ab+a^{-1}b^{-1}+s^2$.  
Since $s \neq 0$ for irreducible representations by Lemma~\ref{lemma:genericreps}, we have
$sS\in (  F_{12}, F_{21} )$ and $S \in (  F_{12}', F_{21}' )$. 
It follows that $X(\Gamma_n)$ is determined by the vanishing set of $D'$, $F_{12}'$, and $F_{21}'$ as written above, and that the coordinate ring of $X(\Gamma_n)$ is $\A^3[x,y,z]/ (D',F_{12}', F_{21}')$.
\end{proof}

We now determine a nicer form for the coordinate ring of $X(\Gamma_n)$ using alternative generators for the vanishing ideal.

\begin{definition}\label{definition:phis}
Let 
\begin{align*}
\varphi_1(x,y,z) &= x^2-1+f_{n-1}(y)=x^2+h_n(y)\ell_n(y)\\ 
\varphi_2(x,y,z) &=zx-y+f_n(y)=zx+h_n(y)k_n(y)\\
\varphi_3'(x,y,z) & = xk_n(y)-z\ell_n(y)\\
\varphi_3(x,y,z) & =x( f_{n+1}(y)-1)-zf_n(y)=j_n(y) \varphi_3' 
\end{align*} 
and let  $I'= ( \varphi_1, \varphi_2, \varphi_3' )$ and $I=( \varphi_1,\varphi_2,\varphi_3).$
We define the algebraic sets   $C'$ and $C$ as the vanishing sets of $I'$ and $I$, respectively.  That is, the coordinate rings of $C'$ and $C$ are $\A^3[x,y,z]/I'$ and  $ \A^3[x,y,z]/I$.
\end{definition}

\begin{prop}\label{prop:irreps}
The variety $X(\Gamma_n)$  equals the  variety $C$ from  Definition~\ref{definition:phis}. 
\end{prop}

\begin{proof}
It suffices to determine the coordinate ring.
By Proposition~\ref{prop:coordring1} the coordinate ring is generated by  $D'$, $F_{12}'$, and $F_{21}'$.  By construction: 
\begin{align*}
 \varphi_1& = yF_{12}' -\tfrac12  (D'+yF_{12}'-xF_{21}' )=  x^2-1 + f_{n-1}(y),  \\
 \varphi_2 & = F_{12}' =zx-y+f_n(y), \\
 \varphi_3 & = F_{21}'-(z-xy)F_{12}'-x\varphi_1 = x(f_{n+1}(y)-1)-z f_n(y).
 \end{align*}
As $\varphi_1$ is linear in $D'$, $\varphi_2$ is $F_{12}'$, and $\varphi_3$ is linear in $F_{21}'$ we conclude that  $I$ is generated by $\varphi_1$, $\varphi_2$, and $\varphi_3$. 
\end{proof}

\subsection{The Non-hyperbolic Cases}\label{section:nonhyperbolic}

By Lemma~\ref{lem:geometry}, the manifolds $M_n$ are non-hyperbolic if and only if $|n|\leq 2$.  
In this section we determine the entire $\text{SL}_2(\C)$ character variety of $\Gamma_n = \pi_1(M_n)$, i.e.\ both varieties $\tilde{X}_{red}(\Gamma_n)$ and $X(\Gamma_n)$, in these non-hyperbolic cases.  

\begin{prop}
The $\SL_2(\C)$ character variety  $\tilde{X}(\Gamma_n)=\tilde{X}_{red}(\Gamma_n) \cup X(\Gamma_n)$ for integers $|n|<2$ is given in Table~\ref{table:nonhypchar}.
\begin{table}
\begin{tabular}{c|c|c}
$n$ & $\tilde{X}_{red}(\Gamma_n)$ & $X(\Gamma_n)$ \\ \hline\hline
$2$ & $\{ (x,2,x)\} \cup \{(x,-2,-x)\} \cup \{ (x,0,z): x^2+z^2=4\}$ &$\{(0,y,0)\} \cup \{ (0,0,z)\}$\\
$1$ & $\{(x,2,x) \} \cup \{(x,-1,z): x^2+z^2+xz=3\}$ & $\{(1 , y, y-1)\}\cup \{(-1 , y, -y+1)\}$\\
$0$ & $\{ (x, 2, x)\} \cup \{(x,-2, -x)\}$ & $\{( \sqrt{2},  \sqrt{2}z, z)\}\cup \{( -\sqrt{2},  -\sqrt{2}z, z)\}$ \\ 
$-1$ &$\{ (x,2,-x)\}$ & $\{(x,x^2-1,x)\}$\\ 
$-2$ & $\{ (x,y,z) : x^2+y^2+z^2-xyz=4\}$ &  $\{(0,0,z)\} \cup \{(2,2,2)\} \cup \{(-2,2,-2)\}$
\end{tabular}
\caption{The $\SL_2(\C)$ character variety of $\Gamma_n = \pi_1(M_n)$ for non-hyperbolic $M_n$.}
\label{table:nonhypchar}
\end{table}

\end{prop}

\begin{proof}
Recall that from Proposition~\ref{prop:redreps} the reducible representations satisfy $x^2+y^2+z^2-xyz=4$ and, when $n\neq -2$,  $y\in \Rb_{n+2}$.  By Proposition~\ref{prop:irreps} the irreducible characters 
$X(\Gamma_n)$ are determined by the vanishing set of $\varphi_1$, $\varphi_2$, and $\varphi_3$.  Here we take $\epsilon = \pm1$.  Recall from Definition~\ref{definition:rootsofunity} that $\Zeta_{n+2}$ is the set of all $|n+2|^{th}$ roots of unity from which we may determine $\Rb_{n+2}$.

{\bf Case $n=2$:} 
Then $\Zeta_4= \{\pm 1,  \pm i\}$ and $y\in \{-2, 0, 2\}$.  The reducible characters are  two lines and  a conic
\[\tilde{X}_{red}(\Gamma_2)= \{ (x,2,x)\} \cup \{(x,-2,-x)\} \cup \{ (x,0,z): x^2+z^2=4\}.\]
The irreducible characters are determined by the vanishing set of the polynomials $\varphi_1$, $\varphi_2$, and $\varphi_3$.  These reduce to  $x^2$, $zx$, $x(y^2-1)-zy$ so that the irreducible characters are two lines,
\[X(\Gamma_2) = \{(0,y,0)\} \cup \{ (0,0,z)\}.\]

{\bf Case $n=1$:}
Then $\Zeta_3=\{1, (-1\pm \sqrt{-3})/2\}$ and so $y\in \{-1,2\}$.   The reducible characters are a line and conic,
\[ \tilde{X}_{red}(\Gamma_1)= \{(x,2,x) \} \cup \{(x,-1,z): x^2+z^2+xz=3\}.\] 
The polynomials $\varphi_1$, $\varphi_2$, and $\varphi_3$ reduce to 
$ x^2-1$, $zx-y+1$, and  $x(y-1)-z$ so that the irreducible characters are two lines,
\[X(\Gamma_1) = \{(1 , y, y-1)\}\cup \{(-1 , y, -y+1)\}.\]

{\bf Case $n=0$:}
In this case, $\Zeta_2=\{\pm 1\}$ so $y=2\epsilon$.  Therefore, the characters of the reducible representations satisfy $x=\epsilon z$.  
The reducible representations consist of  two lines,
\[\tilde{X}_{red}(\Gamma_0)= \{ (x, 2, x)\} \cup \{(x,-2, -x)\}.\]  
The polynomials $\varphi_1$, $\varphi_2$, and $\varphi_3$ reduce to 
$x^2-2$, $zx-y$, and $0$ so that the irreducible characters are two lines,
\[X(\Gamma_0)=\{( \sqrt{2},  \sqrt{2}z, z)\}\cup \{( -\sqrt{2},  -\sqrt{2}z, z)\}.\]

{\bf Case $n=-1$:}
Then $\Zeta_1 = 1$ and $y=2$ so the reducible characters are a line, 
\[\tilde{X}_{red}(\Gamma_{-1})= \{ (x,2,-x)\}.\]
The polynomials $\varphi_1$, $\varphi_2$, and $\varphi_3$ reduce to $x^2-1-y$, $zx-y-1$, and $-x+z$ so that the irreducible characters is a curve,
\[X(\Gamma_{-1}) = \{(x,x^2-1,x)\}.\]

{\bf Case $n=-2$:}
With no restrictions on $y$ in this case, the reducible characters are a surface,
\[ \tilde{X}_{red}(\Gamma_{-2})=\{ (x,y,z) : x^2+y^2+z^2-xyz=4\}.\]
The polynomials $\varphi_1$, $\varphi_2$, and $\varphi_3$ reduce to $x^2-y^2$, $zx-y^2$, and $-2x+yz$ so that the irreducible characters are two points and a line,
\[X(\Gamma_{-2}) = \{(0,0,z)\} \cup \{(2,2,2)\} \cup \{(-2,2,-2)\}.\]
These two points are contained in $\tilde{X}_{red}(\Gamma_{-2})$.
\end{proof}

\subsection{Points of high multiplicity in $X(\Gamma_n)$}\label{section:multiplicity}

In this section we begin our  inspection of  the irreducible components of  $X(\Gamma_n)$.

\begin{prop}\label{prop:multiplicitypoints1} 
The set $C$ is the union of $C'$   and the points
\[ P= \{ (\epsilon \sqrt{2}, y_0,  \epsilon y_0/\sqrt{2} ) : j_n(y_0)=0, \epsilon = \pm 1\}. \]
If $n=0$, then $P=C'$ is the union of two lines. If $0<|n|\leq 2$, then $P=\emptyset$. If $|n| >2$, then 
\begin{itemize}
\item $P$ is a finite set of $2\deg j_n\geq2$ points,
\item $P \subset C'$ and hence they occur in $C$ with multiplicity two,
\item no point in $P$ is the character of a reducible representation of $\Gamma_n$,  and
\item no point in $P$ is the character of a discrete faithful representation of $\Gamma_n$.
\end{itemize}
\end{prop}

\begin{proof}
From Definition~\ref{definition:phis}, since $\varphi_3=j_n(y) \varphi'_3$, $C$ is the union of the vanishing set $C'$ of $I'=(\varphi_1, \varphi_2, \varphi_3')$ and the vanishing set $P$ of the ideal $(\varphi_1, \varphi_2, j_n)$.
If $n=0$, then $\varphi'_3 = 0$ whenever both $\varphi_1=0$ and $\varphi_2=0$.  Since $j_0 = 0$ by Lemma~\ref{lem:constantpolys} it follows that $C'=P$ when $n=0$.
Since  $j_n$ is a non-zero constant if and only if $0<|n| \leq 2$ by Lemma~\ref{lem:constantpolys}, $P= \emptyset$ in these cases.

Now assume $|n| > 2$.  By Lemma~\ref{lemma:fibonacciidentities}  when $j_n(y)=0$ it follows that $f_n(y)=0$ and $f_{n+1}(y)=1$.  Then by the Fibonacci recursion,  $f_{n-1}(y) = -1$. Therefore, $\varphi_1=0$ implies that $x^2=2$ and $\varphi_2=0$ implies that $zx=y$.   Thus $P$ is the set of points with the cardinality as claimed.  
 
It is elementary to verify that if a point in $P$ satisfies the equation $x^2+y^2+z^2-xyz=4$, then $y=\pm 2$.  However, $\pm 2$ is a root of $j_n(y)$ only if $n=0$. 
  Therefore, since $|n|>2$, these points do not correspond to reducible representations by Proposition~\ref{prop:redreps}.

Finally, to show that these points are in $C'$, it suffices to show that the points 
\[ \{ (\epsilon \sqrt{2}, y_0,  \epsilon y_0/\sqrt{2} ) : j_n(y_0)=0, \epsilon = \pm 1\} \]  
satisfy $\varphi_1$, $\varphi_2$ and $\varphi_3'$.   A point of the form $(\epsilon \sqrt{2}, y_0,  \epsilon y_0/\sqrt{2} )$ satisfies $\varphi_1$ and $\varphi_2$ trivially.  As $h_n$ and $j_n$ share no common factors, by Lemma~\ref{lemma:nocommonfactors}~(\ref{lemma:jnhn}),  the vanishing set of  equation $xk_n(y_0)-z\ell_n(y_0)$  equals the vanishing set  of 
\[  h_n(y_0) \big( xk_n(y_0)-z\ell_n(y_0) \big) = x(f_n(y_0)-y_0)-z(f_{n-1}(y_0)-1). \]
As $y_0$ is a root of $j_n$, this equation reduces to $-xy_0+2z$.  When $x= \epsilon \sqrt{2}$ and $z=\epsilon y_0/\sqrt{2}$ this is $ -\epsilon \sqrt{2}y_0 +2 \epsilon y_0/\sqrt{2} = 0.$ Therefore, these points are on $C'$. By Lemma~\ref{lemma:disc1} these representations are not discrete and faithful.
\end{proof}

\begin{remark}

Proposition~\ref{prop:multiplicitypoints1} shows that for $|n|>0$ there is a finite non-empty set of points $P$ in $X(\Gamma_n)$ with multiplicity two.
Such points are invariant under isomorphism and are therefore well-defined. In particular, they do not depend on choices such as the presentation of the group $\Gamma_n$ or the simplification of matrices. 

It is unknown to the authors what significance such points have in the character variety.

\end{remark}

\begin{question}
What geometric significance, if any, do points with multiplicity in the $\SL_2(\C)$ character variety of a hyperbolic $3$--manifold carry?
\end{question}

\begin{prop}\label{prop:multiplicitypointsrep}
For $|n|>2$, the collection of points $P$ corresponds to representations of the group 
\[ \langle \alpha, \beta :   \beta \alpha^2 = \alpha^2 \bar{\beta}, \beta^{n}, \alpha^8 \rangle. \]
When $n=2m$ they are all representations of the quotient of this group by  the normal closure of the subgroup generated by $\beta^m\alpha^{4}$.
\end{prop}

\begin{proof}
As $x^2=2$ for points in $P$ and $x=a+a^{-1}$ we conclude that $a^4=-1$ so that $a=(\pm 1\pm i)/\sqrt{2}$.
Therefore, the matrix $A$ satisfies $A^4=-I$ and $A^8=I$, the identity.  As $y=b+b^{-1}$, $j_n(y)=0$, and $j_n$ is a divisor of $f_n$ by Lemma~\ref{lemma:fibonacciidentities}  it follows that  $b^{2n}=1$.  .  
 In fact, if $n=2m$ then $j_n(y)=f_m(y)$ by Lemma~\ref{definition:fibs}, and  we conclude that $b^m=1$ (since $y\neq -2$) and therefore $B^n=I$.  If $n=2m+1$ then 
\[ j_n(y)=f_{m+1}(y)+f_m(y) = \frac{b^{m+1}-b^{-m-1}+b^m-b^{-m}}{b-b^{-1}}.\]
As this is zero, $b^{2m+2}+b^{2m+1}-b-1=0$, and we conclude that $b^{2m+1}=1$. Therefore $B^n=I$ and $A^4=-B^n$ in this case as well.  

 We have 
\[ A= \mat{a}{0}{s}{a^{-1}} \rightarrow A^4= \mat{a^4}{0}{sf_4(x)}{a^{-4}}= - \mat{1}{0}{0}{1} \]
and 
\[ B= \mat{b}{s}{0}{b^{-1}} \rightarrow B^{n} = \mat{b^n}{sf_n(y)}{0}{b^{-n}} =   \mat{1}{0}{0}{1}. \]   
These are representations of the following group (as $\alpha^4$ is central)
\[ \langle \alpha, \beta :  \bar{\alpha} \beta \alpha^2 \beta \bar{\alpha}, \alpha^4\beta = \beta \alpha^4, \beta^{n}, \alpha^8 \rangle. \]
This presentation reduces to 
\[ \langle \alpha, \beta :   \beta \alpha^2 = \alpha^2 \bar{\beta}, \beta^{n}, \alpha^8 \rangle. \]
since $ \alpha^4\beta = \beta \alpha^4$ follows from $\beta \alpha^2 = \alpha^2 \bar{\beta}$.
If $n=2m$ is even, then $B^m=I$ and the additional relation $\beta^m=\alpha^4$ holds.  In this case, these are representations of the group
\[ \langle \alpha, \beta :   \beta \alpha^2 = \alpha^2 \bar{\beta},  \beta^m=\alpha^4, \beta^{2m}, \alpha^8 \rangle \]
which is a quotient of the previous group.
\end{proof}

When $n\equiv 2 \pmod 4$, we see that the affine line $(0,0,z)$ is contained in $C'$.  This follows since 
$f_{n-1}(0)=1$, and   $f_n(0)=0$.     We now collect information about this line of points.

\begin{prop}\label{prop:extraline}  
When $n\equiv 2 \pmod 4$ the line $(0,0,z)$ is contained in $C'$. 
None of the associated  characters is the character of a discrete faithful representation.  
These are characters of faithful representations of the group 
\[ \langle \alpha, \beta :  \alpha^2=\beta^2, \alpha^4 \rangle.\]
The line intersects the set of reducible characters in the points $(0,0,\pm 2)$.  No points of multiplicity two occur on this line. 
\end{prop}

\begin{proof}
Recall that the Fibonacci recursion implies that $f_{2m}(0)=0$ and $f_{2m+1}(0)=(-1)^m$. 
Therefore, since $n \equiv 2 \pmod 4$, we have $f_n(0) = 0$, $f_{n-1}(0)=1$, and $\ell_n(0) = 0$.  Thus $\varphi_1$, $\varphi_2$, and $\varphi_3'$ all vanish at each point $(0,0,z)$.

Consider a representation in standard form corresponding to $(0,0,z)$.
As $x=0$, we conclude that $a+a^{-1}=0$ so that $a=\pm i$.  Similarly, since $y=0$ we have $b=\pm i$. Hence these representations are (up to conjugation) determined by 
 \[ A =  \mat{ \epsilon_A i}{0}{s}{- \epsilon_A i}, \quad B= \mat{\epsilon_B i}{s}{ 0}{-\epsilon_B i}\]
 where $\epsilon_A, \epsilon_B \in \{ \pm 1\}$.
Thus $A^4=B^4=-I$.  Furthermore, one can verify that if $s\neq 0$ then $AB\neq BA$.
Therefore, if $s\neq 0$ these points correspond to representations of the group
\[ \langle \alpha, \beta : \alpha^2=\beta^2, \alpha^4 \rangle = \langle \alpha, \beta : \alpha^2=\beta^2, \beta^4 \rangle.\] 
As this group is not torsion-free, these are not characters of a discrete faithful representation.

Now we consider the intersection of this line and the set of reducible characters.  If $n\neq -2$, the defining equation for the reducible characters is $x^2+y^2+z^2-xyz=4$ with $y=\Rb_{n+2}$.  Therefore, as $y=0$ on the line, an intersection occurs as $n$ is even.  The equation reduces to $z^2=4$, so that $z=\pm 2$.   

As the points of multiplicity two have $x=\pm \sqrt{2}$, they are not on this line. 
\end{proof}


\subsection{A new model for $C'$}\label{section:newmodel}

In this section, we prove Theorem~\ref{thm:mainsummary1}.  By Proposition~\ref{prop:irreps} $X(\Gamma_n)$ is the set $C$, and  by Proposition~\ref{prop:multiplicitypoints1} the set $C$ is the union of $C'$ and the points $P$ (which are all contained in $C'$). These points $P$ were determined by the vanishing set of $\varphi_1$, $\varphi_2$, and $j_n(y)$.    

We begin with some notation. By Lemma~\ref{lemma:nocommonfactors}~(\ref{lemma:hnln}), $y$ divides both $h_n(y)$ and $\ell_n(y)$  exactly when $n\equiv 2 \pmod 4$ and does so with multiplicity one. 

\begin{definition}\label{definition:hats}
Define $\hat{\ell}_n(y)$ and $\hat{h}_n(y)$ so that $\hat{\ell}_n(y)=\ell_n(y)$ and $\hat{h}_n(y)=h_n(y)$ unless $n\equiv 2 \pmod 4$, in which case we define them so that $y\hat{\ell}_n(y)=\ell_n(y)$ and $y\hat{h}_n(y)=h_n(y)$. 
\end{definition}

We now explicitly determine the variety $X(\Gamma_n)$ in terms of the coordinates $x,y,z$.

\begin{definition}
Let $D$ be the curve  parametrized by 
\[ \Big\{ \Big(\epsilon \sqrt{1-f_{n-1}(y)}, y,  -\epsilon k_n(y)  \sqrt{ - \frac{\hat{h}_n(y)}{\hat{ \ell}_n(y)}} \ \Big) : \epsilon = \pm1, \hat{\ell}_n(y)\neq 0 \Big\}. \]
in $\A^3(x,y,z)$.
Let $L$ be the line  $\{(0,0,z)\}$ in $\A^3(x,y,z)$. 
\end{definition}

\begin{prop}\label{prop:xyzcharvar}
If $n \not \equiv 2 \pmod 4$ then $C'=D$. 
If $n \equiv 2 \pmod 4$ then  $C'=D \cup L.$
\end{prop}

\begin{proof}
By  Proposition~\ref{prop:irreps} we just need to calculate the common zeros of $\varphi_1, \varphi_2$, and  $\varphi_3$ from  Definition~\ref{definition:phis}.

By Lemma~\ref{lemma:fibonacciidentities} $f_{n-1}(y)-1=h_n(y)\ell_n(y)$, and so the zeros of the polynomial $f_{n-1}(y)-1$  are therefore the union of the zeros of $h_n(y)$ and the zeros of $\ell_n(y)$ as given in Lemma~\ref{lem:hellzeros}.  Hence $f_{n-1}(y)=1$ if and only if $y \in (\Rb^{fib}_{2n}-\Rb^{fib}_n) \cup \Rb^{fib}_{n-2}$.  These are disjoint unions except when $n \equiv 2 \pmod 4$ in which case $0$ is the only common element.

If $x\neq 0$ then $\varphi_1=0$ gives $x= \epsilon \sqrt{1-f_{n-1}(y)}$  where $\epsilon = \pm1$.  Then $\varphi_2 = 0$ gives $z = \epsilon (y-f_n(y))/ \sqrt{1-f_{n-1}(y)}$.  Upon substituting these in, $\varphi_3$ reduces to $0$ by Lemma~\ref{lemma:fibonacciidentities}, so in this case $y$ is only subject to the condition that $f_{n-1}(y) \neq 1$.

If $x = 0$, then $\varphi_1=0, \varphi_2=0, \varphi_3=0$ yield $1=f_{n-1}(y)$, $y=f_n(y)$, $0=zf_n(y)$ respectively.  The latter two imply $yz=0$.  If $y=0$ then $z$ is free; but since $1=f_{n-1}(0)$ and $0=f_n(0)$, Lemma~\ref{lemma:fkvals} implies that $n \equiv 2 \pmod 4$.  This gives the extra line $L$ in $C'$ when $n \equiv 2 \pmod 4$.  If on the other hand $z=0$ then the first two equations imply $f_{n-2}(y)=0$ due to the recursion relation; hence by Lemma~\ref{lemma:fibroots} we have $y \in \Rb_{2(n-2)}^{fib}$.  Since $1=f_{n-1}(y)$, $y$ must also be a zero of either $h_n(y)$ or $\ell_n(y)$ as given in Lemma~\ref{lem:hellzeros}.  We conclude that either $y=0$ as well and we are in the previous case or $y\in \Rb_{n-2}^{fib}$. Any point of the form $(0,y,0)$ where $y\in \Rb_{n-2}^{fib}$  satisfies $\varphi_1=\varphi_2=\varphi_3'=0$ and therefore is in $C'$.

 We have shown that when $n \not \equiv 2 \pmod 4$ the set  $C'$ is 
\[ \big\{ \big(\epsilon \sqrt{1-f_{n-1}(y)}, y,  \frac{y-f_n(y)}{\epsilon \sqrt{1-f_{n-1}(y)}} \big): \epsilon = \pm1, f_{n-1}(y) \neq 1\} \cup  \{(0,y,0) : y \in \Rb_{n-2}^{fib}\big\},\]
and when $n\equiv 2 \pmod 4$ then $C'$ is this set union $L$.

Since  Lemma~\ref{lemma:fibonacciidentities} gives $y-f_n(y) = -h_n(y)k_n(y)$ and $1-f_{n-1}(y)= -h_n(y)\ell_n(y)$, when $f_{n-1}(y)\neq 1$ we can write the third coordinate (the $z$ coordinate) as $- \epsilon k_n(y) \sqrt{- h_n(y)/ \ell_n(y)}$. which reduces to $- \epsilon k_n(y) \sqrt{- \hat{h}_n(y)/ \hat{\ell}_n(y)}$ using Definition~\ref{definition:hats}.  The parametrization above excludes the $y$ such that $f_{n-1}(y)=1$ while the parametrization for $D$ excludes only the $y$ satisfying $\hat{\ell}_n(y)=0$.  Therefore it suffices to show that the points $(0,y,0)$ with  $y\in \Rb_{n-2}^{fib}$ are in $D$, as these are precisely the $y$ values which are roots of $f_{n-1}(y)-1$ but not roots of $\hat{\ell}_n(y)$.   Since $f_{n-1}(y)-1= h_n(y)\ell_n(y)$ for all such $y\neq 0$, these $y$ are roots of $h_n(y)$ by Lemma~\ref{lemma:nocommonfactors}.  Therefore the point $(0,y,0)$ is in $D$.  Now consider $y=0$, which only occurs as such a point when $n\equiv 2 \pmod 4$.  The point $(0,0,0)$ is not in $D$, but is in the line $L$.  Thus we have shown that $C'=D$ when $n\not \equiv 2 \pmod 4$ and $C'=D\cup L$ when $n\equiv 2 \pmod 4$.
\end{proof}

\begin{remark}\label{remark:differentparam}
The proof of Proposition~\ref{prop:xyzcharvar} shows we can also write $D$ as
\[ \big\{ \big(\epsilon \sqrt{1-f_{n-1}(y)}, y, \epsilon \frac{y-f_n(y)}{ \sqrt{1-f_{n-1}(y)}} \big): \epsilon = \pm1, f_{n-1}(y) \neq 1\} \cup  \{(0,y,0) : y \in \Rb_{n-2}^{fib}\big\}\]
\end{remark}

 By Proposition~\ref{prop:multiplicitypoints1} and Proposition~\ref{prop:xyzcharvar}, up to multiplicity the variety $X(\Gamma_n) = C$  is given by $D$ when $n\not \equiv 2 \pmod 4$, and by $D \cup L$ when $n\equiv 2\pmod 4$.  By Lemma~\ref{lemma:disc1}  the line $L$ (and the set $ \{(0,y,0) : y \in \Rb_{n-2}^{fib}\}$) does not contain the character of a discrete and faithful representation of $\Gamma_n$.  We will show that $D$ is irreducible and birational to the curve given in Theorem~\ref{thm:mainsummary1}. It will follow that $D$ is the unique canonical component $X_0(\Gamma_n)$.

\begin{lemma}\label{lemma:DE}
The curve $D$ is birational to the curve $E$ given by coordinate ring $\A^2[x,y]/(x^2 -1+f_{n-1}(y))$. 
\end{lemma}

\begin{proof}
Let $\varphi\colon D\rightarrow \A^2(x,y)$ be the projection map $(x,y,z) \mapsto (x,y)$.   The image $\varphi(D)$ is given by the set \[ \Big\{ \big( \epsilon \sqrt{1-f_{n-1}(y)}, y \big):  \epsilon = \pm1, \hat{\ell}_n(y)\neq 0  \Big\}.\] 
This is the set of solutions to
\[ x^2 = 1-f_{n-1}(y).\]
except for the points $(0,y)$ where $\hat{\ell}_n(y)=0$.  Since these $y$ for which $\hat{\ell}_n(y)=0$ form the finite set $(\Rb_{2n}^{fib}-\Rb_n^{fib})-\{0\}$, the map $\varphi$ from  $D$ to $E$ is rational.

On $E$, the map $\varphi$ has inverse mapping $\varphi^{-1} \colon \A^2(x,y) \rightarrow \A^3(x,y,z)$  defined by \[ \varphi^{-1}(x,y) = \Big(x,y, -\epsilon k_n(y)  \sqrt{ - \hat{h}_n(y)/ \hat{\ell}_n(y) }\Big).\]
This is rational since $\hat{\ell}_n(y)=0$ for only finitely many points on $E$, the points $(0,y)$ where $\hat{\ell}_n(y)=0$. Therefore $E$ and $D$ are birational.
\end{proof}

\begin{prop}\label{prop:extraline3}
Assume  $n \equiv 2 \pmod 4$.
The line $L$  intersects $D$ in the two points $(0,0,z_0)$ and $(0,0,-z_0)$ where $z_0 = 2\sqrt{\tfrac12-\tfrac1n}$.   
\end{prop}

\begin{proof}
We use the parametrization from Proposition~\ref{prop:xyzcharvar}.  When $x=y=0$ the set $D$ is given by 
\[ \Big\{ \Big(\epsilon \sqrt{1-f_{n-1}(0)}, 0,  -\epsilon k_n(0)  \sqrt{ - \frac{\hat{h}_n(0)}{ \hat{\ell}_n(0)}} \ \Big) : \epsilon = \pm1 \Big\}. \]

  (When $y=0$, the defining equation for $E$, $x^2-1+f_{n-1}(y)$, reduces to  $x^2$  and  thus determines that $x=0$.)
The line $x=y=0$  is the image of the line  discussed in Proposition~\ref{prop:extraline}.    Using the Fibonacci recursion,  with $\hat{f}_{2l}(y)=f_{2l}(y)/y$ it is elementary to show that $\hat{f}_{2l}(0)=(-1)^{l-1}l$. Also, $k_n(0)=\pm 2$ when $n$ is even.  Therefore, if $n=2m$ we have 
\[ z=- \epsilon k_n(0)  \sqrt{ - \frac{\hat{h}_n(0)}{ \hat{\ell}_n(0)}} =\pm 2 \sqrt{-\frac{\hat{f}_{m-1}(0)}{\hat{f}_{m+1}(0)-\hat{f}_{m-1}(0)}} = \pm 2 \sqrt{\frac{m-1}{2m}} \]
so that $z^2= 4(m-1)/2m=2(n-2)/n = 4(\tfrac12-\tfrac1n)$.
\end{proof}

Now it suffices to consider the algebraic set $E$.

\begin{definition} 
Let $F$ be the algebraic set with coordinate ring  $\A^2[w,y]/ (  w^2+\hat{h}_n(y) \hat{\ell}_n(y)  )$. 
\end{definition}

\begin{lemma}\label{lemma:E'birationaltoF}  
The set $E$ is birationally equivalent to the set $F$.
\end{lemma}

\begin{proof}
The set $E$ is given by  \[ x^2=1-f_{n-1}(y).\]
These sets are identical when $n\not \equiv 2 \pmod 4$.  When $n\equiv 2 \pmod 4$, by Lemma~\ref{lemma:nocommonfactors}  $y^2$ divides $1-f_{n-1}(y)$.  We define the map $\varphi: E\rightarrow F$ by $\varphi(x,y) =(x/y,y).$ This is rational since on $E$ $y=0$ only when $x=0$.  The map has inverse given by $(w,y) \mapsto (wy,y)$.

\end{proof}

\begin{prop}\label{prop:Gsmoothirreducible} 
The affine variety $F$  is smooth and irreducible. 
\end{prop}

\begin{proof}
In $\mathbb{P}^1\times \mathbb{P}^1$ a variety of bidegree $(a,b)$ with $a,b>0$ is irreducible if it is smooth (see \cite{MR2827003} Lemma 2.6) and has genus $(a-1)(b-1)$. We think of $F$ as the affine portion of its projective closure.  Furthermore, smoothness is equivalent to showing that there are no simultaneously vanishing partials. As the defining equation is $w^2=-\hat{\ell}_n(y) \hat{h}_n(y)$, smoothness is equivalent to showing that $\hat{\ell}_n(y)\hat{h}_n(y)$ has no repeated roots. (These polynomials have positive degree for all $|n|>2$.)  By Lemma~ \ref{lemma:nocommonfactors}~(\ref{lemma:hnln}) the functions $\hat{\ell}_n(y)$ and $\hat{h}_n(y)$ have no common factors.   Therefore, the claim follows from the separability of $\hat{\ell}_n$ and $\hat{h}_n$ determined in  Lemma~\ref{lemma:fisseparable}.
\end{proof}

\begin{remark}  Although $F$ is smooth as an affine variety, it is not smooth at infinity.   If $\deg(f)>4$ then $y^2=f(x)$ is singular at infinity.  One can form a smooth model in a natural way in weighted projective space.   That is, if $f$ is even then we can homogenize the equation by adding the variable $w$ and then define a projective space as the quotient of $\A^3(w,x,y)$ by implementing the equivalence relation $(\lambda^{d_1}w, \lambda^{d_2}x, \lambda^{d_3}y)$ is equivalent to $(w,x,y)$.  In this case, the degrees are $d_1=d_2=1$ and $d_3 =\tfrac12 \deg(f)$. This model is smooth. 
\end{remark}

\begin{prop} 
The genus of $F$ is $\lfloor \tfrac12 (|n-1|-2)\rfloor $ if $n\not \equiv 2 \pmod 4$ and is  $\ \lfloor \tfrac12 (|n-1|-4)\rfloor$ if $n\equiv 2 \pmod 4$. 
\end{prop}

\begin{proof}
The smooth hyperelliptic curve given by $v^2=f(u)$ has genus $ \lfloor \tfrac12(d-1) \rfloor$ where $d$ is the degree of $f$. 
Since $F$ is the smooth hyperelliptic
  curve given by $w^2=-\hat{\ell}_n(y) \hat{h}_n(y)$, the result follows from a calculation of the degree of $-\hat{\ell}_n(y) \hat{h}_n(y)$:  By Definition~\ref{definition:hats} and Lemma~\ref{lemma:fibonacciidentities}, $-\hat{\ell}_n(y) \hat{h}_n(y) = 1-f_{n-1}(y)$ when $n \not \equiv 2 \pmod 4$ and  $-y^2\hat{\ell}_n(y) \hat{h}_n(y) = 1-f_{n-1}(y)$ when $n \equiv 2 \pmod 4$.   Hence, by Lemma~\ref{lemma:fibonaccievenodd}, the degree of $-\hat{\ell}_n(y) \hat{h}_n(y)$ is $|n-1|-1$ when $n\not\equiv 2 \pmod4$ and it is $|n-1|-3$ when $n\equiv 2 \pmod4$.
\end{proof}

By Lemma~\ref{lemma:disc1}, as stated in the beginning of this section, $X_0(\Gamma_n)$ is not $L$.  Therefore, it is $D$.   By Lemma~\ref{lemma:DE} $D$ is birational to $E$ and by Lemma~\ref{lemma:E'birationaltoF} $E$ is birational to $F$.  Finally, by Proposition~\ref{prop:Gsmoothirreducible} the set $F$ is smooth and irreducible.  We conclude that $X_0(\Gamma_n)$ is birational to $F$.  We have shown  Theorem~\ref{thm:mainsummary1}.

%
%
\section{Discrete Faithful Representations and the Trace Field}\label{section:tracefield}
%
%

In this section we assume that $|n|>2$ and $\rho_0$ is a discrete faithful representation in standard form.  Since such a representation is irreducible Lemma~\ref{lemma:genericreps} implies we may take $\rho_0(\alpha) = A = A(a,s)$ and $\rho_0(\beta)=B=B(b,s)$. We set $x=\tr(A)=a+a^{-1}$, $y=\tr(B)=b+b^{-1}$ and $z=\tr(BA)=ab+a^{-1}b^{-1}+s^2$ as before. Recall that for $\gamma\in \Gamma_n$, $\chi_{\rho_0}(\gamma)=\tr (\rho_0(\gamma))$.

There are two discrete faithful representations of $\Gamma_n$ into $\PSL_2(\C)$, and these are (entry-wise) complex conjugates of one another.  If a representation from $\Gamma_n$ to $\PSL_2(\C)$ lifts to an $\SL_2(\C)$  representation then there are $|H^1(M_n, \Z/2\Z)|$  such lifts constructed as follows.  Let $I$ be the $2\times 2$ identity matrix, and identify $\Z/2\Z$ with $\{ \pm I\}$.  An element $\epsilon\in H^1(M_n,\Z/2\Z)$ corresponds to  a map $\epsilon:\pi_1(M_2) \rightarrow \{\pm I\}$.  If $\rho_0:\pi_1(M_n) \rightarrow \SL_2(\C)$ is a lift of $\rho_0':\pi_1(M_n) \rightarrow \PSL_2(\C)$ then another lift of $\rho_0'$ is $\epsilon \circ \rho_0$. That is, for all $\gamma \in \Gamma_n$ the representation is $(\epsilon\circ\rho_0)(\gamma) = \epsilon(\gamma) \rho_0(\gamma)$.  (See \cite{MR1695208}.)

By Lemma~\ref{lem:HomGamma}  if $n$ is even $|H^1(M_n,\Z/2\Z)|=4$ and if $n$ is odd then $|H^1(M_n,\Z/2\Z)|=2$.  Since $\Gamma_n$ is generated by $\alpha$ and $\beta$ the mapping $\epsilon$ is determined by $\epsilon(\alpha)$ and $\epsilon(\beta).$  When $n$ is even all four possibilities occur.  When $n$ is odd we have only  the identity and $\epsilon_1 $ defined by  $\epsilon_1(\alpha)=-1$, $\epsilon_1(\beta)=1$.

By Lemma~\ref{lemma:boundarypi1} the peripheral subgroup of $\pi_1(M) \cong \Gamma_n$ is generated by the ``nice'' meridian ${\mu}=\beta \alpha$ and the longitude $\lambda=\alpha \beta \bar{\alpha} \beta \alpha \bar{\beta}\bar{\alpha} \bar{\beta}$.   Since a discrete faithful representation $\rho_0$ must send the peripheral subgroup to parabolics, both $\chi_{\rho_0}(\bar{\mu})$ and $\chi_{\rho_0}(\lambda)$ are $\pm 2$.  

From the discussion above, there are $n$ for which both positive and negative signs occur.
 More precisely, $\epsilon(\mu)=\epsilon(\beta \alpha)$ so  when $n$ is odd  the non-trivial element   $\epsilon_1\in H^1(M_n,\Z/2\Z)$ acts on $\mu$ as $\epsilon_1(\mu)=-1$.  Therefore $\epsilon_1$ acts on $\chi_{\rho_0}(\mu)\in X (M_n)$ by $\epsilon_1 (\chi_{\rho_0}(\mu))=-\chi_{\rho_0}(\mu)$.  Therefore, for a single discrete faithful character in $Y(M_n)$  there  are lifts $\rho_0$ and $\epsilon_1\circ \rho_0$ to $X(M_n)$ such that $\chi_{\rho_0}(\mu)=2$ and  $\chi_{\epsilon_1 \circ \rho_0}(\mu)=-2$.  Similarly, $\chi_{\rho_0}(\beta)=-\chi_{\epsilon_1\circ \rho_0}(\beta)$.

 When $n$ is even there are two elements of $H^1(M_n,\Z/2\Z)$ whose action on $X(M_n)$ sends $\chi_{\rho_0}(\mu)$ to $-\chi_{\rho_0}(\mu)$.  (These are the elements which act non-trivially on exactly one of $\alpha$ and $\beta$.)  It follows that for  a single discrete faithful character $\rho_0'$  in $Y(M_n)$  there  are two lifts $\rho_1$ and $\rho_2$ of $\rho'$ such that  $\chi_{\rho_1}(\mu)=\chi_{\rho_2}(\mu)=2$ and two lifts $\rho_3$ and $\rho_4$ of $\rho_0'$ such that  $\chi_{\rho_3}(\mu)=\chi_{\rho_4}(\mu)=-2$. Similarly, $\chi_{\rho_1}(\beta)=\chi_{\rho_2}(\beta)=-\chi_{\rho_3}(\beta)=-\chi_{\rho_4}(\beta)$.

For any $n$, the exponent sum of both $\alpha$ and $\beta$ in $\lambda$ is even.  Therefore for any $\epsilon \in H^1(M_n,\Z/2\Z)$, $\epsilon(\lambda) = \lambda$.  Below, in Lemma~\ref{lemma:lambdaminus2}, we will show that $\chi_{\rho_0}(\lambda)=-2$ for a discrete faithful character $\rho_0$.

Now, we determine the matrices $T=\rho_0(\bar{\mu})$ and $D=\rho_0(\lambda)$ explicitly.  We compute
\[\rho_0(\mu) =  \mat{ba+s^2 }{s a^{-1}}{s b^{-1} }{b^{-1}a^{-1}} \quad \mbox{ and } \quad
\rho_0(\lambda) = D = \mat{D_{11} }{D_{12}}{D_{21}}{ D_{22}}  \]
 where
 \begin{align*}
D_{11} & =1 + \big( a^{-1}b^{-1}-  a^{-1} b+ a b^{-3} - 2 a b^{-1} + 
         b a \big) s^2 +  \big( 2 - b^{-2} -  a^2 + a^2b^{-2} \big) s^4      \\
         & \quad - a b^{-1}s^6   \\
D_{12} & = \big( -b + b^3 + a^2 b- 
          a^2b^3  \big) s +  \big( - a^{-1} +   a^{-1} b^2+  2 a - a b^{-2} - 3  ab^2 - a^3 +    a^3 b^2  \big) s^3 \\
         & \quad + \big( b^{-1} - 2           b - a^2 b^{-1} + 2 a^2b   \big) s^5  +        a s^7  \\
D_{21} & =\big( a^{-1} +  a^{-1}b^{-4} - 2  a^{-1} b^{-2}\big) s + \big( 2 b^{-3} - 3 b^{-1} + 
          b -  a^{-2} b^{-3} + 2  a^{-2}b^{-1} - a^{-2}b  \big) s^3  \\
          & \quad + \big( 2 a^{-1} - 2  a^{-1} b^{-2}- 
          a + a b^{-2} \big) s^2 t^3 - b^{-1}s^7  \\
D_{22} & =     1 + \big(-  a^{-1} b^{-3}+ 2  a^{-1}b^{-1} - 2a^{-1}b + a^{-1} b^3 
- a b^{-1} + 2  ab -   a  b^3  \big) s^2 \\
          &\quad + \big(4 - 2 b^{-2} - 
          3 b^2 - 2 a^{-2} + a^{-2} b^{-2} +a^{-2} b^2 - a^2 + b^2
           a^2 \big) s^4  \\
           & \quad + \big ( 2 a^{-1}b^{-1} - 2a^{-1} b  - 
 	a b^{-1} + 2 ab \big) s^3 t^3 + s^8. 
\end{align*}

In the remainder of the section,  will use $\epsilon=\pm1$.

\begin{lemma}\label{lemma:lambdaminus2}
If $\rho_0$ is a discrete faithful representation then  $\chi_{\rho_0}(\lambda)=-2$.
\end{lemma}

\begin{proof}
Computing these traces we obtain
\[ \chi_{\rho_0}(\mu)  = ab+a^{-1}b^{-1}+s^2 = z\]
and 
\begin{multline}\label{equation:lambda}
\chi_{\rho_0}(\lambda) = z^4-2xyz^3+(x^2y^2+y^2+2x^2-4)z^2\\+(-2x^3y-y^3x+4xy)z +x^4+x^2y^2-4x^2+2.
\end{multline}
Substitute $\chi_{\rho_0}(\mu) = z=2\epsilon$ where $\epsilon = \pm1$ into the  $F_{21}=0$ equation.  This is equivalent to 
\begin{equation}\label{equation:F21}  (a-\epsilon b)(ab-\epsilon)(-\epsilon(a+a^3)+(1+4a^2+a^4)b-\epsilon(a+a^3)b^2) =0.\end{equation}

Assuming $a \neq \epsilon b^{\pm1}$, this further implies $x^2-\epsilon xy + 2 =0$.
Then together $z=2\epsilon$ and $x^2-\epsilon xy + 2 =0$ imply $ \chi_{\rho_0}(\lambda)=-2$.
On the other hand, if $a = \epsilon b^{\pm1}$ then $x = \epsilon y$.  Since  $z=2\epsilon$ and $x=\epsilon y$ together with $\varphi_2 =0$ (from Definition~\ref{definition:phis}) imply $y=-f_n(y)$. If $y\neq0$ then $\varphi_3=0$ and the recursion further implies  $f_{n+2}(y) =0$.  If $y=0$ then $f_{2m}(y)=0$.  In either case, Lemma~\ref{lemma:disc1} implies $\rho_0$ is not discrete and faithful, a contradiction.
\end{proof}

\begin{prop}\label{prop:disc2}
If $\rho_0$ is a discrete faithful representation then 
\[ 2+x^2-\epsilon xy=0 \]
Moreover, $ x^2 = 1-f_{n-1}(y)$ and $ 2\epsilon x = y-f_n(y)$.
\end{prop}

\begin{proof} 
Since $\rho_0$ is discrete faithful, it is irreducible so that $s\neq 0$. With $z=2\epsilon$, $\varphi_1 = 0$ and $\varphi_2=0$ of Definition~\ref{definition:phis} yield $f_{n-1}(y)= 1- x^2 $ and $ f_n(y) = y-2\epsilon x$.  The recursion implies $f_{n+1}(y) = x^2 + y^2 = 2 \epsilon xy -1$.  Substituting with these,  $\varphi_3$ factors as $(x-\epsilon y)(2+x^2-\epsilon xy)$.  So either $x=\epsilon y$ or $2+x^2-\epsilon xy=0$ as desired.

If $x=\epsilon y$, then with the substitution $z=2\epsilon$, Equation~\eqref{equation:lambda} simplifies to gives $\chi_{\rho_0}(\lambda)=2$, contrary to Lemma~\ref{lemma:lambdaminus2}.
\end{proof}

\begin{definition} \label{defn:phat}
Let $p_n(y) = f_{n+1}(y)-f_{n-1}(y)-y^2+6$.  
By Lemma~\ref{lemma:fkvals}, $y=-2$ is a root of $p_n(y)$ when $n$ is odd but not when $n$ is even.

If $n$ is even, let $\hat{p}_n(y) = p_n(y)$.   If $n$ is odd, let $\hat{p}_n(y)$ be the polynomial such that $p_n(y) =(y+2)\hat{p}_n(y)$. 
\end{definition}

\begin{prop}\label{prop:discfaithfulyeqn} 
For a discrete faithful representation of $\Gamma_n$,
\begin{enumerate}
\item $\hat{p}_n(y) = 0$, 
\item $f_n(y) = \pm \sqrt{y^2-8}$, and
\item $(x,y,z) = ( \tfrac{1}2 \epsilon(y\mp \sqrt{y^2-8}),y,2\epsilon)$.
\end{enumerate}
\end{prop}

\begin{proof} 
For the first, we use the equations from Proposition~\ref{prop:disc2}.
Upon multiplying the  equation $x^2-\epsilon x y+2=0$ by $2$ and substituting $1-f_{n-1}(y)$ for $x^2$ and $y-f_n(y)$ for $2\epsilon x$ we have 
$6-2f_{n-1}(y)-y^2+yf_n(y)=0$. The Fibonacci identity makes this $f_{n+1}(y)-f_{n-1}(y)-y^2+6=0$, i.e.\ $p_n(y)=0$. Since Lemma~\ref{lemma:nonperipheralgenerators} shows that $\beta\in \Gamma_n$ is a hyperbolic element, $y = \chi_{\rho}(\beta) \neq \pm2$. By Lemma~\ref{lemma:fkvals} $p_n(2) \neq 0$ but $p_n(-2)=0$ when $n$ is odd.  Thus $y$ must be a root of $\hat{p}_n(y)$.

Since $\hat{p}_n(y)=0$, we may simplify the expression of $f_n(y)$.  Using the substitution $y=s+s^{-1}$ we have $p_n(s+s^{-1}) = s^n-s^2+4-s^{-2}+s^{-n}$.  Therefore, setting $Q=s^2-4+s^{-2}$, 
\begin{align*}
 s^n p_n(s+s^{-1}) &= s^{2n}-Qs^n+1\\
  &=(s^n - \tfrac{1}{2} (Q+\sqrt{Q^2-4}))(s^n - \tfrac{1}{2} (Q-\sqrt{Q^2-4})).
  \end{align*}
 Assuming $y=s+s^{-1}$ is a root of $p_n(y)$ other than $\pm2$, then $s^n p_n(s+s^{-1}) =0$.   Hence $s^n = \tfrac{1}{2} (Q\pm\sqrt{Q^2-4})$.  Then we may simplify $s^n-s^{-n} = \pm \sqrt{Q^2-4}$ so that 
 \[f_n(y) = \frac{s^n-s^{-n}}{s-s^{-1}} = \pm \frac{\sqrt{Q^2-4}}{s-s^{-1}} = \pm \sqrt{s^2-6+s^{-2}}.\]
  Then observe that $s^2-6+s^{-2}=(s+s^{-1})^2-8 = y^2-8$ gives the second result.
  
 Finally, since $x=\epsilon (y-f_n(y))/2$ by Proposition~\ref{prop:disc2} and $z= \chi_{\rho} (\bar{\mu})=2\epsilon$, the third result follows.
\end{proof}

\begin{remark}
Because $f_n(y) = -f_{-n}(y)$, the roots of $f_{n+1}(y)-f_{n-1}(y)-y^2+6$, and hence the roots of $\hat{p}_n(y)$, only depend on $|n|$.
\end{remark}

\begin{prop}\label{prop:tracefield} 
For $|n|>2$ the trace field of $M_n$  is contained in $\Q(y_0)$ where $y_0=\tr(\rho_0(\beta))$ for the discrete faithful representation. This $y_0$ is a root of $\hat{p}_n(y)$.   
The degree of  $\hat{p_n}(y)$ is   $|n|-e$ where $e=0$ if $n$ is even and $e=1$ if $n$ is odd.
\end{prop}

\begin{proof}
As $2\epsilon x=y-f_n(y)$ and $z=2\epsilon$ we conclude that the trace field is $\Q(y_0)$ where $y_0$ is the $y$--value corresponding to  a discrete faithful representation.  From Proposition~\ref{prop:discfaithfulyeqn} we see that such a $y_0$ is a root of $\hat{p}_n(y)$. By Lemma~\ref{lemma:fibonaccievenodd} the degree of $f_k$ is $|k|-1$.  
\end{proof}

\begin{thm}\label{thm:tracefield}
When $|n|>2$ the polynomial $\hat{p}_n(y)$ is irreducible over $\Q$ and is the minimal polynomial for the trace field of $M_n$. The degree of the trace field is $|n|-e$ where $e=0$ if $n$ is even and $e=1$ if $n$ is odd.
\end{thm}

\begin{proof}
Using the substitution $y=b+b^{-1}$ after clearing denominators (of $b^n$),  $ p_n(y)$ 
is the  polynomial  \[ b^{2n}-b^{n+2}+4b^n-b^{n-2}+1.\]   

If $n=2m$ is even the polynomial factors as 
\[ (b^{2m}-b^{m+1}+b^{m-1}+1)(b^{2m}+b^{m+1}-b^{m-1}+1).\]  
This factorization mirrors the symmetry of the substitution $y=b+b^{-1}$.  The polynomials $b^{2m}-b^{m+1}+b^{m-1}+1$ and $b^{2m}+b^{m+1}-b^{m-1}+1$ both have the property that if $\omega$ is a root, then so is $-1/\omega$.  Let 
$q_m(b)= b^{2m}-b^{m+1}+b^{m-1}+1$.  To show that $p_{2m}$ is irreducible, it suffices to show that $q_m$ is irreducible over $\Q$.

If $n$ is odd, then $(b+1)^2$ is a factor corresponding to the $y+2$ factor of $p_n(y)$. 
Using the substitution $b=c^2$  and the fact that $n$ is odd, we can write 
\[ b^{2n}-b^{n+2}+4b^n-b^{n-2}+1= (c^{2n}-c^{n+2}+c^{n-2}+1)(c^{2n}+c^{n+2}-c^{n-2}+1).\] 
Let $r_n(x)=x^{2n}-x^{n+2}+x^{n-2}+1$. As in the even case, if $\omega$ is a root of $r_n(x)$ then so is $-1/\omega$.
The factor $(b+1)^2$ corresponds to a factor of $c^2+1$ in both $r_n(c)$ and in $r_n(-c)$.  To show that $\hat{p}_n$ is irreducible it suffices to show that $r_n(x)$ is $x^2+1$ times a polynomial which is  irreducible over $\Q$.

The irreducibility of each of these two polynomials is determined in Lemma~\ref{lemma:tracepolyirred} below.
\end{proof}

\begin{lemma}\label{lemma:tracepolyirred}
The polynomial $q_m(x)$ is irreducible. When $n$ is odd the polynomial $r_n(x)$ is the product of $x^2+1$ times an irreducible polynomial. 
 \end{lemma}

\begin{proof} 
The proof of irreducibility of $q_m(x)$ is due to Farshid Hajir.  The other case is based on this idea as well. 

Consider a polynomial of the form $f(x) = x^{k_1}+\epsilon_1x^{k_2}+\epsilon_2x^{k_3}+\epsilon_3$ where $\epsilon_i=\pm1$ for $i=1,2,3$.   By a theorem of Ljunggrem \cite{MR0124313} if $f$ has no zeros which are roots of unity then $f$ is irreducible over $\Q$.  Further, if $f$ has $\ell$ roots of unity as roots then $f$ can be decomposed into two factors, one of degree $\ell$ which has these roots of unity as zeros and the other which is irreducible over $\Q$.

For the first assertion  it suffices to show that no root of unity is a root of $q_m$.  Suppose that $\omega$ is a root of unity such that  $q_m(\omega)=0$. Then
\[ \omega^{2m}-\omega^{m+1}+\omega^{m-1}+1=0.\]
Dividing the equation by $\omega^m$ we see that 
 \[ \omega^m+\omega^{-m} =\omega-\omega^{-1}.\]
Since $\omega$ is a root of unity, the left hand side is real, but the right hand side is $2\text{Im}(\omega) i$. We conclude that $\text{Im}(\omega)=0$, so that  $\omega$ is real.  Therefore, $\omega=\pm1$.  This provides a contradiction, as $q_m(\pm 1) \neq 0$.

For the second assertion it suffices to show that $\pm i$ are the only roots of unity which are roots of $r_n$  and that these occur with multiplicity one.   Let $\omega$ be a root of unity   such that  $r_n(\omega)=0$.  Then 
\[\omega^{2n}-\omega^{n+2}+\omega^{n-2}+1=0.\]
Dividing the equation by $\omega^{n}$ we see that
\[ \omega^n+\omega^{-n} = \omega^2-\omega^{-2}.\]
Since $\omega$ is a  root of unity, the left hand side is  real, but the right hand side is $ 2\text{Im}(\omega^2)i$. We conclude that  $\omega^2=\pm 1$.  Therefore $\omega\in \{\pm 1, \pm i\}$.  It is easy to see that $\pm 1$ is not a root of $r_n$ while $\pm i$ are roots of $r_n$.  It is elementary to verify that $r_n'(\pm i)\neq 0$, so $\pm i$ are roots of multiplicity one. 
\end{proof}

%
%
\section{Computation of the $\PSL_2(\C)$ Character Variety}\label{section:PSLcalc}
%
%

In this section, we compute the $\PSL_2(\C)$ character variety of $\Gamma_n$.   We refer the reader to Section~\ref{section:PSL} for the construction of the $\PSL_2(\C)$ character variety.

First, we note that
if the degree of a separable polynomial $f(x)$ is greater than four, then the equation $y^2=f(x)$ determines a hyperelliptic curve.  The elliptic curves are those where the degree of $f$ is 3 or 4, and when the degree is less than three the equation defines lines or a conic.  For ease, will will call all such curves hyperelliptic.   Each of these curves have an involution given by $(x,y) \mapsto (x,-y)$ called the {\em hyperelliptic involution}.  It is well-known (in fact, it is a defining feature of hyperelliptic curves) that the quotient by this involution is birational to  $\A^1$.

\subsection{Odd $n$}
As stated in Section~\ref{section:PSL}, $Y(\Gamma_n)=X(\Gamma_n)/\mu_2$  where the action of $\mu_2$ is determined by $(\chi_{\rho}(\alpha), \chi_{\rho}(\beta), \chi_{\rho}(\alpha \beta)) \mapsto (-\chi_{\rho}(\alpha), \chi_{\rho}(\beta), -\chi_{\rho}(\alpha \beta))$.

\begin{thm}\label{thm:pslasinvolution} Assume $|n|>2$ is odd. 
The identification $Y(\Gamma_n) = X(\Gamma_n)/\mu_2$ corresponds to the identification of $F$ under the hyperelliptic involution $(y,w) \mapsto (y,-w)$.  The algebraic set $Y(\Gamma_n)=Y_0(\Gamma_n)$ and is birational to an affine line. 
\end{thm}

\begin{proof}
In the natural model, $C$,  for $X(\Gamma_n)$ the hyperelliptic involution of $F$  is readily seen to correspond to $(x,y,z) \mapsto (-x,y,-z)$.  Identification by this map is precisely identification by the action of $\mu_2$.  The quotient of a hyperelliptic curve by this involution yields an $\A^1$.
Since $n$ is odd, $X_0(\Gamma_n) = X(\Gamma_n)$ by Theorem~\ref{thm:mainsummary1} and hence $Y_0(\Gamma_n) = Y(\Gamma_n)$.
\end{proof}

\subsection{Even $n$}

When $|n|>2$ is even,  the $\PSL_2(\C)$ character variety is determined by the $\chi_{\rho}(\gamma)$ for $\gamma \in \Gamma$ with even exponent sum in both $\alpha$ and $\beta$. Therefore, it is determined by the variables $\bar{x}=x^2$, $\bar{y}=y^2$ and $\bar{z}=z^2$ where $x=\chi_{\rho}(\alpha)$, $y=\chi_{\rho}(\beta)$ and $z=\chi_{\rho}(\alpha \beta)$ as before.

We begin by considering those representations that lift to $\SL_2(\C)$.

\begin{lemma}\label{lemma:psleven1}
 Assume $|n|>2$ is even. The quotient $Y_0(\Gamma_n)= X_0(\Gamma_n)/\mu_2 $ corresponds to the identification of $F$ under the hyperelliptic involution $(y,w)\mapsto (y,-w)$ and the involution $(y,w) \mapsto (-y,w)$.  The set $Y_0(\Gamma_n)$ is birational to an $\A^1$.  When $n\equiv 2 \pmod 4$ the identification on the additional line component $L$ is given by an involution. 
\end{lemma}

\begin{proof}

The action of $H^1(\Gamma_n; \Z/2\Z)$ is generated by  $(x,y,z)$ by $\sigma_1 (x,y,z) = (-x,y,-z)$ and $\sigma_2(x,y,z)=(x,-y,-z)$.  These correspond to negating $\rho(\alpha)$ and negating $\rho(\beta)$, respectively.  The hyperelliptic involution corresponds to $\sigma_1$, and $\sigma_2$ descends to the involution of $y$ on both the canonical component and the line $(0,0,z)$.   This descends to the involutions generated by $(w,y) \mapsto (-w,y)$ and $(w,y)\mapsto (w,-y)$ on  $F$ which is given by the model  $w^2=  -\hat{h}_n(y)\hat{\ell}_n(y).$  Since the polynomial $-\hat{h}_n(y)\hat{\ell}_n(y)$ is even, we may write it as $p(y^2)$.  Then the action on $w^2=p(y^2)$ gives a quotient $\bar{w}=p(\bar{y})$,  where $\bar{w}=w^2$ and $\bar{y} = y^2$. This is birational to $\A^1$. 
\end{proof}

Since $H^2(\Gamma_n; \Z/2\Z)$ is non-trivial there are representations into $\PSL_2(\C)$ which do not lift to representations of $\SL_2(\C)$.  These are precisely those representations into $\PSL_2(\C)$ which extend to representations  $\rho$ to $\SL_2(\C)$ with the property that $\rho(\beta^{-n})=-\rho(\alpha^{-1}\beta \alpha^2\beta \alpha^{-1})$.  Let $\rho(\alpha) = \epsilon_A A$ and $\rho(\beta)=\epsilon_B B$ where $A$ and $B$ are chosen as in Definition~\ref{defn:genericexceptional} and $\epsilon_A, \epsilon_B =\pm 1$.  The set $Y(\Gamma_n)$ is determined by the solutions to $B^{-n}=\epsilon A^{-1}BA^2BA^{-1}$ where $\epsilon=\pm1$.  We now consider these representations that do not lift.

  It was shown in Proposition~\ref{prop:irreps} that $X(\Gamma_n)$ is defined by its coordinate ring $I=(\varphi_1, \varphi_2, \varphi_3)$ for the $\varphi_i$  defined in  Definition~\ref{definition:phis}.  
The sign $\epsilon$ changes these equations to 
\begin{align*}
\varphi_1 & = f_{n-1}(y) + \epsilon (x^2-1) \\
\varphi_2 & = f_n(y)+\epsilon (xz-y)\\
\varphi_3 & = x(f_{n+1}(y)-\epsilon) -z f_n(y).
\end{align*}
As the representations with $\epsilon=1$ lift, we will now assume $\epsilon=-1$.  We wish to determine the variety defined by 
\begin{align*}
\phi_1 & = f_{n-1}(y) - (x^2-1) \\
\phi_2 & = f_n(y)-(xz-y)\\
\phi_3 & = x(f_{n+1}(y)+1) -z f_n(y)
\end{align*}
in terms of the variables $x^2$, $y^2$ and $z^2$.

First, we consider characters such that $f_n(y)=0$ and $f_{n-1}(y)+1=0$.  (There are $\tfrac{n}2-1$ values which satisfy both conditions.)
\begin{lemma}\label{lemma:psllemma1}
 Assume $\rho$ is a representation from $\Gamma_n$ to $\PSL_2(\C)$ and  $f_n(y)=0$. 
 If $n \equiv 2 \pmod 4$ then the characters are 
 \[ \{(2,\bar{y},\tfrac12 \bar{y})\in \A^3(\bar{x},\bar{y},\bar{z}): \bar{y}\in (\Rb_{2n}^{fib})^2- (\Rb_{n}^{fib})^2 \}. \]
If  $n\equiv 0 \pmod 4$ we also have the characters $\{(0,0,\bar{z})\in \A^3(\bar{x},\bar{y},\bar{z}) \}$. 
  \end{lemma}

\begin{proof}
If $f_n(y)=0$ then $y\in \Rb_{2n}^{fib}$.  Therefore either $f_{n-1}(y)=-1$ and $f_{n+1}(y)=1$ or $f_{n-1}(y)=1$ and $f_{n+1}(y)=-1$.  (If $y=b+b^{-1}$ then $b^{2n}=1$ and the first case is for those $y$ such that $b^n=1$ and the second corresponds to those $y$ such that $b^n=-1$.)   In the first case,  the equation $\phi_1=0$ implies that $x=0$ and $\phi_2=0$ implies that $y=0$. The $\phi_3=0$ equation holds for any $z$.   However, $f_{n-1}(0)=-1$ holds  only if $n\equiv 0 \pmod 4$.  
In the second case, $f_{n-1}(y)=1$ and $f_{n+1}(y)=-1$ and $\phi_1$ reduces to $2-x^2$.  Similarly, $\phi_2= y-xz$ and $\phi_3$ is identically zero.  Therefore, $x^2=2$ and $y^2=x^2z^2=2z^2$. 
\end{proof}

\begin{lemma}\label{lemma:pslroots2}
 Assume $\rho$ is a representation from $\Gamma_n$ to $\PSL_2(\C)$ and  $f_n(y)=-1$.
  If $n \equiv 2 \pmod 4$ then the characters are 
\[ \{ (0, \bar{y},0)\in \A^3(\bar{x}, \bar{y},\bar{z}): \bar{y} \in (\Rb_{n-2}^{fib})^2-(\Rb_{(n-2)/2}^{fib})^2    \}. \]  
If $n\equiv 0 \pmod 4$ we also have the characters $\{ (0,0,\bar{z})\in  \A^3( \bar{x},\bar{y}, \bar{z})\}$.
\end{lemma}

\begin{proof} The roots of $f_{n-1}(y)+1$ are those $y=b+b^{-1}$ where $b^n=1$ or $b^{n-2}=-1$.  The first case is the first case of Lemma~\ref{lemma:psllemma1}, giving the characters $(0,0,\bar{z})$ for $n\equiv 0 \pmod 4$.   Assume $b^{n-2}=-1$, so that $f_{n-2}(y)=0$, $f_n(y)=-y$ and $f_{n+1}(y)=1-y^2$. The equations reduce to $\phi_1=-x^2  $, $\phi_2 = -xz $, and $\phi_3= x(2-y^2)+yz$.  Therefore, $x=0$ and $yz=0$.  The solutions are $(0,0,z^2)$ (which can occur only if $n\equiv 0 \pmod 4$) and $(0,y^2,0)$ for $y^2 \in (\Rb_{n-2}^{fib})^2-(\Rb_{(n-2)/2}^{fib})^2  $.
\end{proof}

\begin{remark}
The characters $(0,0,\bar{z})$  correspond to representations of $\langle \alpha, \beta: \alpha^2, \beta^2 \rangle \cong \Z/2\Z * \Z/2\Z$.  When $n\equiv 2 \pmod 4$ there are $\PSL_2(\C)$  representations corresponding to the characters $(0,0,z)$ that lift.  These are given by 
\[ \rho(\alpha) = \epsilon_A \mat{i}{0}{s}{-i}, \quad \rho(\beta)= \epsilon_B \mat{i}{s}{0}{-i}.\]
The representations corresponding to the characters $(0,0,\bar{z})$ when $n\equiv 0 \pmod 4$ are also as above but do not lift.  They do not lift because the relation $\beta^{-n} = \alpha^{-1}\beta \alpha^2 \beta \alpha^{-1}$ on the level of matrices is $(-I)^{n/2}=-I$, where $I$ is the $2\times 2$ identity matrix.

The characters $(2, \bar{y}, \tfrac12 \bar{y})$ with $\bar{y}\in (\Rb_{2n}^{fib})^2- (\Rb_{n}^{fib})^2$   correspond to representations of $\langle \alpha, \beta: \beta^n, \alpha^4,  \beta \alpha^2 \beta = \alpha^2 \rangle$.  The condition on $\bar{y}$ ensures that considering the matrices in $\SL_2(\C)$, $\rho(\beta^n) = -I$. 
The representations corresponding to the characters $(0,\bar{y},0)$ with $\bar{y} \in (\Rb_{n-2}^{fib})^2-(\Rb_{(n-2)/2}^{fib})^2  $  are representations of  $\langle \alpha, \beta : \alpha^2, \beta^{n-2}, \beta^2  \alpha \beta^2=\alpha  \rangle.$ The obstruction to lifting many of these representations if evident by the existence of elements of order two.

\end{remark}

Equations $\phi_1$, $\phi_2$ and $\phi_3$ allow us to parametrize the remaining solutions.   The coordinate ring is contained in the ring $R$ generated by the polynomials $\phi_1$, $\phi_2' = \phi_2(f_n(y)+(xz-y))$, and $\phi_3' = \phi_3(x(f_{n+1}(y)+1)+zf_n(y))$. Using that $\phi_2 = 0$, these polynomials are 
\begin{align*}
\phi_1 & = f_{n-1}(y)+1 -x^2\\
\phi_2' & = (f_n(y)+y)^2-z^2x^2\\
\phi_3' & = x^2(f_{n+1}(y)+1)^2-z^2f_n(y)^2.
\end{align*}
Note that these can be expressed in terms of the variables $x^2$, $y^2$ and $z^2$.  
 By Lemma~\ref{lemma:fibonaccievenodd} the polynomial $f_k$ is even when $k$ is odd and odd when $k$ is even.  Therefore, since $n$ is even, $f_{n+1}(y)$, $f_{n-1}(y)$, and $(f_n(y)+y)^2$ are polynomials in $y^2$.  
Substitution for $x^2$ can be obtained by considering $\phi_2'-z^2\phi_1$ and $\phi_3'+(f_{n+1}(y)+1)^2x^2$.  These equations are linear in $\phi_2'$ and $\phi_3'$,  so the polynomials
\begin{align*}
\phi_1 & = f_{n-1}(y)+1 -x^2\\
\phi_2'' & = (f_n(y)+y)^2-(f_{n-1}(y)+1)z^2\\
\phi_3'' & = (f_{n-1}(y)+1)(f_{n+1}(y)+1)^2-z^2f_n(y)^2.
\end{align*}
generate $R$ as well.
Assuming that $f_{n-1}(y)+1$ and $f_n(y)$ are non-zero, this gives $z^2$ as $(f_{n}(y)+y)^2/(f_{n-1}(y)+1)$ and $(f_{n-1}(y)+1)(f_{n+1}(y)+1)^2/f_n(y)^2$.  By Lemma~\ref{lemma:pslfib} this is in fact an identity.  Therefore this is a parametrization of the solutions as
\[  \big( f_{n-1}(y)+1, y, (f_n(y)+y)^2/(f_{n-1}(y)+1) \big). \]

\begin{definition} Define the polynomials $q_1(u)$, $q_2(u)$ and $q_3(u)$ so that $q_1(u^2) = f_{n-1}(u)+1$, $q_2(u^2)= (f_n(u)+u)^2$, and $q_3(u^2) = f_{n-1}(u)+1$. \end{definition}
Therefore, the parametrization is 
\[ \big(q_1(y^2), y^2, \frac{q_2(y^2)}{q_3(y^2)}\big).\] 
 Together with Lemma~\ref{lemma:psleven1} we have shown the following.

\begin{thm}\label{thm:psleven}

Assume $|n|>2$ is even. The quotient $Y_0(\Gamma_n)= X_0(\Gamma_n)/\mu_2 $ corresponds to the identification of $F$ under the hyperelliptic involution $(y,w)\mapsto (y,-w)$ and the involution $(y,w) \mapsto (-y,w)$.  The set $Y_0(\Gamma_n)$ is birational to an $\A^1$.  When $n\equiv 2 \pmod 4$ the identification on the additional line component is an involution.

 The portion of $Y(\Gamma_n) \subset \A^3(\bar{x},\bar{y},\bar{z})$ that does not lift to $X(\Gamma_n)$ is determined by components isomorphic to points and lines. The points are $(2,\bar{y}, \tfrac12 \bar{y})$ for $\bar{y} \in (\Rb_{2n}^{fib})^2- (\Rb_{n}^{fib})^2$ and $(0,\bar{y},0)$ for $\bar{y}\in (\Rb_{n-2}^{fib})^2- (\Rb_{(n-2)/2}^{fib})^2$.
The lines are given   parametrically by $(q_1(\bar{y}), \bar{y}, \frac{q_2(\bar{y})}{q_3(\bar{y})})$ and when $n\equiv 0 \pmod 4$ there is an additional  line component with characters $(0,0,\bar{z})$.
\end{thm}

\begin{remark}
The parametrization is not well-defined for $y$ which are roots of $f_{n+1}(y)+1$  these solutions are those from Lemma~\ref{lemma:pslroots2}.
\end{remark}

%
%
\section{The Whitehead link and lens space fillings}\label{section:whitehead}
%
%

The  manifolds $M_n$ are precisely the once-punctured torus bundles that are obtained by filling one boundary component of the Whitehead link exterior.

\subsection{Fillings of the Whitehead link}
Let $W=W(\cdot,\cdot)$ denote the exterior of the Whitehead link shown in Figure~\ref{fig:whiteheadlink}, and
 $W(p/q, \cdot)$ denote $p/q$ Dehn filling on one boundary component.
In particular, for a component $K$ of the Whitehead link with regular solid torus neighborhood $N(K)$, $\bdry N(K)$ is a component of $\bdry W$, and $W(p/q,\cdot)$ is formed by attaching a solid torus to $W$ along $\bdry N(K)$ so that an oriented curve representing $p \mu_K + q \lambda_K \in H_1 (\bdry N(K))$ bounds a meridional disk.  
Here for some orientation of the knot $K$, $\mu_K$ is the class in $H_1(\bdry N(K))$ of the oriented meridian of $N(K)$ linking $K$ once positively and $\lambda_K$ is the class of the boundary of a  Seifert surface oriented parallel to $K$.

\begin{figure}
\begin{center}
\includegraphics[width=1.5in]{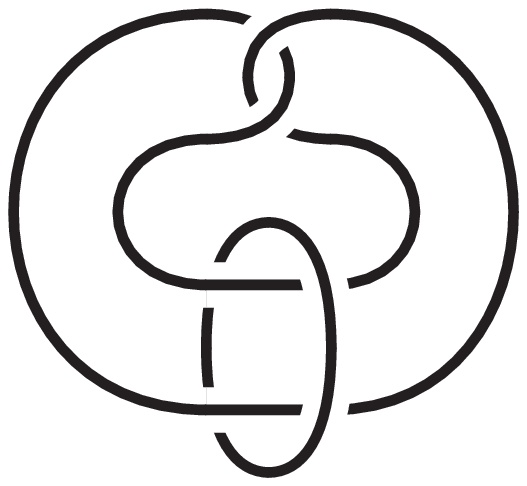}
\end{center}
\caption{}
\label{fig:whiteheadlink}
\end{figure}

\begin{lemma}[Proposition~3, \cite{hmw:sotwlygsm}]\label{lem:whiteheadfiber}
$W(p/q,\,\cdot\,)$ fibers over the circle if and only if $|q|\leq1$.  For each $p \in \Z$, $W(p/1, \,\cdot\,)$ fibers with a once-punctured torus fiber and monodromy $\phi \cong \tau_c \tau_b^{-p}$.
\end{lemma}

Thus these manifolds are surgeries on the Whitehead link exterior: $M_n \cong W(-(n+2),\cdot)$.  We also want to consider the lens space fillings of $M_n$.   Fortunately, the lens space fillings of $W$ all factor through at least one of our $M_n$.

\begin{lemma}[Martelli-Petronio,  \cite{mp:dfotm3m}]\label{lem:whiteheadlensspace}
$W(\gamma_1, \gamma_2)$ is a lens space if and only if $\{\gamma_1,\gamma_2\}$ is  $\{-1,-6+1/k\}$, $\{-2, -4+1/k\}$, $\{-3,-3+1/k\}$, or $\{p/q, \infty\}$ for some $k,p,q\in\Z$ with $(p,q)=1$. 
\end{lemma}
\begin{proof}
A bit of Kirby Calculus shows this is a direct consequence of the results listed in Table~A.5 of  \cite{mp:dfotm3m}.
\end{proof}

\begin{lemma}\label{lem:lensspacefillings}
The tunnel number one, once-punctured torus bundle $M_n = W(p, \cdot)$ for $p\in \Z$ has the single lens space filling $W(p,\infty)=L(p,1)$ with the following exceptions.
\begin{itemize}
\item $p=-1$:   $W(-1,-6+1/k) = L(6k-1, 2k-1)$, $k\in \Z$
\item $p=-2$:   $W(-2, -4+1/k) = L(8k-2, 2k-1)$, $k\in \Z$
\item $p=-3$:   $W(-3, -3+1/k)=L(9k-3,3k-2)$, $k\in \Z$
\item $p=-4$:  $W(-4, \infty)=L(4,-1)$ and $W(-4,-3)=L(12,5)$
\item $p=-5$:  $W(-5,\infty)=L(5,-1)$,  $W(-5,-1)=L(5,1)$, and $W(-5,-2)=L(10,3)$
\item $p=-7$:  $W(-7,\infty)=L(7,-1)$ and $W(-7,-1)=L(7,3).$
\end{itemize}
\end{lemma}

\begin{proof}  
This follows from Lemma~\ref{lem:whiteheadfiber} and Lemma~\ref{lem:whiteheadlensspace}.
\end{proof}

\begin{remark}
Lemma~\ref{lem:lensspacefillings} may be viewed as describing the knots in lens spaces whose exteriors are tunnel number one, once-punctured torus bundles and hence homeomorphic to some $M_n$.  Each core of the filling $W(\gamma_1, \gamma_2)$ listed gives such a knot except in the cases where the other filling slope is $\infty$ (which causes the knot exterior to be a  solid torus).  
   
Among the null homologous, genus one, fibered knots in lens spaces, those whose exteriors have  tunnel number one account for only a portion, \cite{BJK}.  On the other hand, if the exterior of a non null homologous knot in a lens space is a once-punctured torus bundle, then it necessarily has tunnel number one, \cite{optils}.
\end{remark}

\begin{lemma}\label{lemma:cyclicquotients}
For each $n \in \Z$ the only cyclic quotient of $\Gamma_n$ by a primitive peripheral element is $\Gamma_n/\langle \mu \rangle = \langle \beta \colon \beta^{n+2}\rangle \cong \Z/ (n+2)\Z$ with the following additions:
\begin{itemize}
\item $\Gamma_{-1}/\langle \lambda^{k}\mu^{6k-1}\rangle  = \langle \eta \vert \eta^{6k-1} \rangle \cong \Z/ (6k-1)\Z$, $k \in \Z$ where $\eta = \alpha \beta \alpha \beta \alpha$
\item $\Gamma_{0}/\langle \lambda^{k}\mu^{4k-1}\rangle = \langle \alpha \vert \alpha^{8k-2} \rangle \cong \Z/ (8k-2)\Z$, $k \in \Z$ where $\beta = \alpha^{4k-1}$
\item $\Gamma_{1}/\langle \lambda^{k}\mu^{3k-1}\rangle  = \langle \alpha \vert \alpha^{9k-3} \rangle \cong \Z/ (9k-3)\Z$, $k \in \Z$ where $\beta = \alpha^{3k-1}$
\item $\Gamma_{2}/\langle \lambda\mu^3\rangle= \langle \alpha \vert \alpha^{12} \rangle \cong \Z/ 12\Z$, where $\beta = \alpha^3$
\item $\Gamma_{3}/\langle \lambda\mu\rangle= \langle \beta \vert \beta^5 \rangle \cong \Z/ 5\Z$ where $\alpha= \beta^4$ 
\item $\Gamma_{3}/\langle \lambda\mu^2\rangle = \langle \alpha \vert \alpha^{10} \rangle \cong \Z/10\Z$ where $\beta = \alpha^4$,
\item $\Gamma_{5}/\langle \lambda\mu\rangle = \langle \eta \vert \eta^7 \rangle \cong \Z/7\Z$  where $\eta = \beta \bar{\alpha}$.
\end{itemize}
\end{lemma}

\begin{proof}
Any quotient of $\Gamma_n$ by a primitive peripheral element gives the fundamental group of the corresponding Dehn filling of $M_n$.  Thus Lemma~\ref{lem:lensspacefillings} determines the possible cyclic quotients by primitive peripheral elements.  To determine the resulting generators, is a straightforward (yet tedious) exercise.   For this purpose, note that we have
 $\Gamma_n = \langle  \alpha, \beta \colon
\bar{\beta}^n = \bar{\alpha} \beta \alpha^2 \beta \bar{\alpha} \rangle$ with $\mu = \beta \alpha$ and 
$\lambda = \gamma \beta \bar{\gamma} \bar{\beta} = \alpha \beta \bar{\alpha} \beta \alpha \bar{\beta} \bar{\alpha} \bar{\beta}$.
\end{proof}

\subsection{Characters of Lens Space Fillings}

The representation of a quotient of $\Gamma_n$ also induces a representation of $\Gamma_n$.  Lemma~\ref{lemma:cyclicquotients} gives the list of cyclic quotients of $\Gamma_n$ by primitive peripheral elements of the form $\lambda^p \mu^q$ where $(p,q)=1$.  We will identify the characters of these cyclic quotients on the character variety of $\Gamma_n$.  As usual, the characters of a representation $\rho \colon \Gamma_n \to \SL_2(\C)$ are given as points $(x,y,z)$ where $x=\text{tr}(\rho(\alpha))$, $y=\text{tr}(\rho(\beta))$, and $z=\text{tr}(\rho(\alpha\beta))$.

Recall from Definition~\ref{definition:rootsofunity} that $\Zeta_n$ in the set of all $|n|^{th}$ roots of unity.  Set \[ Z(\zeta) = \mat{\zeta}{0}{0}{\zeta^{-1}}.\]

Assume $\rho$ is a representation of such a cyclic quotient of $\Gamma_n$ with $\rho(\alpha)=A(a,s)$ and $\rho(\beta)=B(b,s)$.  By Lemma~\ref{lemma:2generatorconjugation} we may assume both $A$ and $B$ are upper-triangular.    Since $\rho$ is abelian, the group relation in $\Gamma_n$ implies $B^{n+2} = I$.  Also being abelian implies $\rho(\lambda)=I$ so that we further obtain the relation $\rho(\lambda^p \mu^q) = B^q A^q = I$ in the quotient.     

 \begin{prop}\label{prop:quotientbymeridian}
For each integer $n\neq-2$ the characters of the cyclic quotient $\Gamma_n/\langle \mu \rangle \cong \Z/ (n+2)\Z$ are given by the finite set of points
$ \{(y,y,2) : y \in \Rb_{n+2} \}$.

The characters  of the cyclic quotient $\Gamma_{-2}/\langle \mu \rangle \cong \Z$ are given by the points in the line $\{(y,y,2)\} $.
\end{prop}

\begin{proof}
By Lemma~\ref{lemma:cyclicquotients}, we have the cyclic quotients $\Gamma_n/\langle \mu \rangle = \langle \beta : \beta^{n+2} \rangle$ for all $n \in \Z$ so that $A = B^{-1}$.  Therefore $x=y$ and $z=2$.  When $n \neq -2$, $B =Z(\zeta)$ for $\zeta \in \Zeta_{n+2}$, and hence $y \in \Rb_{n+2}$.   When $n=-2$, $B$ may be any element of $\SL_2(\C)$.  
\end{proof}

Proposition~\ref{prop:quotientbymeridian} together with Lemma~\ref{prop:redreps} gives the following.

\begin{lemma} 
For each $n\in \Z$ there is a representation of $\Gamma_n$ corresponding to a lens space filling (or quotient of a lens space filling) of $M_n$ on each reducible (and therefore on each abelian) conic and line component. \qed
\end{lemma}

 Lemma~\ref{lemma:cyclicquotients} also demonstrates that there are additional cyclic quotients when $n\in \{ -1,0,1,2,3,5\}$.   We determine the characters of these.
 \begin{prop}\
 \begin{itemize}
 \item The characters of $\Gamma_{-1}/\langle \lambda^{k}\mu^{6k-1}\rangle   \cong \Z/ (6k-1)\Z$, $k \in \Z$, are given by the points $\{(x,2,x) : x \in \Rb_{18k-3}\}$.
 
 \item The characters of $\Gamma_{0}/\langle \lambda^{k}\mu^{4k-1}\rangle  \cong \Z/ (8k-2)\Z$, $k \in \Z$, are given by the points 
 $\{(x,2,x) : x \in \Rb_{4k-1}\} \cup \{(x,-2,x) : x \in \Rb_{8k-2}-\Rb_{4k-1}\}$.
 
 \item The characters of $\Gamma_{1}/\langle \lambda^{k}\mu^{3k-1}\rangle  \cong \Z/ (9k-3)\Z$, $k \in \Z$, are given by the points 
$\{ (2\Re(\zeta),2\Re(\zeta^{3k-1}),2\Re(\zeta^{3k}) ) : \zeta \in \Zeta_{9k-3} \}$.

 \item The characters of $\Gamma_{2}/\langle \lambda\mu^3\rangle \cong \Z/ 12\Z$ are given by the points 
$\{ (2\Re(\zeta),2\Re(\zeta^3),2\Re(\zeta^4)) : \zeta \in \Zeta_{12}\}$.

 \item The characters of  $\Gamma_{3}/\langle \lambda\mu\rangle \cong \Z/ 5\Z$ are given by the points 
 $\{(x,x,2) : x \in \Rb_5 \}$
 
  \item The characters of $\Gamma_{3}/\langle \lambda\mu^2\rangle  \cong \Z/10\Z$ are given by the points 
  $\{(x,x,2) : x \in \Rb_5 \} \cup \{(x,-x,-2) : x \in \Rb_{10}-\Rb_5\}$.

  \item The characters of $\Gamma_{5}/\langle \lambda\mu\rangle \cong \Z/7\Z$ are given by the points 
  $\{(x,x,2) : x \in \Rb_7 \}$
  
 \end{itemize}

 For $n=3$ and $n=5$, the characters of the quotient by $\lambda \mu$ are the same as the characters of the quotient by $\mu$.

\end{prop}

\begin{proof}

For $\Gamma_{-1}/\langle \lambda^k \mu^{6k-1} \rangle \cong \Z/(6k-1)\Z$,
since $n=-1$ we have $B = I$.   Lemma~\ref{lemma:cyclicquotients} then implies $\rho(\eta) = A^3 = Z(\zeta')$ for $\zeta' \in \Zeta_{6k-1}$ so that $A = Z(\zeta)$ for $\zeta \in \Zeta_{18k-3}$ .  Thus $y=2$ and $x=z\in \Rb_{18k-3}$.

\medskip

For $\Gamma_{0}/\langle \lambda^k \mu^{4k-1} \rangle$, since $n=0$ we have $B^2 = I$.  Lemma~\ref{lemma:cyclicquotients} implies $A = Z(\zeta)$ for $\zeta \in \Zeta_{8k-2}$  so that $x \in \Rb_{8k-2}$.  We also must have $B = A^{4k-1}$. So if $B=I$ then $x \in \Rb_{4k-1}$, $y=2$, and $z=x$.  But if $B=-I$ then $x \in \Rb_{8k-2}-\Rb_{4k-1}$, $y=-2$, and $z=x$.

\medskip

For $\Gamma_{1}/\langle \lambda^k \mu^{3k-1} \rangle$, since $n=1$ we have $B^3 = I$.  Lemma~\ref{lemma:cyclicquotients} implies $A = Z(\zeta)$ for $\zeta \in \Z_{9k-3}$ so that $x \in \Rb_{9k-3}$.  We also must have $B=A^{3k-1} = Z(\zeta^{3k-1})$ and $AB =Z(\zeta^{3k})$.  
Thus $x = 2\Re(\zeta)$, $y =2\Re(\zeta^{3k-1})$, and $z = 2\Re(\zeta^{3k})$ where $\zeta \in \Zeta_{9k-3}$.
In particular  if $B=I$ then $\zeta \in \Zeta_{3k-1}$ and $x\in\Rb_{3k-1}$, $y=2$, and $z=x$.

\medskip

For $\Gamma_2/\langle \lambda \mu^3 \rangle$,  Lemma~\ref{lemma:cyclicquotients} implies $A=Z(\zeta)$ for $\zeta \in \Zeta_{12}$, $B = A^3 = Z(\zeta^3)$, and $AB = Z(\zeta^4)$.  Thus $x = 2\Re(\zeta)$, $y =  2\Re(\zeta^3)$, and  $z =2\Re(\zeta^4)$.

\medskip

For $\Gamma_3/\langle \lambda \mu \rangle$,  Lemma~\ref{lemma:cyclicquotients} implies $B=Z(\zeta)$ for $\zeta \in \Zeta_5$, $A = B^4 = B^{-1}$, and $AB = I$.  Thus $x=y \in \Rb_5$ and  $z =2$, the same characters as the quotient by $\mu$. 

\medskip

For $\Gamma_3/\langle \lambda \mu^2 \rangle$, Lemma~\ref{lemma:cyclicquotients} implies $A=Z(\zeta)$ for $\zeta \in \Zeta_{10}$, $B = A^4 = Z(\zeta^4)$, and $AB = Z(\zeta^5)=\pm I$.  Thus $x = 2\Re(\zeta)$, $y =  2\Re(\zeta^4)$, and  $z =2\Re(\zeta^5)$ for $\zeta \in \Zeta_{10}$.
Hence if $AB=I$ then $\zeta\in\Zeta_5$ and $x \in \Rb_5$, $y=x$, and $z=2$.  If $AB=-I$, then $\zeta \in \Zeta_{10}-\Zeta_5$ and $x \in \Rb_{10}-\Rb_5$, $y=-x$, and $z=-2$.

\medskip

Finally, for $\Gamma_5/\langle \lambda \mu \rangle$ we have $n=5$ so that $B^7 = I$ implying $B = Z(\zeta)$ for $\zeta \in \Zeta_7$.   Since $AB=1$, Lemma~\ref{lemma:cyclicquotients} then implies $\rho(\eta) =BA^{-1} = B^2 = Z(\zeta^2) $ giving no further relations.  Thus $x=y \in \Rb_7$ and $z = 2$,  the same characters as the quotient by $\mu$. 
\end{proof}

%
%
\section{The Intersection of the Reducible and Irreducible  Components}\label{section:intersections}
%
%

\begin{lemma}  
If $n\neq -2$, the intersection of the reducible components and the variety $C$ is the set of points
\[\{(\epsilon y, y, 2\epsilon ) : y \in \Rb_{2(n+2)}^{fib}, \epsilon = \pm1 \}.\]
with the points
 \[ \{(-\sqrt{2-n}, 2, -\sqrt{2-n}), (\sqrt{2-n}, 2, \sqrt{2-n})\}\]
and, if $n$ is even, the points 
\[ \{(\sqrt{2-n},-2,-\sqrt{2-n}), (-\sqrt{2-n},-2, \sqrt{2-n})\}. \] 
Therefore each conic or line of reducible (abelian) representations intersects the irreducible component twice.  

When $|n|>2$, each of these intersections occurs on the canonical component $X_0(\Gamma_n)$ except when $n\equiv 2 \pmod 4$.  
In that case all points  are on the canonical component,  except $(0,0,\pm 2)$ which are on the line $x=y=0$. 
\end{lemma}

\begin{proof}
By Lemma~\ref{lemma:redchars}, when $n\neq -2$ the reducible components $\tilde{X}_{red}(\Gamma_n)$ are defined by $y\in \Rb_{n+2}$   and the equation $x^2+y^2+z^2-xyz=4$.  
By Proposition~\ref{prop:irreps} the irreducible component $X(\Gamma_n)$ is given by the variety $C$ from Definition~\ref{definition:phis}.
Assume $(x,y,z)$ is in the intersection of these components.

Recall $\Rb_N^{fib} = \Rb_N - \{ \pm 2\}$.  First we consider $y\neq \pm 2$. Then by Lemma~\ref{lemma:fibroots} $f_{n+2}(y) =0$ from with the Fibonacci recurrence gives $f_n(y) = y f_{n+1}(y)$ and $f_{n-1}(y) = (y^2-1)f_{n+1}(y)$.
  Therefore, 
\begin{align*}  
\varphi_1(x,y,z) & = x^2-1+f_{n-1}(y) = x^2-1+(y^2-1)f_{n+1}(y) \\
\varphi_2(x,y,z) & =  zx-y+f_n(y)  = zx-y+y f_{n+1}(y)\\
\varphi_3(x,y,z) & = x(f_{n+1}(y)-1)-zf_n(y) = -x+(x-zy)f_{n+1}(y)
\end{align*}
If $y=0$ then $f_{n+1}(0) =0$ or $\pm1$, by Lemma~\ref{lemma:fkvals}.  First consider the case when $f_{n+1}(0)=0$.  By the Fibonacci recursion, this can only occur  when $n$ is odd.   However, $\pm i$ is not a $2(n+2)$nd root of unity in this case, and therefore $y\neq 0$.  

Now, assume that $f_{n+1}(0)=\pm1$, so that $x^2= 1\pm 1$ by $\varphi_1$, $zx=0$ by $\varphi_2$, and $x = \pm x$ by $\varphi_3$.  If $\pm = +$, then $x^2 =2$ and $z=0$, but this contradicts that $x^2 + y^2 + z^2 -xyz=4$.  Hence $x=0$ and $z= \pm2$.

If $y\neq0$ then $\varphi_2$ implies $y\varphi_3 = -xy + (x-zy)(y-zx) = -z(x^2+y^2-xyz).$  Thus, using $x^2+y^2+z^2-xyz=4$, either $z=0$ or $z = \pm2$.
If $z=0$ then $x^2+y^2 = 4$ but also $\varphi_2$ implies $f_{n+1}(y) = 1$ so that $\varphi_1$ implies $x^2 + y^2 = 2$, a contradiction.  Hence $z=\epsilon 2$ where $\epsilon = \pm1$.  Thus $x^2- \epsilon2xy+y^2 = 0$ from which we conclude $x=\epsilon y$.  Then $\varphi_2 = y(1+f_{n+1}(y))$ so that $f_{n+1}(y) = -1$, satisfying $\varphi_1$ and $\varphi_3$.  Thus, together with the points from the case $y=0$, we have the points
\[\{(\epsilon y, y, 2\epsilon ) : y \in \Rb_{2(n+2)}^{fib}, \epsilon = \pm1 \}.\]

If $y=2$, then as $f_k(2)=k$ by Lemma~\ref{lemma:fkvals}, at the point $(x,2,z)$, 
\[  \varphi_1 = x^2+n-2, \ \varphi_2=  zx+n-2, \ \text{ and } \  \varphi_3  = n(x-z). \]
This corresponds to the points $(-\sqrt{2-n}, 2, -\sqrt{2-n})$ and $(\sqrt{2-n}, 2, \sqrt{2-n})$.

If $y=-2$, then $n$ must be even to have $-2 \in \Rb_{2(n+2)}$.  As $f_k(-2)=k(-1)^{k+1}$ by Lemma~\ref{lemma:fkvals},  at the point $(x,-2,z)$, 
\[  \varphi_1 = x^2+n-2, \ \varphi_2  =  zx-n+2, \ \text{ and } \  \varphi_3  = n(x+z).\]
This corresponds to the points $(\sqrt{2-n},-2,-\sqrt{2-n})$ and $(-\sqrt{2-n},-2, \sqrt{2-n})$.

When $|n|>2$, these points are all on the canonical component unless perhaps $x=y=0$, as $C$ corresponds to the canonical component except when $n\equiv 2 \pmod 4$ in which case there is the additional line determined by $x=y=0$ by Proposition~\ref{prop:extraline}.  The points $(0,0,\pm 2)$ are of this form.  By Proposition~\ref{prop:extraline3} they are not on $X_0$.  
\end{proof}

Schematic diagrams of the  components of $\tilde{X}(M_n)$ and the intersections of the various components are shown in Figure~\ref{fig:schematic} and Figure~\ref{fig:schematic2}.

\begin{figure}
\begin{center}
\includegraphics[width=5in]{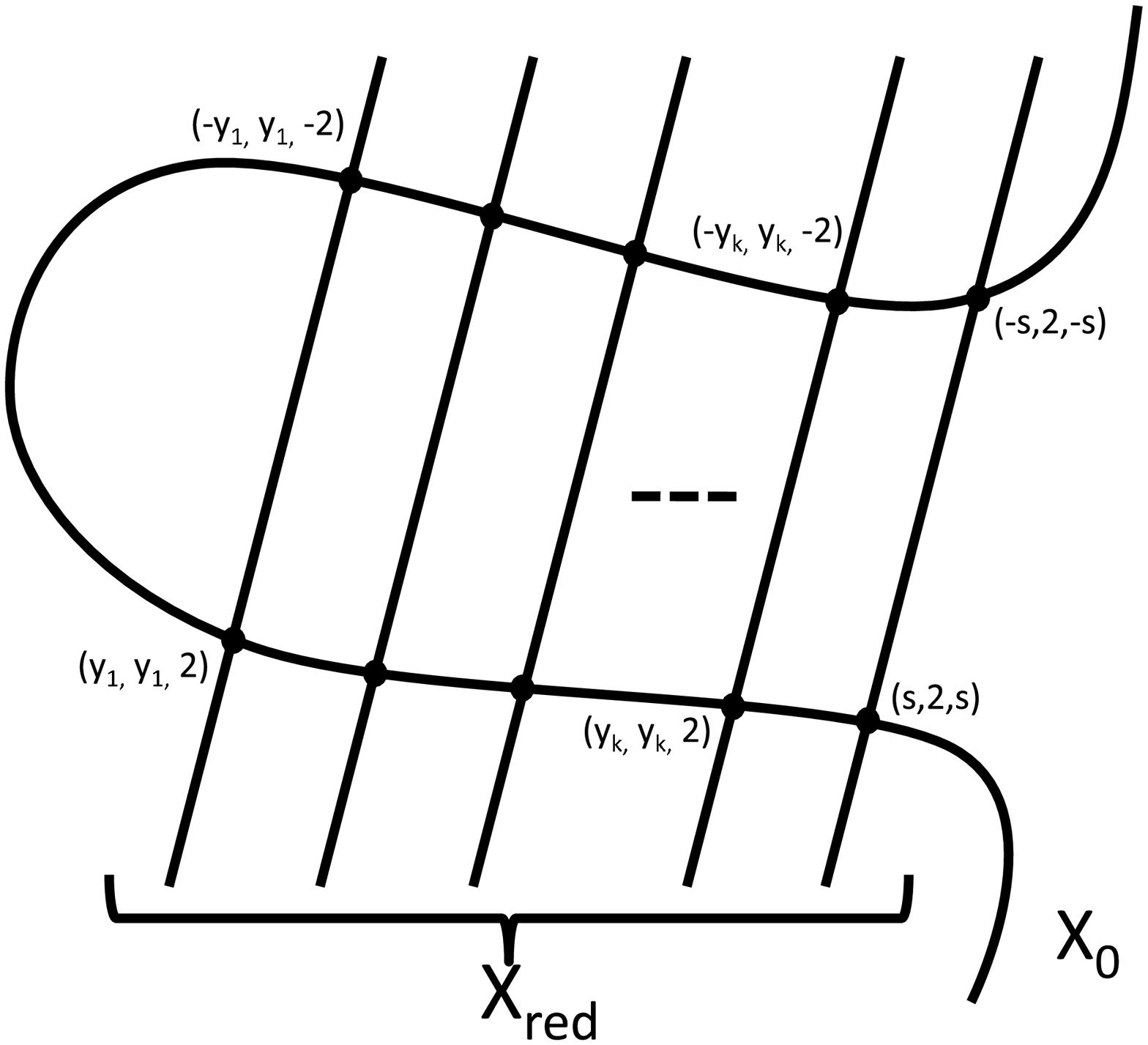}
\end{center}
\caption{A schematic of the  intersection of $X_0(\Gamma_n)$ and $X_{red}(\Gamma_n)$ when $n$ is odd. The $y_k$ are elements of $\Rb_{2(n+2)}^{fib}$ and $s=\sqrt{2-n}$.}
\label{fig:schematic}
\end{figure}

\begin{figure}
\begin{center}
\includegraphics[width=5in]{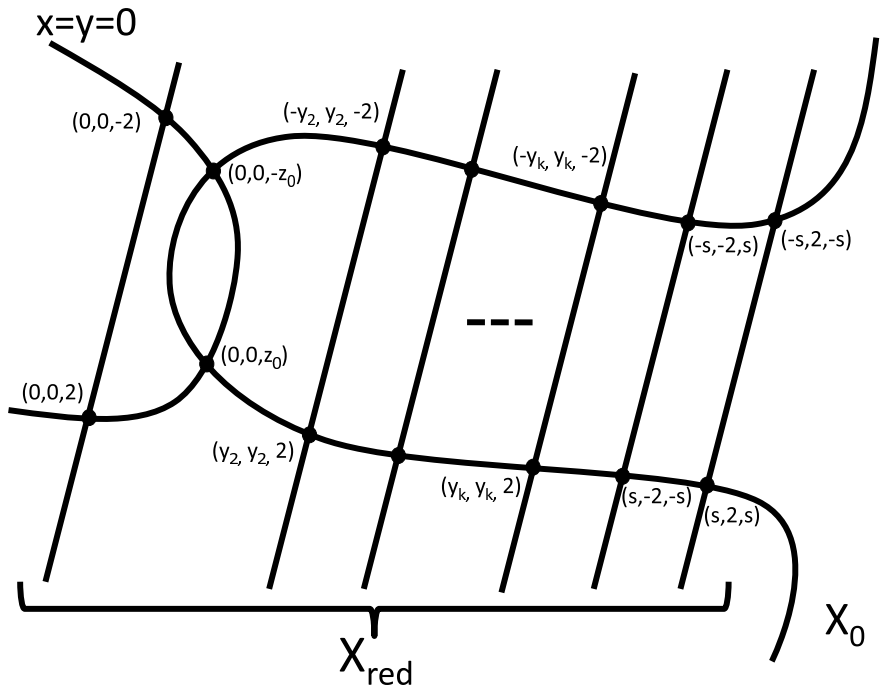}
\end{center}
\caption{A schematic of the  intersection of $X_0(\Gamma_n)$ and $X_{red}(\Gamma_n)$ when $n\equiv 2 \pmod 4$ with $z_0=2\sqrt{\tfrac12-\tfrac1n}$ (as discussed in Proposition~\ref{prop:extraline3}).}
\label{fig:schematic2}
\end{figure}

\subsection{Representations Corresponding to Intersections}

For all integers $n$ such that $|n|>2$ 
we now describe representations corresponding to the intersection points. That is, we determine the intersections in the hyperbolic cases.  For the intersection points of $X_{red}(\Gamma_n)$ and $X(\Gamma_n)$ we describe the associated diagonal representation.  For the points $(y,y,2)$ where $y\in \Rb_{n+2}^{fib}$, these correspond to the cyclic quotients by $\langle \mu \rangle$, to $\Z/(n+2)\Z$ and its quotients (except the trivial quotient and $\Z/2\Z$ quotient, when $n$ is even).   (When $n\equiv 2 \pmod 4$ the representations with $x=y=0$ and $z=\pm 2$ are not on the canonical component, but are on the line $x=y=0$. These correspond to faithful  representations of $\Z/4\Z$.)
When $n=3$ these also correspond to the quotients by $\lambda \mu$ and $\lambda \mu^2$.  When $n=5$ these correspond also to the quotient by $\lambda \mu$.  

The intersection points $(y,-y,-2)$ where $y\in \Rb_{n+2}^{fib}$ correspond to representations where $\rho(\alpha)=-\rho(\beta)$ and $\rho(\alpha^{n+2})=\rho(\beta^{n+2})=I$.  These are representations of $\Z/(n+2)\Z $ and its quotients. When $n=3$ these also correspond to the quotient by $\langle \lambda \mu^2 \rangle$.  The  $(y,-y,2)$ points correspond  to faithful representations of $\Z/10 \Z$ and the $(y,y,2)$ points correspond to faithful representations of $\Z/5\Z$.

Next, consider the points $(\sqrt{2-n},2,\sqrt{2-n})$ and $(-\sqrt{2-n}, 2, -\sqrt{2-n})$ which lie on the line $x=z, y=2$.
These correspond to diagonal representations where $\rho(\beta)=I$ and are faithful representations of $\Z$, generated by $\rho(\alpha)$.  Note that the characters of the trivial representation and of the representation $\alpha \mapsto -I$, $\beta \mapsto I$ are on this line, but are not on the canonical component.

When $n$ is even, the line $x=-z,y=-2$ 
intersects the canonical component in the points $(\sqrt{2-n},-2,-\sqrt{2-n})$ and $(-\sqrt{2-n},-2,\sqrt{2-n})$.  The corresponding diagonal representation is defined by $\beta \mapsto -I$, (so $z=-x$) and $\rho(\alpha)$ is free.  It is elementary to verify that the representations at the intersection points are faithful representations of $\Z$. 
Note that the representation,  $\alpha \mapsto I$, $\beta \mapsto -I$ as well as $\alpha \mapsto -I$, $\beta \mapsto -I$  are on this component, but not on the canonical component.   

As discussed in  Proposition~\ref{prop:extraline3}, in the case when $n\equiv 2 \pmod 4$, the intersection of the line $x=y=0$ with the canonical component consists of the points $(0,0, z_0)$ and $(0,0,-z_0)$  where $z_0 = 2\sqrt{\tfrac{1}{2}-\tfrac{1}{n}}$.  These correspond to faithful representations of $\langle \alpha, \beta : \alpha^2, \beta^2, \alpha^4 \rangle$. (These are not diagonal representations.)

%
%
\section{Symmetries of $M_n$}\label{section:symmetries}
%
%
We give explicit descriptions of two involutions, called $\spin$ and $\flip$, of our once-punctured torus bundles $M_n$ and their actions upon the fundamental group $\Gamma_n$.  Lemma~\ref{lem:symmetries} shows that these two involutions generate the symmetry group of $M_n$ for $|n|>3$.  Then Proposition~\ref{prop:action} shows how they each act on the character variety $X(\Gamma_n)$.   When $|n|=3$, there is also an orientation reversing diffeomorphism of order $4$ whose square is $\spin$.

\begin{lemma}\label{lem:symmetries}
For $|n|>3$, the group $\pi_0\, {\rm diff}(M_n)$ of diffeomorphisms of the manifolds $M_n$ modulo isotopy is $\Z/2\Z \times \Z/2\Z$.  When $n=-3$, $M_{-3}$ is the figure eight knot exterior and this group is the dihedral group $D_4$.  When $n=3$, $M_{-3}$ is the sister of the figure eight knot exterior and this group is $\Z/2\Z \times \Z/4\Z$.
\end{lemma}

\begin{proof}
The second paragraph of page 268 of \cite{FH} describes a method for determining the group $\pi_0\, {\rm diff}(M_n)$ based on the fact that any diffeomorphism of $M_n$ is isotopic to one that is fiber preserving and linear on each fiber.   Therefore $\pi_0\, {\rm diff}(M_n)$ is the quotient $N\langle [\phi] \rangle/\langle [\phi] \rangle$ where $\langle [\phi] \rangle$ is the subgroup of $\SL_2(\Z)$ generated by the monodromy matrix $[\phi]$ and $N\langle [\phi] \rangle$ is its normalizer in $\SL_2(\Z)$.  Furthermore, the $\Z/2\Z$--quotient $N\langle [\phi] \rangle/\langle [\phi], \pm I \rangle$ can be viewed in terms of the Farey Diagram of $\PSL_2(\Z)$ in the hyperbolic plane $\H$.  

When $|n|>2$ so that $M_n$ is hyperbolic and its monodromy $\phi$ is pseudo-Anosov, $\phi$ acts on $\H$ by translation and thus defines a periodic strip $\Sigma_\phi$ (shown in Figure~\ref{fig:strip}) of triangles in the Farey Diagram, as described at the bottom of page 266 of \cite{FH}.  The group  $N\langle [\phi] \rangle/\langle [\phi], \pm I \rangle$ is then the symmetry group of the triangulated cylinder $\Sigma_\phi/\phi$.  When $|n|>3$ one readily observes that this triangulated cylinder has symmetry group $\Z/2\Z$, generated by an orientation reserving involution of the cylinder that preserves its boundary components: flip the cylinder across two spanning arcs.    When $|n|=3$, the cylinder is triangulated with just two triangles and there is another orientation reversing involution that exchanges its boundary components: mirror across the center curve of the cylinder and then rotate around the cylinder by $\pi$.  These two involutions of the cylinder commute and thus the symmetry group of this triangulated cylinder is $\Z/2\Z \times \Z/2\Z$.

To complete the proof, below we describe two commuting involutions $\spin$ and $\flip$ upon each manifold $M_n$, shown schematically in Figure~\ref{fig:spininvolution} and Figure~\ref{fig:flipinvolution}.   Their induced actions on $\bdry M_n$, for example, show they are inequivalent.   Thus these generate $\pi_0\, {\rm diff}(M_n)$ for $|n|>3$.  Indeed, in the computation above, the $\spin$ involution corresponds the $\Z/2\Z$ action of $-I$ on $\SL_2(\Z)$ while the (orientation preserving) $\flip$ involution corresponds to the flip of the  triangulated cylinder $\Sigma_\phi/\phi$ that preserves its boundary components.  
When $|n|=3$, the two manifolds are the familiar figure eight knot complement and its sister. These are archiral and they each admit an orientation reversing diffeomorphism $r$ of order $4$ whose square is $\spin$ and descends to the extra involution on the triangulated cylinder.  One may check that $r$ commutes with $\flip$ on the sister $M_3$ whereas $r \flip = \flip r^{-1}$ on the figure eight knot complement $M_{-3}$.  We then have $\pi_0\, {\rm diff}(M_3) = \Z/2\Z \times \Z/4\Z$ and $\pi_0\, {\rm diff}(M_{-3}) =D_4$.
\end{proof}

\begin{figure}
\begin{center}
\includegraphics[height=1in]{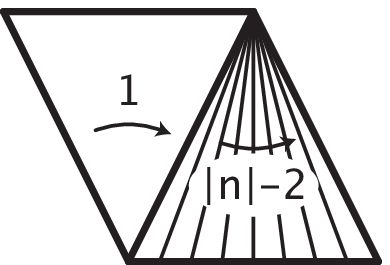}\\
\includegraphics[height=1in]{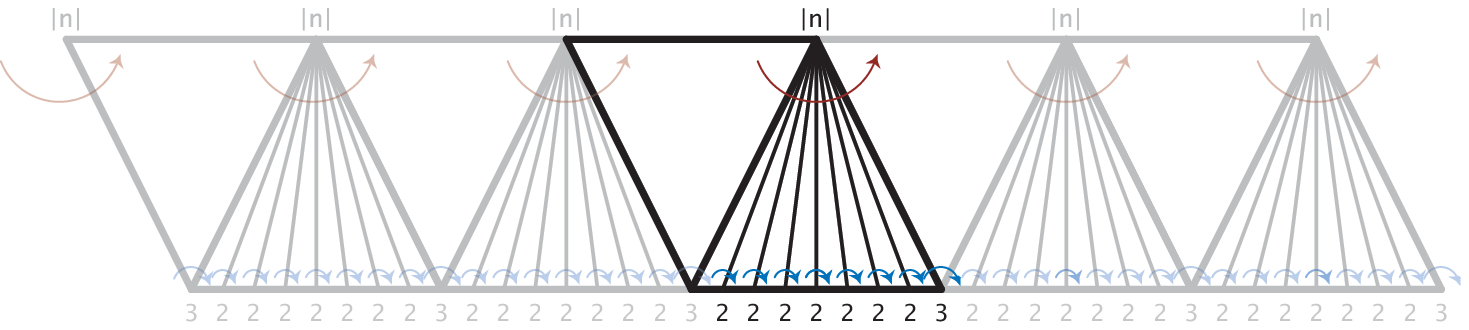}
\end{center}
\caption{}
\label{fig:strip}
\end{figure}

\subsection{Spin and Flip}
To describe these symmetries, we begin with a construction of our manifolds $M_n$ in terms of surgeries on two curves in the trivial once-punctured torus bundle.

Let $T$ be the once punctured torus, viewed as a square, minus a small open disk about the vertices, with opposite sides identified.  The horizontal and vertical midlines of this square give rise to curves $c$ and $b$ on $T$.  Starting at the basepoint $* = c \cap b$ and following $c$ to the right gives the fundamental representative of the homotopy class $\gamma \in \pi_1(T,*)$.   Starting at $*$ and following $b$ upwards gives $\beta$.   It will also be useful to instead have the basepoint $*$ of our homotopy classes on $\bdry T$; for this we will conjugate by the ``straight'' path from $c \cap b$ to the leftmost side of the bottom edge of the square.  We will use both of these positions of $*$.

Consider the trivial once-punctured torus bundle $T \times S^1$, and write $T_\theta$ for the fiber $T \times \{\theta\}$ and $c_\theta, b_\theta$ for the corresponding curves on it.  Frame these curves by the fiber on which they sit. 
The result of gluing of $T \times (\theta_1-\epsilon, \theta_1]$ to $T \times [\theta_0, \theta_0+\epsilon)$ by the positive Dehn twist map $(x,\theta_1) \sim (\tau_c(x),  \theta_0)$ may be also be obtained by gluing with the identity $(x, \theta_1) \sim (x, \theta_0)$ and then performing $-1$ Dehn surgery on $c_{\theta_0}=c_{\theta_1}$.
Thus our manifolds $M_n$ with monodromy $\phi = \tau_c \tau_b^{n+2}$ may be regarded as being obtained from $T \times S^1$ by $-1$ Dehn surgery on $c_{\frac{\pi}{4}}$ and $-1/(n+2)$ Dehn surgery on $b_{\frac{3\pi}{4}}$.   (Notice that the last surgery may be decomposed into a $-1/n$ Dehn surgery on $b_{\frac{3\pi}{4}}$ and $-1$ surgeries on the two push-offs $b_{\frac{3\pi}{4}-\epsilon}$ and $b_{\frac{3\pi}{4}+\epsilon}$.  The Whitehead link exterior is  then $(T \times S^1)-N(b_{\frac{3\pi}{4}})$ with $-1$ surgeries on $c_{\frac{\pi}{4}}$, $b_{\frac{3\pi}{4}-\epsilon}$, and $b_{\frac{3\pi}{4}+\epsilon}$.)

Each manifold $M_n$ admits two involutions which we will call $\spin$ and $\flip$ and are shown schematically in Figure~\ref{fig:spininvolution} and Figure~\ref{fig:flipinvolution}.   Clearly these are involutions of $T \times S^1$.  That these are involutions of $M_n$ may be observed from noting that $\spin$ and $\flip$ take $c_{\frac{\pi}{4}}$ and $b_{\frac{3\pi}{4}}$ back to themselves and induce orientation preserving homeomorphisms on small regular neighborhoods of each of these curves.

These involutions induce isomorphisms of $\Gamma_n \cong \pi_1(M_n,*)$. 

\begin{lemma}\label{lemma:spinflipgroup}
The involutions $\spin$ and $\flip$ induce the isomorphisms 
\[
\spin_* \colon  
\begin{cases}
 \alpha & \mapsto  \beta \alpha \bar{\beta}^{n+1}\\
 \beta & \mapsto  \bar{\beta}
\end{cases} 
\quad \quad \mbox{ and }\quad \quad 
\flip_* \colon  
\begin{cases}
 \alpha & \mapsto  (\beta \alpha) \beta \bar{\alpha} \bar{\beta} ( \bar{\alpha} \bar{\beta})\\
 \beta & \mapsto  (\beta \alpha) \bar{\beta} (\bar{\alpha} \bar{\beta}) .
\end{cases} 
\]

\end{lemma}

\begin{proof}

Figures~\ref{fig:spininvolution} and \ref{fig:flipinvolution} show the actions of $\spin$ and $\flip$ respectively on the curves $\gamma$ colored green, $\beta$ blue, and $\mu$ orange. A sequence of isotopies (fixing the base points) then makes the homotopy classes of their images more easily discerned.   From these figures one may read off the following:
\[
\spin_* \colon  
\begin{cases} 
 \gamma & \mapsto  \bar{\gamma}\\
 \beta & \mapsto  \bar{\beta}\\
 \mu & \mapsto    \gamma \mu \bar{\beta}^{n+2}
\end{cases} 
\quad \quad \mbox{ and }\quad \quad 
\flip_* \colon  
\begin{cases}
 \gamma & \mapsto \beta\gamma\bar{\beta} \\
 \beta & \mapsto \bar{\gamma}\bar{\beta} \\
 \mu & \mapsto \bar{\mu}.
\end{cases} 
\]

Using the relations $\alpha = \bar{\beta} \mu$ and $\gamma = \bar{\beta} \mu \beta \bar{\mu}$ gives the result.
\end{proof}

\begin{figure}
\centering
\includegraphics[height=5in,keepaspectratio]{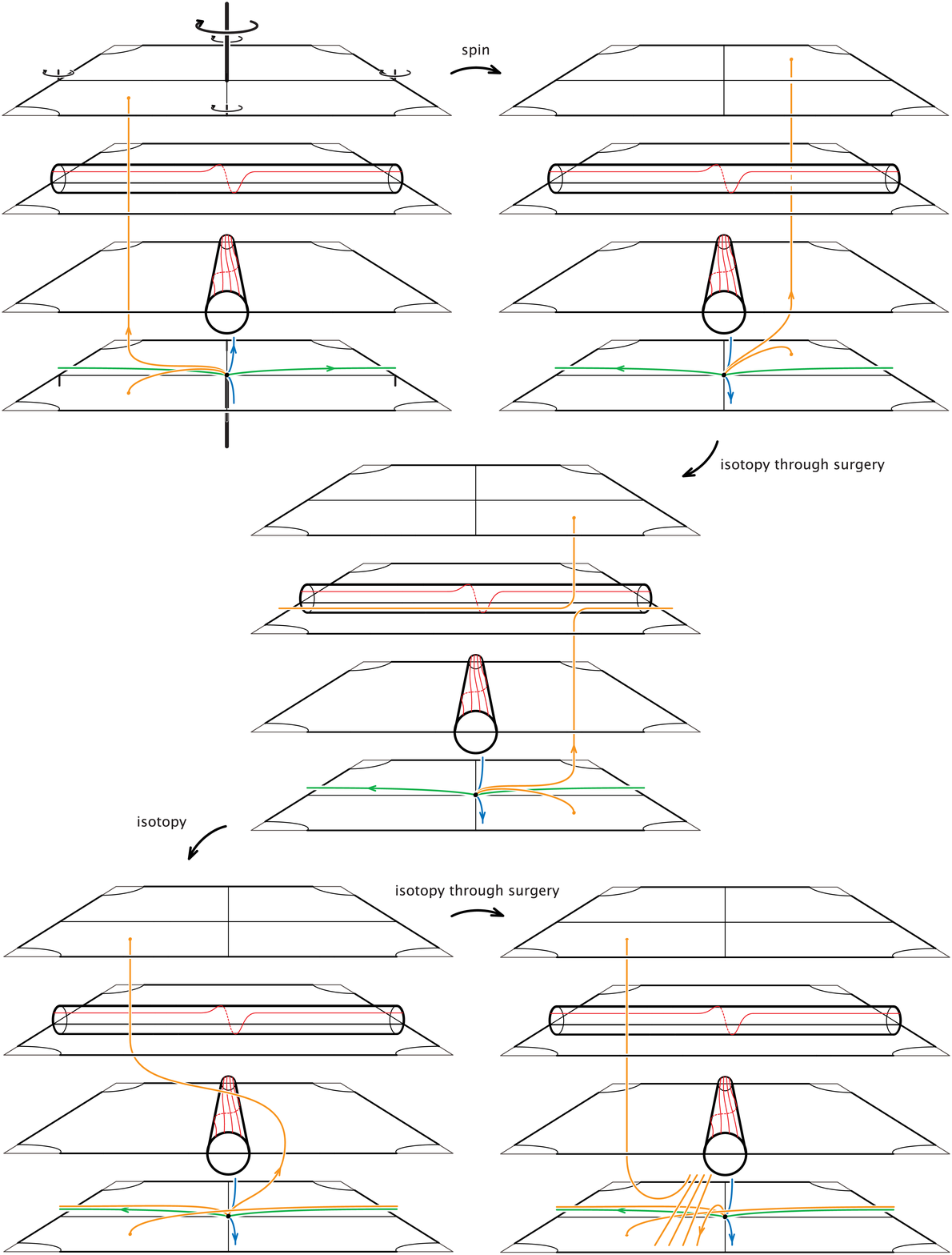}
\caption{The $\spin$ involution.}
\label{fig:spininvolution}
\end{figure}

\begin{figure}
\centering
\includegraphics[height=4.5in,keepaspectratio]{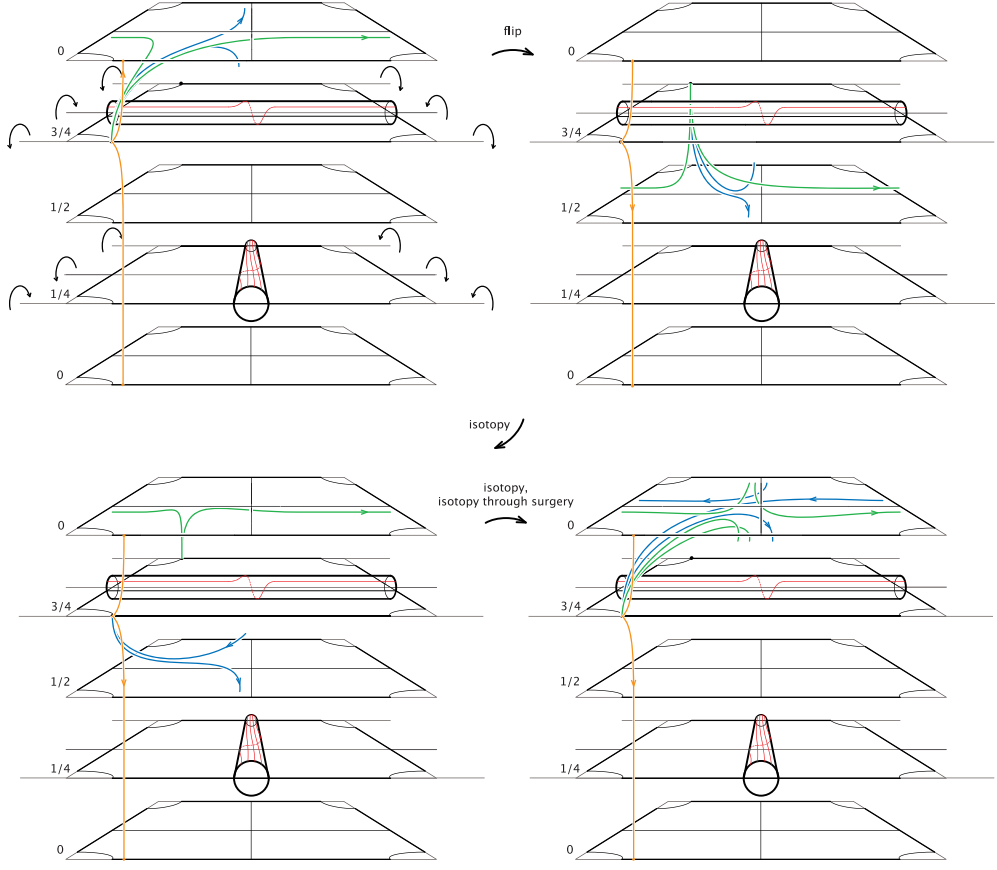}
\caption{The $\flip$ involution.}
\label{fig:flipinvolution}
\end{figure}

\subsection{Actions of  $\spin$ and $\flip$ on the character variety}

Lemma~\ref{lemma:spinflipgroup} demonstrates the action of $\spin$ and $\flip$ on the group $\Gamma_n=\pi_1(M_n)$.  In general, if $\varsigma$ acts on a group $\Gamma$ this induces an (possibly trivial) action on $\tilde{X}(\Gamma)$.  Let $\varsigma_*$ denote this action on $\tilde{X}(\Gamma)$.  This can be defined by 
\[ \varsigma_*( \chi_{\rho}(\gamma)) = \chi_{\rho}(\varsigma(\gamma))\]
for all $\gamma \in \Gamma$. (This can also be expressed as  $\chi_{\rho\circ \varsigma}(\gamma).$) If for all $\gamma \in \Gamma$,  $\varsigma(\gamma)$ is conjugate to $\gamma^{\pm}$ then $\rho(\varsigma(\gamma))$ is a matrix conjugate to $\rho(\gamma)^{\pm}$. It follows that the traces of these matrices are equal, and the action $\varsigma_*$ on $\tilde{X}(\Gamma)$ is trivial.  That is, if $\varsigma$ acts trivially on the unoriented free homotopy classes of loops in $M_n$ the action $\varsigma_*$ on $\tilde{X}(\Gamma_n)$ is trivial. We now explore the actions of $\spin$ and $\flip$ on $\tilde{X}(\Gamma)$.

Recall that $x=\tr(\rho(\alpha))$, $y=\tr(\rho(\beta))$ and $z=\tr(\rho(\alpha \beta))$.  If $\rho$ is a generic representation in the standard form of Definition~\ref{defn:genericexceptional} then, as in Remark~\ref{rem:genericxyz}, we have $x=a+a^{-1}$, $y=b+b^{-1}$, and $z=ab+a^{-1}b^{-1}+st$.  Also recall that by Proposition~\ref{prop:irreps}, the closure of the set of irreducible representations, $X(\Gamma)$, is naturally isomorphic to the  vanishing set of  $( \varphi_1, \varphi_2,\varphi_3)$.  Let $\flip_*$ and $\spin_*$  denote the induced actions of $\flip$ and $\spin$ on $\A^3(x,y,z)$, respectively.

\begin{prop}\label{prop:action}
The involution $\flip$ acts trivially on the unoriented free  homotopy classes of loops in $M_n$. 
It induces a trivial action on $\A^3(x,y,z)$.  

\smallskip
The involution $\spin$ acts non-trivially on the unoriented free  homotopy classes of loops in $M_n$.  It induces the action defined by 
\[ \spin_* \colon  
\begin{cases}
 x & \mapsto -z f_n(y)+x f_{n+1}(y) \\
 y &\mapsto   y\\
 z&  \mapsto  z.
\end{cases}\] 
This action is trivial on $\tilde{X}(\Gamma_n)$.
\end{prop}

\begin{proof}
First, we consider the $\flip$ symmetry.  Notice that $\flip_*(\alpha)$ is conjugate to $\bar{\alpha}$, $\flip_*(\beta)$ is conjugate to $\bar{\beta}$, and $\flip_*(\alpha \beta)$ is conjugate to $\bar{\alpha} \bar{\beta}$, and therefore conjugate to $\overline{ \alpha \beta}$.   Hence, $\flip_*$ acts trivially on  unoriented homotopy classes of loops, and the action on $\A^3(x,y,z)$ is trivial, as it fixes $x,y$, and $z$.

Now we consider the $\spin$ symmetry.  As $\spin_*(\beta)=\bar{\beta}$, we see that $\spin_*(y)=y$, and as $\spin_*(\alpha \beta)$ is conjugate to $\alpha \beta$, it follows that $\spin_*(z)=z$.   We now devote our attention to $\spin_*(\alpha)$, which is conjugate to $\bar{\beta}^n\alpha$.  The group relation implies that  $\bar{\beta}^n\alpha = \bar{\alpha}\beta \alpha \alpha \beta$. 

Let $\theta\colon \Gamma_n \rightarrow  \langle a, b : b^{n-2}, ab=ba \rangle$ be the abelianization map where $\theta(\alpha)=a$ and $\theta(\beta)=b$.  As $\theta(\alpha)=a$ and $\theta (\bar{\alpha}\beta \alpha \alpha \beta)=ab^2$ where $a^{\pm 1}\neq ab^2$ if $n\neq 2$, we conclude that $\alpha^{\pm}$ is not homotopic to $\spin_*(\alpha)$. 

First, we consider the action on $X(\Gamma_n)$.
We compute 
\[ \tr(\rho(\spin_*(\alpha))= \tr(\rho(\bar{\beta}^{n+1}\alpha \beta)) = \tr (\rho(\bar{\beta}^n\alpha)) = ab^{-n}+a^{-1}b^n -stf_n(y)\] using the equations from Section~\ref{section:charactervarietycalculations}.  Upon substitution, we see that
\[ \tr(\rho(\spin_*(\alpha))=-z f_n(y)+x f_{n+1}(y).\]
Therefore, the action induced by $\spin$ on $X(\Gamma_n)$ is 
\[ \spin_*(x) =-z f_n(y)+x f_{n+1}(y), \quad
\spin_*(y)=y, \quad \text{ and } \quad \spin_*(z)=z. \]
The set of points in $\A^3(x,y,z)$ fixed under this action is the set determined by
\[ x-\spin_*(x), y-\spin_*(y), \text{ and } z-\spin_*(z)\] 
which is the set determined by $x-(-z f_n(y)+x f_{n+1}(y))=-\varphi_3$.  As the character variety has coordinate ring  $\A^3[x,y,z]/( \varphi_1, \varphi_2,\varphi_3 )$ this action is non-trivial on $\A^3(x,y,z)$ but is trivial on $X(\Gamma_n)$.  

Using Proposition~\ref{prop:redreps} it is elementary to verify that the action is also trivial on $\tilde{X}_{red}(\Gamma_n)$.
\end{proof}

The fixed point set  of $\spin_*$ is the (complex) surface defined by the vanishing set of $\varphi_3$.  The fixed point set of $\flip_*$ is all of $\A^3$.   The whole of $\tilde{X}(\Gamma_n)$ is contained in the intersection of these two sets.  Recall that  the action of $\flip_*$ is trivial on all unoriented free homotopy classes of simple closed curves, but the action of $\spin_*$ is not.  If we form the orbifold quotient of $M_n$  by $\flip_*$ this property ensures that all characters of $\pi_1(M_n)$ are characters of this  quotient.  Since the action of $\spin_*$ is non-trivial on some free homotopy classes of simple closed curves, the same does not immediately follow.  For example in \cite{MR2827003} it was shown that the $7_4$ knot (which is a two-bridge knot with an order eight symmetry group) has a symmetry  for which there are irreducible representations which are not characters of the  orbifold quotient.  In fact the set of such representations is a $\C$ curve.  It is shown in \cite{MR2827003} that all so-called $J(2m,2m)$ knots (all are two-bridge knots) share this property.  That is, this symmetry has the effect of factoring the character variety.  (This factorization is also shown in \cite{MR1248091} for general two bridge knots by looking at ideal points.)

In contrast, the calculation above  shows that for the manifolds $M_n$, all characters of $\pi_1(M_n)$ extend to characters of  $\pi_1^{orb}(M_n/ \langle \spin_* \rangle)$. 
\begin{thm}
For every $n$, every irreducible representation $\rho\colon \pi_1(M_n) \to \SL_2(\C)$ is the restriction of an irreducible representation of the orbifold fundamental group of $O_n$, the quotient of $M$ by the spin and flip symmetries.

\end{thm}

%
%
\section{The Twisted Alexander Polynomial}\label{section:alexanderpoly}
%
%

Let $\mathcal{T}_{M_n}^\rho(T) \in \C[T,T^{-1}]$ denote the symmetrized Alexander polynomial of $M_n$ twisted by the irreducible representation $\rho \colon \pi_1(M_n) \to \SL_2(\C)$ (taken with respect to the epimorphism $\phi \colon \pi_1(M_n) \to S^1=\langle T \rangle$ dual to the fiber in which $\phi(\mu)=T$).  Since $M_n$ is a once-punctured torus bundle $\mathcal{T}_{M_n}^\rho(T)$ is a monic symmetric degree $2$ polynomial \cite{friedlvidussi,dfj-twisted}.  Hence $\mathcal{T}_{M_n}^\rho(T) \equiv T^1 + Z_n T^0 + T^{-1}$ for some number $Z_n= \mathcal{T}_{M_n}^\rho(i) \in \C$.   
Here we calculate $\mathcal{T}_{M_n}^\rho(T)$ and, in particular, this number $Z_n$.

 \begin{thm}\label{thm:twistedalex}
The twisted Alexander polynomial of $M_n$, twisted by a representation $\rho$ corresponding to the point $(x,y,z)$ on the character variety is
\[\mathcal{T}_{M_n}^\rho(T)= T^{-1}+\frac{2(z-x)}{y-2}  + T\]
and $\displaystyle Z_n = \frac{2(z-x)}{y-2}$.
\end{thm}

\begin{proof}
Recall our presentation $\pi_1(M) = \langle \alpha, \beta : \bar{\beta}^n = \omega \rangle$ 
where $\omega = \bar{\alpha} \beta \alpha \alpha \beta \bar{\alpha}$.
Since $\beta$ is homotopic to a loop on the fiber while $\alpha = \bar{\beta} \bar{\mu}$  is homotopic to a loop transversally intersecting the fiber once, $\phi(\alpha) = T^{-1}$ while $\phi(\beta) = T^0$.   
Therefore the twisted Alexander polynomial may be computed as
\[\mathcal{T}_{M_n}^\rho(T) \equiv \frac{\det ((\rho \otimes \phi)(\bdry_\beta(\omega-\bar{\beta}^n)))}{\det((\rho \otimes \phi)(\alpha-1))}.\]

Using the standard form of the irreducible representation $\rho \colon \pi_1(M_n) \to \SL_2(\C)$ (in which $s=t$),
\[\alpha \mapsto A = \mat{a}{0}{s}{a^{-1}} \mbox{ and } \beta \mapsto B = \mat{b}{s}{0}{b^{-1}},\ \]
we obtain

\begin{align*}
\mathcal{T}_{M_n}^\rho(T) 
&= \frac{\det ((\rho \otimes \phi)(\bdry_\beta(\omega-\bar{\beta}^n)))}{\det((\rho \otimes \phi)(\alpha-1))}  \\
&= \frac{\det (A^{-1}T+A^{-1}BA^2 T^{-1} - (I- A^{-1}BA^2BA^{-1})(I-B)^{-1})}{\det(A T^{-1} - I)}\\
&=1+\left(\frac{xz^2-x^2yz+x^3-yz-4x+xy^2+2z}{y-2}\right)T+T^2 \\
&\equiv T^{-1}+\frac{2(z-x)}{y-2}  + T. 
\end{align*}

In this calculation, Fox Calculus gives $\bdry_\beta(\omega-\bar{\beta}^n) = \bar{\alpha}+\bar{\alpha}\beta\alpha^2 - \bdry_\beta(\bar{\beta}^n)$ where $\bdry_\beta(\bar{\beta}^n) =\bdry_\beta(\beta^{-n}) =  
\frac{\beta^n-1}{\beta-1}$.
Then, since $I-B$ is non-singular ($b \neq 1$), we may write $(I+B+ \dots + B^{n-1}) = (I-B^n)(I-B)^{-1}$.  Thus the relation $\rho(\bar{\beta}^n) = \rho(\omega)$ gives
\[\rho(\bdry_\beta(\bar{\beta}^n))=(I-B^{-n})(I-B)^{-1} = (I- A^{-1}BA^2BA^{-1})(I-B)^{-1}.\]
The quotient of the determinants is simplified by the substitution $x=a+a^{-1}$, $y=b+b^{-1}$, and $z=ab+a^{-1}b^{-1} + s^2$ as usual, followed by an application of the relation $F_{21} =0$.  Multiplication by the unit $T^{-1}$ brings the polynomial into the symmetric form claimed.
\end{proof}

\begin{remark}
Using Proposition~\ref{prop:xyzcharvar} we express $Z_n =\frac{2(z-x)}{y-2}$ for points $(x,y,z)$ on the character variety in terms of a single coordinate. If $0 \neq 1-f_{n-1}(y) = -h_n(y)\ell_n(y)$ so that $x \neq 0$, then 
\[Z_n =-2\epsilon \frac{f_n(y) - f_{n-1}(y)-y+1}{(2-y)\sqrt{1-f_{n-1}(y)}}= -2\epsilon \sqrt{-\frac{h_n(y)}{\ell_n(y)}} \frac{ k_n(y)+\ell_n(y)}{y-2}. \]
Otherwise either $x=z=0$ so that $Z_n = 0$ or  $n \equiv 2 \pmod 4$ and $(x,y,z)=(0,0,z)$ is on the extra line so that $Z_n = -z$.
\end{remark}

\begin{cor}\label{cor:dfZ}
If $|n|>2$ and $(x,y,z)$ is a point in $X(\Gamma_n)$,  corresponding to a discrete faithful representation $\rho_0$, then
\[Z_n = \epsilon \frac{4-y+f_n(y)}{y-2}= \epsilon \frac{4-y\pm\sqrt{y^2-8}}{y-2}\]
 where $\epsilon = \pm1$ and $y$ is a root of the polynomial $\hat{p}_n(y)$ of Definition~\ref{defn:phat} (i.e.\ a root of $p_n(y)=f_{n+1}(y)-f_{n-1}(y)-y^2+6$ other than $-2$).
\end{cor}

\begin{proof}
By Propositions~\ref{prop:disc2} and \ref{prop:discfaithfulyeqn}, $y$ is a root of the polynomial $\hat{p}_n(y)$, $x=\epsilon \tfrac{1}{2} (y-f_n(y))=\epsilon (y\mp \sqrt{y^2-8})/2$, and $z = 2\epsilon$.
\end{proof}

\begin{prop}
For each $n\to\infty$ and $n\to-\infty$, there is a sequence of discrete faithful representations of $\Gamma_n$ so that  both of the following occur:   $Z_n \to -\epsilon (\tfrac{3}{2}\pm\tfrac{1}{2}i)$ and $|Z_n| \to \infty$.
\end{prop}

\begin{proof}
Since for $|n|$ large, $M_n$ is the result of hyperbolic Dehn filling on one boundary component of the Whitehead link exterior, there is a sequence $\rho_n$ of discrete faithful $\SL_2(\C)$ representations of $\Gamma_n$ that converge to a discrete faithful $\SL_2(\C)$ representation $\rho_{\pm\infty}$ of the fundamental group of the Whitehead link exterior as $n \to \pm \infty$.  Because $\beta \in \Gamma_n = \pi_1(M_n)$ is represented by the core of the filling of the Whitehead link producing $M_n$ (which is a curve in a fiber along which we perform Dehn twists to generate our family of manifolds), it follows that $\beta \in \Gamma_n$ converges to a peripheral element in the fundamental group of the Whitehead link exterior.  Therefore $|y| = |\tr(\rho_n(\beta))| \to 2$ as $n \to \pm\infty$.  By Corollary~\ref{cor:dfZ}, if $y \to -2$ then $Z_n \to -\epsilon (\tfrac{3}{2}\pm\tfrac{1}{2}i)$.  If $y\to +2$ then $|Z_n|\to \infty$.   By the discussion in the introduction to Section~\ref{section:tracefield}, for all $n$ there is an $\epsilon \in H^1(M_n,\Z/2\Z)$ such that $\epsilon \circ \chi_{\rho} (\beta)=-\chi_{\rho}(\beta)$.  That is, $\epsilon(y)=-y$.  Therefore we have discrete faithful representations with $y\rightarrow 2$ and with $y\rightarrow -2$.
\end{proof}

\begin{remark}
When $n=-3$, $M_{-3}$ is the figure eight knot exterior. If $(x,y,z)$ corresponds to a discrete, faithful representation then we may calculate $Z_{-3} = \pm4$.

When $n=3$, $M_3$ is the sister of the figure eight knot exterior.  If $(x,y,z)$ corresponds to a discrete, faithful representation, then we may calculate $Z_{3} =  \pm2\sqrt3 i$.
\end{remark}

\begin{remark}
For discrete faithful representations $\rho$ where $|n|>2$,  computer calculations suggest that  $Z_{n}$ and hence $\mathcal{T}_{M_n}^\rho(T)$ are real only when $n=-3$.  
\end{remark}

%
%
\section{Dilatation}\label{section:dilatation}
%
%

A homeomorphism $\psi \colon F \to F$ of a compact surface $F$ is pseudo-Anosov  if no power of it fixes the homotopy class of an essential simple closed curve.  
Within a mapping class of pseudo-Anosov homeomorphisms we may choose $\psi$ to be one so that there is a transverse pair of measured singular foliations $\F_s$ and $\F_u$ on $F$ where $\psi(\F_u) = \lambda \F_u$ and $\psi(\F_s) = \lambda^{-1} \F_s$ for a number $\lambda >1$.  This number $\lambda$ is the {\em dilatation} of $\psi$, \cite{flp}.   It is an invariant of the conjugacy class of $\psi$.  As such it is also an invariant of the mapping torus of $\psi$.   In this sense the dilatation of $\phi_n$ gives a measure of complexity of $M_n$ and $\Gamma_n$.

\begin{thm}\label{thm:dilatation}
 For $|n|>2$ the dilatation of $\phi_n$ is $\frac{1}{2} ( |n|+\sqrt{n^2-4})$.
\end{thm}

\begin{proof}
Since $\phi_n = \tau_c \tau_b^{n+2} \colon T \to T$ for $|n|>2$ is a homeomorphism of the once-punctured torus $T$ generated by Dehn twists along curves $c$ and $b$ which intersect once, its dilatation may be computed as the largest eigenvalue of $\pm[\phi_n] \in \PSL_2(\Z)$.  The matrix $[\phi_n]$ is given in the proof of Lemma~\ref{lem:geometry}.
\end{proof}

We observe that the dilatation of $\phi_n$ is approximately the genus of the canonical component of the $\SL_2(\C)$ character variety of $\Gamma_n$. 

\begin{thm}\label{thm:dilatationgenus}
 Assume $|n|>2$.  Let $g$ denote the genus of $X_0(\Gamma_n)$ and $d$ denote the floor of the dilatation of $\phi_n$.  Then $d=2g+\alpha$ where 
\[ \alpha = \left\{ \begin{aligned}  \quad 
4+\text{sgn}(n) \quad & \text{ if } n \equiv 2\pmod 4\\ 
2+\text{sgn}(n) \quad  & \text{ if } n \equiv 0 \pmod 4 \\ 
1+\text{sgn}(n) \quad & \text{ if } n\equiv 1,3 \pmod 4. 
\end{aligned} \right. \] 
\end{thm}

\begin{proof} With $d$ being the floor of the dilatation of $\phi$, we obtain from Theorem~\ref{thm:dilatation} below that $d=|n|-1$ when $|n|>2$.
The formula now follows from  Theorem~\ref{thm:mainsummary1}. 
\end{proof}

\bibliographystyle{plain}
\bibliography{TNOOPTbib.bib}

\end{document}